\documentclass[11pt]{article}
\usepackage[english]{babel}

\usepackage{latexsym,graphicx,caption,subcaption}
\usepackage{amsmath}
\usepackage{amsfonts}
\usepackage{mathrsfs}
\usepackage{amstext}
\usepackage{amssymb}
\usepackage{amsthm} 
\usepackage{xcolor,times,authblk}

\usepackage{hyperref}

\usepackage{tabu}
\usepackage{multirow}
\usepackage{array}

\usepackage{xspace}

\newtheorem{thm}{Theorem}[section]
\newtheorem{prop}[thm]{Proposition}

\newtheorem{defn}[thm]{Definition}
\newtheorem{rmk}[thm]{Remark}
\newtheorem{lem}[thm]{Lemma}
\renewenvironment{proof}{\begin{trivlist}
      \item[]\hspace{0cm}{\bf Proof.}
      \hspace{0cm}}{\hfill $\blacksquare$
      \end{trivlist}}
\newenvironment{proofof}[1]{\begin{trivlist}
      \item[]\hspace{0cm}{\bf Proof of #1.}
      \hspace{0cm}}{\hfill $\blacksquare$
      \end{trivlist}}

\makeatletter

\@addtoreset{equation}{section}
\makeatother

\usepackage[margin=3cm]{geometry}
\def\bn{{\boldsymbol\nsf}}

\def\R{\mathbb R}

\def\N{\mathbb N}

\def\Rn{\R^n}
\def\Rnp{\R^n_+}

\def\Rnm{\R^n_-}
\def\Bn#1{\cB(0,#1)}

\def\Bnp#1{\cB^+(0,#1)}

\def\Bnm#1{\cB^-(0,#1)}

\def\cB{\mathcal B}
\def\cO{\mathcal O}
\def\cU{\mathcal U}
\def\cV{{\cal V}}

\def\fE{{\mathfrak E}}
\def\fF{{\mathfrak F}}
\def\fG{{\mathfrak G}}

\def\fL{{\mathfrak L}}
\def\fM{{\mathfrak M}}
\def\fN{{\mathfrak N}}

\def\fU{{\mathfrak U}}
\def\fV{{\mathfrak V}}

\def\sB{{\mathscr B}}
\def\sC{{\mathscr C}}

\def\sL{{\mathscr L}}

\def\sR{{\mathscr R}}

\def\sV{{\mathscr V}}

\def\Ca#1{\sC^{#1,\alpha}}

\newcommand{\px}{\mathsf{x}}
\newcommand{\pt}{\mathsf{X}}
\newcommand{\ps}{\mathsf{Y}}
\newcommand{\py}{\mathsf{y}}
\def\msigma{\mathsf{\varsigma}}
\def\p{\mathsf p}

\newcommand{\nsf}{\mathsf{n}}

\def\se{\mathsf e}
\def\ssei{{\mathsf s}_{\sharp}}
\def\sse{{\mathsf s}_{e}}
\def\ssi{{\mathsf s}_{i}}

\def\G{{G}}

\def\eps{\varepsilon}
\def\e#1{\eps_{#1}}
\def\ea{{\eps^{\rm ad}}}
\def\ear#1{{\eps^{#1}}}
\def\be{{\boldsymbol\varepsilon}}
\def\bea{{\be^{\rm ad}}}

\def\Ie#1{{]0,\eps^{#1}[}} 
\def\Jea{{]-\ea,\ea[}} 
\def\Je#1{{]-\eps^{#1},\eps^{#1}[}} 
\def\ec2{\e2^{\star}}
\def\b0{{\boldsymbol{0}}}
\def\bIea{{]{\b0},\bea[}} 
\def\bIep{{]{\b0},\be'[}} 
\def\bIes{{]{\b0},\be''[}} 
\def\bIe#1{{]{\b0},\be^{#1}[}} 
\def\bJea{{]-\bea,\bea[}} 
\def\bJep{{]-\be',\be'[}} 
\def\bJes{{]-\be'',\be''[}} 
\def\bJe#1{{]-\be^{#1},\be^{#1}[}} 
\def\bm{{\boldsymbol\mu}}
\def\bde{{\boldsymbol \delta}}
\def\bfi{{\boldsymbol \phi}}

\def\domgen{\mathcal D}
\def\domeps{\Omega_{\be}}
\def\dome{\Omega_{\eps}}
\def\dom{\Omega}
\def\incl{\omega}
\def\incleps{\omega_{\be}}
\def\incle{\omega_{\eps}}
\def\dO{\partial\dom}
\def\dincl{\partial\incl}
\def\dOeps{\partial\domeps}
\def\d0O{\partial_{0}\dom}
\def\dpO{\partial_{+}\dom}

\def\go{g^{\mathrm o}}
\def\gi{g^{\mathrm i}}
\def\ugi{\underline{g}^{\mathrm i}}

\def\uzero{u_0}
\def\ueps{u_\be}
\def\ue{u_\eps}
\def\uue{\underline{u}_\be}

\newcommand{\Be}{\color{blue}}

\title{A Dirichlet problem for the Laplace operator in a domain with a small hole close to the boundary}
\date{}
\author[1]{Virginie Bonnaillie-No\"el\thanks{bonnaillie@math.cnrs.fr}}
\author[2,3]{Matteo Dalla Riva\thanks{matteo.dallariva@gmail.com}}
\author[4]{Marc Dambrine\thanks{Marc.Dambrine@univ-pau.fr}}
\author[3]{Paolo Musolino\thanks{musolinopaolo@gmail.com}}
\affil[1]{D\'epartement de math\'ematiques et applications, \'Ecole normale sup\'erieure, CNRS, PSL Research University, 45 rue d'Ulm, 75005 Paris, France}
\affil[2]{Department of Mathematics, The University of Tulsa, 800 South Tucker Drive, Tulsa, Oklahoma 74104, USA}
\affil[3]{Department of Mathematics, Aberystwyth University, Ceredigion SY23 3BZ, Wales, UK}
\affil[4]{CNRS, Univ Pau \& Pays Adour,  Laboratoire de math\'ematiques et de leur applications de Pau - F\'ed\'eration IPRA, UMR 5142, Pau, France}

\begin{document}
\maketitle

\begin{abstract}
We study the Dirichlet problem in a domain with a small hole close to the boundary.
 To do so, for each pair $\be = (\e1 ,\e2 )$ of positive parameters, we consider a perforated
 domain $\domeps$ obtained by making a small hole of size $\e1 \e2 $ in an open regular
 subset $\dom$ of $\mathbb{R}^n$ at distance $\e1$ from the boundary $\partial\Omega$.
 As $\e1 \to 0$, the perforation shrinks to a point and, at the same time, approaches the boundary.
 When $\be \to (0,0)$, the size of the hole shrinks at a faster rate than its approach to the boundary.
We denote by $\ueps$ the solution of a Dirichlet problem for the Laplace equation in $\domeps$.
For a space dimension $n\geq 3$, we show that the function mapping $\be$
to $\ueps$ has a real analytic continuation in a neighborhood  of $(0,0)$. By contrast,
for $n=2$ we consider two different regimes: $\be$ tends to $(0,0)$, and $\e1$ tends
to $0$ with $\e2$ fixed.
When $\be\to(0,0)$, the solution $\ueps$ has a logarithmic behavior;
when only $\e1\to0$ and $\e2$ is fixed, the asymptotic behavior of  the solution can be
described in terms of real analytic functions of $\e1$. We also show that for $n=2$,
the energy integral and the total flux on the exterior boundary have different limiting  values
in the two regimes. We prove these results by using functional analysis methods
in conjunction with certain special layer potentials.
\end{abstract}

\par
\vspace{9pt}


\section{Introduction}\label{introd}

Elliptic boundary value problems in domains where a small part has been removed arise in the study
of mathematical models for bodies with small perforations or inclusions, and are of interest not only
for their mathematical aspects, but also for their applications to elasticity, heat conduction,
fluid mechanics, and so on.  They play a central role in the treatment of inverse problems
(see, {\it e.g.}, Ammari and Kang \cite{AmKa07}) and in the computation of the so-called
`topological derivative', which is a fundamental tool in shape and topological optimization (see, {\it e.g.},
Novotny and Soko\l owsky \cite{NoSo13}).  Owing to the difference in size between the small removed part
and the whole domain, the application of standard numerical methods requires the
use of highly nonhomogeneous meshes that often lead to inaccuracy and instability.
This difficulty can be overcome and the validity of the chosen numerical strategies
can be guaranteed only if adequate theoretical studies are first conducted on the problem.

In this paper, we consider the case of the Dirichlet problem for the Laplace equation
in a domain  with a small hole  `moderately close' to the boundary, {\it i.e.}, a hole that
approaches the outer boundary of the domain at a certain rate, while shrinking
to a point at a faster rate.
In two-dimensional space, we also consider the case where the size of the hole and its distance
from the boundary are comparable. It turns out that the two types of asymptotic behavior in this setup
are different: the first case gives rise to logarithmic behavior, whereas the second one
generates a real analytic continuation result. Additionally, the energy integral and the total flux of the solution
on the outer boundary may have different limiting values.

We begin by describing the geometric setting of our problem. We take $n \in \mathbb{N}\setminus \{0,1\}$
and, without loss of generality, we place the problem in the upper half space, which we denote by $\Rnp$.
More precisely, we define
\[
\Rnp \equiv\{\px=(x_1,\dots,x_n)\in\mathbb{R}^n\;:\; x_n>0\}\, .
\]
We note that the boundary $\partial\Rnp$ coincides with the hyperplane $x_n=0$. 
Then we fix a domain $\dom$ such that
\begin{equation}\tag{$H_{1}$} \label{Omega}
\text{$\dom$ is an open bounded connected subset of $\Rnp$ of class $\Ca{1}$,}
\end{equation}
where $\alpha\in]0,1[$ is a regularity parameter. The definition of functions and sets of the usual Schauder classes
$\Ca{k}$ ($k=0,1$) can be found, for example, in  Gilbarg and Trudinger~\cite[\S6.2]{GiTr83}.
We denote by $\dO$ the boundary of $\dom$. In this paper, we assume that a part
of $\dO$ is flat and that the hole is approaching it (see Figure~\ref{fig.geom}). This is described by setting
$$\d0O\equiv\dO\cap\partial\mathbb{R}^n_{+},\qquad \dpO\equiv\dO\cap\mathbb{R}^n_{+},$$
and assuming that
\begin{align}\tag{$H_{2}$}
\label{DOmega}
&\text{$\d0O$ is an open neighborhood of $0$ in $\partial\mathbb{R}^n_{+}$.}
\end{align}
The set $\dom$ plays the role of the `unperturbed' domain. To define the hole, we consider another set $\incl$ satisfying the following assumption:
\[
\text{$\incl$ is a bounded open connected subset of $\Rn$ of class $\Ca{1}$ such that $0\in \incl$.}
\]
The set $\incl$ represents the shape of the perforation. Then we fix a point
\begin{equation}\label{p}
\p=(p_1,\dots,p_n)\in\Rnp,
\end{equation}
and define the inclusion $\incleps $ by
\[
\incleps \equiv\e1 \p+\e1 \e2 \incl\,,\qquad\forall \be\equiv(\e1 ,\e2) \in\R^2\,.
\]
We adopt the following notation. If  $\be'\equiv(\e1',\e2'),\be''\equiv(\e1'',\e2'')\in\R^2$, then we write $\be'\leq\be''$ (respectively, $\be'<\be''$)
if and only if $\e{j}'\leq\e{j}''$ (respectively, $\e{j}'<\e{j}''$), for $j=1,2$, and denote by $]\be',\be''[$
the open rectangular domain of $\be\in\R^2$ such that $\be'<\be<\be''$. We also set $\b0\equiv(0,0)$.
Then it is easy to verify that there is $\bea\in]0,+\infty[^2$ such that
\[  \overline{\incleps }\subseteq\dom, \qquad \forall \be\in\bIea. \]
In addition, since we are interested in the case where the vector
$(\varepsilon_{1},\varepsilon_{1}\varepsilon_{2})$ is close to $\b0$, we may assume without loss of generality that
\[ \text{$\e{1}^{\rm ad}<1$ and $1<\e{2}^{\rm ad}<1/\e{1}^{\rm ad}$.} \]
Hence, $\e{1}\e{2}<1$ for all $\be\in\bIea$.
This technical condition allows us to deal with the function $1/\log(\e{1}\e{2})$
as in Section \ref{n=2}, and to consider the case where $\e{2}=1$ in Section \ref{e2=1}.

In a certain sense, $\bIea$ is a set of admissible parameters for which we can define the perforated domain
$\domeps$ obtained by removing from the unperturbed domain $\dom$ the closure $\overline\incleps$
of $\incleps $,  {\it i.e.},
\[ \domeps \equiv\dom\setminus\overline{\incleps},\qquad \forall \be\in \bIea. \]
We remark that, for all $\be\in\bIea$, $\domeps $ is a bounded connected open domain of class $\Ca{1}$
with boundary $\dOeps$ consisting of two connected components: $\dO$ and
$\partial \incleps =\e1 \p+\e1 \e2 \dincl$.  The distance of the hole $\incleps $ from the boundary $\dO $
is controlled by $\e1$, while its size  is controlled by the product $\e1 \e2 $.
Clearly, as the pair $\be\in\bIea$ approaches the singular value $(0,\e2^*)$, both the size of the cavity
and its distance from the boundary $\dO$ tend to $0$. If $\e2^*=0$, then the ratio of the size of the hole
to its distance from the boundary tends to $0$, and we can say that the size tends to zero `faster'
than the distance. If, instead, $\e2^*>0$, then the size of the hole and its distance from the boundary
tend to zero at the same rate.
Figure~\ref{fig.geom} illustrates our geometric setting.

\begin{figure}[h!t]
\begin{center}
\includegraphics[scale=1]{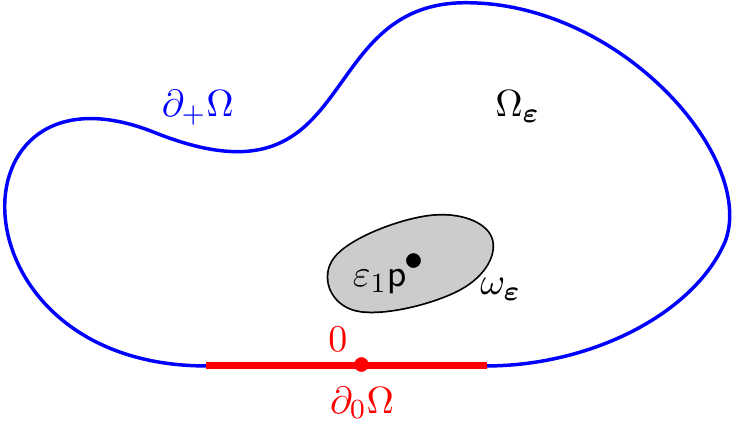}
\caption{Geometrical setting.\label{fig.geom}}
\end{center}
\end{figure}

On the $\be$-dependent domain $\domeps $, for $\be\in\bIea$ fixed we now consider the Dirichlet problem
\begin{equation}
\label{bvpe}
\begin{cases}
\Delta u(\px)=0, & \forall \px \in \domeps \,,\\
u(\px)=\go(\px), & \forall \px \in \dO\, ,\\
u(\px)=\gi\Big(\frac{\px-\e1 \p}{\e1 \e2 }\Big), & \forall \px \in \partial \incleps \, ,
\end{cases}
\end{equation}
where $\go \in \Ca{1}(\dO)$ and $\gi \in \Ca{1}(\dincl)$ are prescribed functions.
As is well known, \eqref{bvpe} has a unique solution in $\Ca{1}(\overline{\domeps })$.
To emphasize the dependence of this solution on $\be$, we denote it by $\ueps$. The aim of this paper is
to investigate the behavior of $\ueps $ when the parameter $\be=(\e1,\e2)$ approaches
the singular value $\b0\equiv(0,0)$. In two-dimensional space we also consider the case where
$\e1\to 0$ with $\e2>0$ fixed, and show that this leads to a specific asymptotic behavior.  Namely, in such regime, there are no logarithmic terms appearing in the asymptotic behavior of the solution, in contrast with what happens when $\be=(\e1,\e2) \to (0,0)$ in dimension two  (cf. Subsections \ref{beto0sec} and \ref{fixed} below). In higher dimension, instead, the case where
$\e1\to 0$ with $\e2>0$ fixed does not present specific differences and the tools developed for analyzing the situation when $\be=(\e1,\e2) \to (0,0)$ can be exploited by keeping $\e2>0$ ``frozen''. Then the corresponding results on the macroscopic and microscopic behavior would follow. For this reason, we confine here to analyze  the case where
$\e1\to 0$ with $\e2>0$ fixed only in dimension two.

We remark that every point $\px\in\dom$ stays
in $\domeps$ for $\e1$ sufficiently close to $0$. Accordingly,  if we fix a point $\px\in\dom$,
then $ \ueps (\px)$ is well defined for $\e1$ sufficiently small and we may ask
the following question:
\begin{equation}\label{eq:question}
\text{What can be said about the map $\be \mapsto \ueps (\px)$ for $\be>\b0$ close to ${\bf 0}$?}
\end{equation}
We mention that here we do not consider the case where $\e2$ is close to $0$ and $\e1$
remains positive. This case corresponds to a boundary value problem in a domain with a hole
that collapses to a point  in its interior, and has already been studied in the literature.

\subsection{Explicit computation on a toy problem}\label{toy}

To explain our results, we first consider
a two-dimensional test problem that has an explicit solution. We denote by $\cB(\px,\rho)$
the ball centered at $\px$ and of radius $\rho$, take a function $\gi \in \Ca{1}( \partial \cB(0,1))$,
and, for $\be\in]\b0,(1,1)[$, consider the following Dirichlet problem in the perforated
half space $\R^{2}_+ \setminus \cB((0,\e1),\e1\e2)$:
\begin{equation}
\label{bvpe:modele}
\begin{cases}
\Delta \ueps(\px)=0, & \forall \px \in \R^{2}_+ \setminus \cB((0,\e1),\e1\e2)\\
\ueps(\px)=0, & \forall \px \in \partial \R^{2}_+,\\
\ueps(\px)=\gi\Big(\frac{\px-\e1 \p}{\e1 \e2 }\Big), & \forall \px \in \partial \cB((0,\e1),\e1\e2),\\
\lim_{\px\to\infty}\ueps(\px)=0\,,
\end{cases}
\end{equation}
where $\p=(0,1)$. We also consider the conformal map
$$\varphi_a: ~ z\mapsto \cfrac{z-ia}{z+ia}\,,$$
with inverse
$$\varphi^{-1}_a:~z\mapsto - ia\ \cfrac{z+1}{z-1}\,.$$
When $a \ne 0$ is real, $\varphi_a$ maps the real axis onto the unit
circle.  Moreover, if
$$a(\be)=a(\e1,\e2)= \e1\sqrt{1-\e2^2}\,,$$
then $\varphi_{a(\be)}$ maps the circle centered at $(0,\e1)$ and of radius $\e1\e2$ to the circle
centered at the origin and of radius
$$\rho(\e2)= \sqrt{ \cfrac {1-\sqrt{1-\e2^{2}}} {1+\sqrt{1+\e2^{2}}}}\, .$$
We note that the maps $a:\ (\e1,\e2)\mapsto a(\be)$ and $\rho:\ \e2\mapsto\rho(\e2)$
are analytic. We mention that a similar computation
is performed in Ben Hassen and Bonnetier \cite{BenHassenBonnetier} for the case of two balls removed from an infinite medium.

Since harmonic functions are transformed  into harmonic functions by a conformal map, we can now
transfer problem \eqref{bvpe:modele} onto the annular domain  $\Bn{1} \setminus \Bn{\rho(\e{2})}$
by means of the map $\varphi_{a(\be)}$
and see that the unknown function $\uue=\ueps\circ\varphi^{-1}_{a(\be)}$ satisfies
\[
\begin{cases}
\Delta \uue=0, & \mbox{ in } \Bn{1} \setminus \Bn{\rho(\e{2})},\\
\uue=0, & \mbox{ on }\partial\Bn{1},\\
\uue(z)=\ugi_{\,\be}(\arg{z}), & \mbox{ for all  }z\in\partial\Bn{\rho(\e{2})},
\end{cases}
\]
and the new boundary condition
$$\ugi_{\,\be}(\theta)=\gi\left(-\cfrac{i}{\e2}\ \left(\sqrt{1-\e2^{2}} \ \cfrac{\rho(\e{2}) e^{i\theta}+1}
{\rho(\e{2}) e^{i\theta}-1} \ +1\right)\right),\qquad \forall \theta\in [0,2\pi[ .$$
To obtain the analytic expression of the solution, we expand $\ugi_{\,\be}$ in the Fourier series
$$\ugi_{\,\be}(\theta)=a_{0}(\ugi_{\,\be})+\sum_{k\geq 1}a_{k}(\ugi_{\,\be})
\cos{k\theta}+b_{k}(\ugi_{\,\be})\sin{k\theta},$$
so that, in polar coordinates,
$$\uue(r,\theta)= a_{0} (\ugi_{\,\be}){\frac {\log r}{\log \rho(\e{2})}} + \sum_{k\geq 1}
\left( a_{k}(\ugi_{\,\be}) \cos{k\theta}+b_{k}(\ugi_{\,\be})\sin{k\theta} \right)
\frac{r^{k}-r^{-k}}{\rho(\e{2})^{k}-\rho(\e{2})^{-k}}.$$
We can then recover $\ueps$ by computing $\ueps=\uue\circ \varphi_{a(\be)}$.  To this end,
we remark that in polar coordinates  we have $\varphi_{a(\be)}(\px)=r_\be(\px)e^{i\theta_\be(\px)}$, with
$$
r_\be(\px)=\left| \cfrac{x_{1}+ix_{2}-ia(\be)} {x_{1}+ix_{2}+ia(\be)} \right|, \qquad
\theta_\be(\px)=\arg\left(\cfrac{x_{1}+ix_{2}-ia(\be)} {x_{1}+ix_{2}+ia(\be)}\right).$$
As an example, if we assume that $\gi=1$, then the  solution of \eqref{bvpe:modele} is
\begin{equation}\label{utoy}
\ueps(\px)=\frac{\log r_\be(\px)}{\log \rho(\e{2})} \ =  \ \cfrac
{\log{ \left(x_1^2+\left(x_{2}-\e1\sqrt{1-\e2^2} \right)^{2} \right)} \ -\   \log{\left(x_1^2+\left(x_{2}
+\e1\sqrt{1-\e2^2} \right)^{2}\right)}}
{\log{ \left(1-\sqrt{1-\e2^{2}} \right)} - \log{\left(1+\sqrt{1+\e2^{2}}\right)} }.
\end{equation}
We note that
{for any  fixed $\px \in \R^{2}_+$ and $\e1$, $\e2$ positive and sufficiently small,
 the map $\be\mapsto \ueps(\px)$ is analytic. When $\be\to \b0$, the function $ \ueps$ tends to $0$
 with a main term of order $\e1|\log\e2|^{-1}$. In addition, for $\e2>0$ fixed,
 the map $\e1\mapsto \ueps(\px)$ has an analytic continuation around $\e1=0$.}
 
In what follows, we intend to prove similar results also for problem \eqref{bvpe}, and thus answer
the question \eqref{eq:question} by investigating the analyticity properties of the function
$\be\mapsto \ueps(\px)$. Furthermore, instead of evaluating $\ueps$ at a point $\px$,
we consider its restriction to suitable subsets of $\dom$ and the restriction of the rescaled function
$\pt\mapsto \ueps(\e{1}\p+\e{1}\e{2}\pt)$ to suitable open subsets of $\R^2\setminus\overline{\incl}$.
This permits us to study functionals related to  $\ueps$, such as the energy integral and the total flux
on $\partial\dom$. Our main results are described in Subsection \ref{mr}, in the next subsection instead we present our strategy.

\subsection{Methodology:  the functional analytic approach}

In the literature, most of the papers dedicated to the analysis of problems with small holes employ expansion
methods to provide asymptotic approximations of the solution.  As an example,  we mention  the  method
of matching asymptotic expansions proposed by Il'in (see, \textit{e.g.}, \cite{Il78, Il92,  Il99}),
the compound asymptotic expansion method of Maz'ya, Nazarov, and Plamenevskij \cite{MaNaPl00}
and of Kozlov, Maz'ya, and Movchan \cite{KoMaMo99}, and the mesoscale asymptotic approximations
presented by Maz'ya, Movchan, and Nieves \cite{MaMoNi13, MaMoNi16}. We also mention the works of Bonnaillie-No\"el, Lacave, and Masmoudi
\cite{BoLaMa}, Chesnel and Claeys \cite{ChCl14}, and Dauge, Tordeux, and Vial \cite{DaToVi10}.
Boundary value problems in domains with moderately close small holes have been
analyzed by means of multiple scale asymptotic expansions by Bonnaillie-No\"el, Dambrine, Tordeux,
and Vial \cite{BoDaToVi07, BoDaToVi09}, Bonnaillie-No\"el and Dambrine \cite{BoDa13},
and Bonnaillie-No\"el, Dambrine, and Lacave \cite{BoDaLa}. 

A different technique,  proposed by Lanza de Cristoforis and referred to as a
`functional analytic approach', aims at expressing the dependence of the solution on perturbation in terms of real analytic functions. This approach has so far
been applied  to the study of various elliptic problems, including
problems with nonlinear conditions.
For problems involving the Laplace operator we refer the reader to the papers of Lanza de Cristoforis (see, {\it e.g.},
 \cite{La07,La10}),  Dalla Riva and Musolino (see, {\it e.g.}, \cite{DaMu12,DaMu15,DaMu}),
 and  Dalla Riva, Musolino, and Rogosin \cite{DaMuRo15}, where the computation
 of the coefficients of the power series expansion of the resulting analytic maps is reduced to the
 solution of certain recursive systems of boundary integral equations.

In the present paper, we plan to exploit the functional analytic approach to represent the map that associates $\be$ with (suitable restrictions of) the solution $\ueps $
in terms of real analytic maps  with values in convenient Banach spaces of functions and of known
elementary functions of $\e{1}$ and $\e{2}$ (for the definition of real
analytic maps in Banach spaces, see Deimling \cite[p.~150]{De85}). Then we can recover
asymptotic approximations similar to those obtainable from the expansion methods. For example,
if we know that, for $\e1 $ and $\e2 $ small and positive, the function in \eqref{eq:question} equals
a real analytic function defined in a whole neighborhood of $(0,0)$, then we know that such a map
can be expanded in a power series for $\e1 $ and $\e2 $ small, and that a truncation of this series is
an approximation of the solution.

To conclude the presentation of our strategy,  we would like to comment on
some novel techniques that we bring into the functional analytic approach for the analysis of our problem.
First, we describe how the functional analytic approach `normally' operates on a boundary
value problem defined on a domain that depends on a parameter $\be$ and degenerates in
some sense as $\be$ tends to a limiting value $\bf 0$. The initial step consists in applying potential
theory techniques to transform the boundary value problem into a system of boundary integral equations.
Then, possibly after some suitable manipulation, this system is written
as a functional equation of the form $\fL[\be,\bm]=0$, where $\fL$ is a (nonlinear) operator acting
from an open subset of a  Banach space $\sR\times\sB_1$  to another Banach  space $\sB_2$.
Here $\sR$ is a neighborhood of $\bf 0$
and the Banach spaces $\sB_1$ and $\sB_2$ are usually the direct product of Schauder spaces on the
boundaries of certain fixed domains.
The next step is to apply the implicit function theorem to the equation $\fL[\be,\bm]=0$ in order to understand
the dependence of $\bm$ on $\be$. Then we can deduce the dependence of the solution of the original
boundary value problem on $\be$.

The strategy adopted in this paper differs from the standard application of the functional analytic approach
in two ways.

\begin{itemize}
\item The first one concerns the potential theory used to transform the problem into a system of
integral equations.  To take care of the special geometry of the problem, instead of the classical
layer potentials for the Laplace operator, we construct layer potentials where the role of the fundamental
solution is taken by the Dirichlet Green's function of the upper half space. Since the hole collapses
on $\partial \Rnp \cap \dO$ as $\be$ tends to $\bf 0$, such a method allows us to eliminate the integral
equation defined on the part of the boundary of $\domeps $ where the boundary of the hole and the
exterior boundary interact {for} $\be={\bf 0}$. In Section \ref{prel}, we collect a number of general
results on such special layer potentials. We remark that if the union of $\Omega$ and its reflection
with respect to $\partial\Rnp$ is a regular domain, then there is no need to introduce special
layer potentials and the problem may be analyzed by means of a technique based on the functional analytic approach
and on a reflection argument (see~Costabel, Dalla Riva, Dauge, and Musolino \cite{CDDM16}).
However, under our assumption, the union of $\Omega$ and its reflection
with respect to $\partial\Rnp$ produces an edge on $\partial\Rnp$ and, thus, is not a regular domain.

\item By using the special layer potentials
mentioned above, we can transform problem \eqref{bvpe} into an equation
of the form $\fL[\be,\bm]=0$, where
the operator $\fL$ acts from an open set $\bJea\times\sB_1$  into a Banach  space $\sB_2$
whose construction is, in a certain sense, artificial. $\sB_2$ is the direct product
of a Schauder space and the image of a specific integral operator (see Propositions \ref{V} and \ref{prop.N}).
In this context, we have to be particularly
careful to check that the image of $\fL$ is actually contained in such a Banach space $\sB_2$,
and that $\fL$ is a real analytic operator (see Proposition \ref{prop.N}). We remark that this step
is instead quite straightforward in previous applications of the functional analytic approach
(see, {\it e.g.},  \cite[Prop.~5.4]{DaMu}). Once this work is completed, we are ready to use
the implicit function theorem and deduce the dependence of the solution on $\be$.
\end{itemize}

\subsection{Main results}\label{mr}

To perform our analysis, in addition to \eqref{Omega}--\eqref{DOmega} we also assume that
$\dom$ satisfies the condition
\begin{equation}\tag{$H_{3}$} \label{D+Omega}
\text{$\overline{\dpO}$ is a compact submanifold with boundary of $\Rn$ of class $\Ca{1}$.}
\end{equation}
In the two-dimensional case, this condition takes the form
\begin{equation}\tag{$H_{3}$} \label{D+Omega2D}
\text{$\overline{\d0O}$ is a finite union of closed disjoint intervals in $\partial\mathbb{R}^2_+$.}
\end{equation}
In particular, we note that assumption \eqref{D+Omega} implies the existence of linear and continuous
extension operators $E^{k,\alpha}$ from $\Ca{k}(\overline{\dpO})$ to $\Ca{k}(\dO)$, for $k=0,1$
(cf.~Lemma \ref{ext} below). This allows us to change from functions defined on $\dpO$ to
functions defined on $\dO$ (and {\it viceversa}), preserving their regularity.

To prove our analyticity result, we consider a regularity condition on the
Dirichlet datum around the origin, namely
\begin{equation}\tag{$H_{4}$}\label{go_analytic}
\text{there exists $r_0>0$ such that the restriction $\go_{|\Bn{r_0}\cap\d0O}$ is real analytic\Be.}
\end{equation}
As happens for the solution to the Dirichlet problem in a domain with a small hole
`far' from the boundary, we show that $\ueps $ converges as $\e1 \to 0$ to a function $\uzero$
that is the unique solution in $\Ca{1}(\overline{\dom})$ of the following Dirichlet problem in the
unperturbed domain $\dom$:
\[
\begin{cases}
\Delta u=0 & \mbox{ in } \dom \,,\\
u=\go & \mbox{ on } \dO\, .\\
\end{cases}
\]

We note that $\uzero$ is harmonic, and therefore  analytic, in the interior of $\dom$.
This fact is useful in the study of the Dirichlet problem in a domain with a hole that shrinks
to an interior point of $\Omega$. If, instead, the hole shrinks to a point on the boundary,
as it does in this paper, then we have to introduce condition \eqref{go_analytic} in order
to ensure that $\uzero$ has an analytic (actually, harmonic) extension around the limit point.
Indeed, by  \eqref{go_analytic} and a classical argument based on the Cauchy-Kovalevskaya Theorem,
we can prove the following assertion (cf.~\ref{app:CK}).
\begin{prop}
\label{Ub}
There is $r_1\in]0,r_0]$ and a function $U_0$ from $\overline{\Bn{r_1}}$ to $\mathbb{R}$
such that $\Bnp{r_1}\subseteq\dom$ and
$$
\begin{cases}
\Delta U_0 =0&\mbox{ in }\Bn{r_1},\\
U_0 =\uzero&\mbox{ in }\overline{\Bnp{r_1}},
\end{cases}
$$
where $\Bnp{r}=\Bn{r}\cap\mathbb{R}^n_+$.
\end{prop}
 Then, possibly shrinking $\e{1}^{\rm ad}$, we may assume that
\begin{equation}\label{Ubassumption}
\e1 \p+\e1 \e2 \overline{\incl}\subseteq\Bn{r_1},\qquad\forall\be\in \bJea\,.
\end{equation}
We now give our answers to question \eqref{eq:question}. 
We remark that, instead of  the evaluation of $\ueps $ at a point $\px$, we consider  its restriction
to a suitable subset $\Omega'$ of $\Omega$.

\subsubsection{The case $\be\to\b0$ in spaces of dimension $n\ge 3$}

For $\be\to\b0$, the question \eqref{eq:question} is answered differently when $n\geq3$ and $n=2$.
If $n\geq3$, the statement is easier.
\begin{thm}\label{Ue1e2}
Let $\dom'$ be an open subset of $\dom$ such that $0\notin\overline{\dom'}$.
There are {$\be'\in\bIea$} with
$\overline{\incleps }\cap\overline{\dom'}=\emptyset$ for all $\be\in \bJep$
and a real analytic map ${\fU}_{\dom'}$ from $\bJep$ to $\Ca{1}(\overline{\dom'})$ such that
\begin{equation}\label{Ue1e2.eq1}
{\ueps}_{ |\overline{\dom'}}={\fU}_{\dom'}[\be]\qquad\forall \be\in\bIep\,.
\end{equation}
Furthermore,
\begin{equation}\label{Ue1e2.eq2}
{\fU}_{\dom'}[{\bf 0}]=u_{0|\overline{\dom'}}\,.
\end{equation}
\end{thm}

Theorem \ref{Ue1e2} implies that there are $\be''\in]0,\be'[$
and  a family of functions $\{U_{i,j}\}_{i,j\in\N^2}\subseteq \Ca{1}(\overline{\dom'})$  such that
\[
u_{\be}(\px)=\sum_{i,j=0}^\infty U_{i,j}(\px)\;\e1^i\e2^j\,,\qquad\forall\be\in]0,\be''[\,,\;\px\in\overline{\dom'},
\]
with the power series $\sum_{i,j=0}^\infty U_{i,j}\; \e1^i\e2^j$ converging in the norm of
$ \Ca{1}(\overline{\dom'})$ for $\be$ in an open neighborhood of $\b0$. Consequently,
one can compute asymptotic approximations for $\ueps $ whose convergence
is  guaranteed by our preliminary analysis.

A result similar to Theorem \ref{Ue1e2} is expressed in Theorem \ref{Ve1e2}
concerning the behavior of $\ueps$ close to the boundary of the hole, namely,
for the rescaled function  $\pt\mapsto\ueps(\e1\p+\e1\e2\,\pt)$. Later, in Theorems \ref{Enge3}
and \ref{Enge3g}  we present real analytic continuation results also for the energy integral
$\int_{\domeps}|\nabla\ueps|^2d\px$. In particular, we show that the limiting value of the energy integral   for $\be\to\b0$  is the energy of the unperturbed solution $u_0$.

\subsubsection{The case $\be\to\b0$ in two-dimensional space}\label{beto0sec}

Here, we need to introduce a curve $\eta \mapsto \be(\eta)\equiv (\e1(\eta),\e2(\eta))$
that describes the values attained  by the parameter $\be$ in a specific way. The reason is
the presence of the quotient
\begin{equation}\label{eq:gam0}
\frac{\log\e1}{\log(\e1\e2 )},
\end{equation}
which plays an important role in the description of  $\ueps$ for $\be$ small. We remark that the expression
\eqref{eq:gam0} has no limit as $\be\to\b0$. Therefore, we choose a function $\eta\mapsto\be(\eta)$ from $]0,1[$ to
$\bIea$ such that
\begin{equation}\label{eq:gam1}
\lim_{\eta \to 0^+}\be(\eta)=\b0,
\end{equation}
and for which
\begin{equation}\label{eq:gam2}
\lim_{\eta \to 0^+}\frac{\log\e1(\eta)}{\log(\e1(\eta) \e2(\eta) )}\quad\text{ exists and equals }
\lambda\in[0,1[.
\end{equation}
It is also convenient to denote by $\bde$ the function
\begin{equation}\label{eq:bfdelta}
\begin{array}{rcl}
\bde :\quad ]0,1[ &\to&  \mathbb{R}^2,\\
\eta &\mapsto&
\bde(\eta)\equiv \big(\delta_1(\eta),\delta_2(\eta)\big)\equiv \biggl(\frac{1}{\log \big(\e1(\eta)\e2(\eta)\big)},
\frac{\log\e1(\eta) }{\log \big(\e1(\eta)\e2(\eta)\big)}\biggr)\,,
\end{array}\end{equation}
so that
\[
\lim_{\eta\to 0^+}\bde(\eta)=(0,\lambda)\,.
\]
In Section \ref{n=2}, we prove an assertion
that describes $u_{\be(\eta)}$ in terms of a real analytic function of four real variables
evaluated at $(\be(\eta),\bde(\eta))$.

\begin{thm}\label{thm:Ue1gamma-introd}
Let $\lambda\in[0,1[$. Let $\Omega'$ be an open subset of $\Omega$ with $0 \notin \overline{\Omega'}$.
Then there are  $\be'\in\bIea$,
an open neighborhood  $\cV_{\lambda}$ of $(0,\lambda)$ in $\mathbb{R}^2$, and a real analytic map
\[
\fU_{\dom'}:]-\be',\be'[\times \cV_{\lambda}\to\Ca{1}(\overline{\dom'}),
\]
such that
\begin{equation}\label{Ue1gamma.eq1}
u_{\be(\eta)|\overline{\dom'}}=\fU_{\dom'}\bigl[\be(\eta),\bde(\eta)\bigr]\,,\qquad\forall \eta \in]0,\eta'[\,.
\end{equation}
The equality in \eqref{Ue1gamma.eq1} holds for all parametrizations $\eta\mapsto\be(\eta)$ from $]0,1[$ to  $\bIea$
that satisfy  \eqref{eq:gam1} and \eqref{eq:gam2}. The function $\eta\mapsto\bde(\eta)$ is defined
as in \eqref{eq:bfdelta}. The pair  $\be'\in\bIea$ is small enough to yield
\begin{equation}\label{Ue1gamma.eq0}
\overline{\omega_{\be}} \cap \overline{\Omega'}=\emptyset\,,\qquad\forall \be \in ]-\be',\be'[\,,
\end{equation}
and $\eta'$ can be any number in $]0,1[$ such that
\[
\left(\be(\eta),\bde(\eta)\right) \in ]0,\be'[\times \cV_{\lambda}\,,\qquad\forall  \eta \in]0,\eta'[\,.
\]
At the singular point $(\b0, (0,\lambda))$, we have
\begin{equation}\label{Ue1gamma.eq2}
\fU_{\dom'}[\b0, (0,\lambda)]=u_{0|\overline{\dom'}}\,.
\end{equation}
\end{thm}

As a corollary to Theorem \ref{thm:Ue1gamma-introd}, we can write the solution $u_{\be(\eta)}$
in terms of a  power series in $\left(\be(\eta),\bde(\eta)\right)$ for $\eta$ positive and small.
Specifically, there are $\eta''\in]0,\eta']$ and a family of functions
$\{U_{\boldsymbol\beta}\}_{{\boldsymbol\beta}\in\N^4}\subseteq \Ca{1}(\overline{\dom'})$ such that
\[
u_{\be(\eta)}(\px)=\sum_{{\boldsymbol\beta}\in\N^4}U_{\boldsymbol\beta}(\px)\;\e1(\eta)^{\beta_1}
\e2(\eta)^{\beta_2}\delta_1(\eta)^{\beta_3}(\delta_2(\eta)-\lambda)^{\beta_4}\qquad
\forall\eta\in]0,\eta''[\,,\;\px\in\overline{\dom'}\,.
\]
Moreover, the power series $\sum_{{\boldsymbol\beta}\in\N^4}U_{\boldsymbol\beta}\;\e1^{\beta_1}
\e2^{\beta_2}\delta_1^{\beta_3}(\delta_2-\lambda)^{\beta_4}$ converges in the norm of $ \Ca{1}
(\overline{\dom'})$ for $(\e1,\e2,\delta_1,\delta_2)$ in an open neighborhood of $(\b0, (0,\lambda))$.

We emphasize that the map $\fU_{\dom'}$, as well as the coefficients
$\{U_{\boldsymbol\beta}\}_{{\boldsymbol\beta}\in\N^4}$, depends on the limiting value $\lambda$,
but not on the specific curve $\be(\cdot)$ that satisfies \eqref{eq:gam2}.
 A result similar to Theorem \ref{thm:Ue1gamma-introd}  also holds,
which describes the behavior of the solution of problem \eqref{bvpe} close to the hole  (cf.~Theorem \ref{Ve1gamma}), for the energy integral
$\int_{\dome}|\nabla u_\eps|^2\,d\px$ (cf.~Theorem \ref{Enge2}), and for the total flux  through
the outer boundary $\int_{\partial\dom}\bn_\dom\cdot\nabla u_\eps\,d\sigma$ (cf.~Theorem \ref{Fluxe2}). In particular, we show that the limiting value of the energy integral is
\begin{equation}\label{Elimit}
\lim_{\eta\to 0}\int_{\dome}|\nabla u_{\be(\eta)}|^2\,d\px=
\int_{\Omega}|\nabla u_0|^2\,d\px+\int_{\mathbb{R}^2\setminus\incl}\left|\nabla v_0\right|^2\,d\px,
\end{equation}
where $v_0\in \Ca{1}_{\mathrm{loc}}(\mathbb{R}^2\setminus\incl)$ is the unique solution of
\begin{equation}\label{v0}
\left\{
\begin{array}{ll}
\Delta v_0=0&\text{in }\mathbb{R}^2\setminus\overline{\incl}\,,\\
v_0=\gi&\text{on }\dincl\,,\\
\sup_{\R^2\setminus\incl}|v_0|<+\infty\,.&
\end{array}
\right.
\end{equation}
In addition, we show that the flux on $\dO$ satisfies
\[
\lim_{\eta\to0}\int_{\partial\Omega}\bn_\Omega\cdot\nabla u_{\be(\eta)}\,d\sigma=0.
\]

Finally, we remark that the functions $\delta_1$ and $\delta_2$ are not uniquely defined.
For example, we may choose
\[
\tilde\delta_2(\eta)\equiv \frac{\log\e2(\eta) }{\log \big(\e1(\eta)\e2(\eta)\big)}
\]
or other similar alternatives instead of $\delta_2(\eta)$ (we note that $\tilde\delta_2(\eta)=1-\delta_2(\eta)$).
Furthermore, the solution may not depend on the quotient \eqref{eq:gam0}
if we consider problems with a different geometry. For instance, in the toy problem of Subsection \ref{toy},
the solution  \eqref{utoy} can be written as an analytic map of three variables evaluated at
$(\e1,\e2,({\log\e2})^{-1})$.  As we emphasize in a comment at the end of Subsection \ref{sub:L},
the reason for this simpler behavior is that in the  toy problem we do not have
an exterior boundary $\dpO$. It is  worth noting that a quotient similar to \eqref{eq:gam0}  plays a fundamental
role also in the two-dimensional Dirichlet problem with moderately close small holes,
which was investigated in \cite{DaMub} and where
it was shown that an analog of the limiting value $\lambda$ (cf.~\eqref{eq:gam2})
appears explicitly in the second term of the asymptotic expansion of the solution.

\subsubsection{The case  $\e1\to0$ with $\e2>0$ fixed in two-dimensional space}\label{fixed}

We remark that we may restrict our attention to the problem with $\e2=1$. Then the generic case of
$\e2=\e2^*\in]0,\e{2}^{\rm ad}[$  fixed is obtained by rescaling the reference domain
$\incl$ using the factor $\e2^*$.  We also remark that the restricted case is
a one-parameter problem. Consequently, it is convenient to define  $\ea\equiv\e{1}^{\rm ad}$,
$\incle\equiv\omega_{\e1,1}$, $\dome\equiv\Omega_{\e1,1}$, and $\ue\equiv u_{\e1,1}$ for all
$\varepsilon\in \Jea$. The next assertion is proved in Section \ref{e2=1}.

\begin{thm}\label{thm:Ue}
Let $\dom'$ be an open subset of $\dom$ such that $0\notin\overline{\dom'}$.
Then there are $\eps'\in]0,\e{1}^{\rm ad}[$ such that
\begin{equation}\label{Ue.eq0}
\overline{\incle }\cap\overline{\dom'}=\emptyset\qquad\forall\eps\in ]-\eps',\eps'[\,
\end{equation}
and a real analytic map $\fU_{\dom'}$ from $]-\eps',\eps'[$ to $\Ca{1}(\overline{\dom'})$ satisfying
\begin{equation}\label{Ue.eq1}
u_{\eps |\overline{\dom'}}=\fU_{\dom'}[\eps],\qquad\forall \eps\in]0,\eps'[\,.
\end{equation}
Furthermore,
\begin{equation}\label{Ue.eq2}
\fU_{\dom'}[0]=u_{0|\overline{\dom'}}\,.
\end{equation}
\end{thm}

Theorem \ref{thm:Ue} implies that there are $\eps''\in]0,\epsilon'[$
and  a sequence of functions $\{U_j\}_{j\in\N}\subseteq \Ca{1}(\overline{\dom'})$  such that
\[
u_{\eps}(\px)=\sum_{j=0}^\infty U_j(\px)\,\eps^j\qquad\forall\eps\in]0,\eps''[\,,\;\px\in\overline{\dom'},
\]
with the power series $\sum_{j=0}^\infty U_j\,\eps^j$ converging in the norm of
$ \Ca{1}(\overline{\dom'})$ for $\eps$ in an open neighborhood of $0$.

A result similar to Theorem \ref{thm:Ue} is also established for  the behavior of $u_\eps$ near
the boundary of the hole (cf.~Theorem \ref{Ve}), for the energy integral
$\int_{\dome}|\nabla u_\eps|^2\,d\px$ (cf.~Theorem \ref{Enge}), and for the total flux  through
the outer boundary $\int_{\partial\dom}\bn_\dom\cdot\nabla u_\eps\,d\sigma$ (cf.~Theorem \ref{Fluxe}).
In particular, we show that the limiting value of the energy integral is
\begin{equation}\label{Elimite}
\lim_{\eps\to 0}\int_{\dome}|\nabla u_{\eps}|^2\,d\px=
\int_{\Omega}|\nabla u_0|^2\,d\px+\int_{{\mathbb{R}^2_+}\setminus(\p+\incl)}
\left|\nabla w_\ast\right|^2\,d\px\,,
\end{equation}
and that the limiting value of the total flux is
\begin{equation}\label{Flimite}
\int_{\p+\partial\incl}  \bn_{\p+\incl}\cdot \nabla w_\ast\,d\sigma
\end{equation}
where $w_\ast$ is the unique solution in $\Ca{1}_{\mathrm{loc}}(\overline{\mathbb{R}^2_+}
\setminus(\p+\incl))$ of
\begin{equation}\label{w_*}
\left\{
\begin{array}{ll}
\Delta w_\ast=0&\text{in }{\mathbb{R}^2_+}\setminus(\p+\overline\incl)\,,\\
w_\ast(\pt)=\gi (\pt-\p)&\text{for all }\pt\in\p+\dincl\,,\\
w_\ast=\go(0)&\text{on }\partial{\mathbb{R}^2_+}\,,\\
\lim_{\pt\to\infty}w_\ast(\pt)=\go(0)\,.&
\end{array}
\right.
\end{equation}
We remark that for suitable choices of $\go$ and $\gi$, the limiting value of the energy integral differs
from the one in \eqref{Elimit}, which emphasizes the difference between the two regimes.  Besides, the limit value of the total flux \eqref{Flimite} equals $0$ only for special choices of $\go$, $\gi$.

\subsection{Numerical illustration of the results.}

In our numerical simulations,
the domain $\Omega$ is a `stadium' represented by the union of the rectangle $[-2,2]\times[0,2]$ and
two half-disks. The origin $(0,0)$ is in the middle of a segment of the boundary.  We choose $\p=(1,1)$,
and the inclusion is a small disk as described in Figure \ref{Figure:description}.
The small parameter $\be$ is $\varepsilon_{1}=\left(\frac{2}{3}\right)^{n_{1}}$,
$\varepsilon_{2}=\left(\frac{2}{3}\right)^{n_{2}}$ for integers $1\leq n_{1}\leq 16$,
and $1\leq n_{2}\leq 20$.
\begin{figure}[!ht]
\centering
\begin{subfigure}[b]{.45\textwidth}
 \centering
 \includegraphics[width=\textwidth]{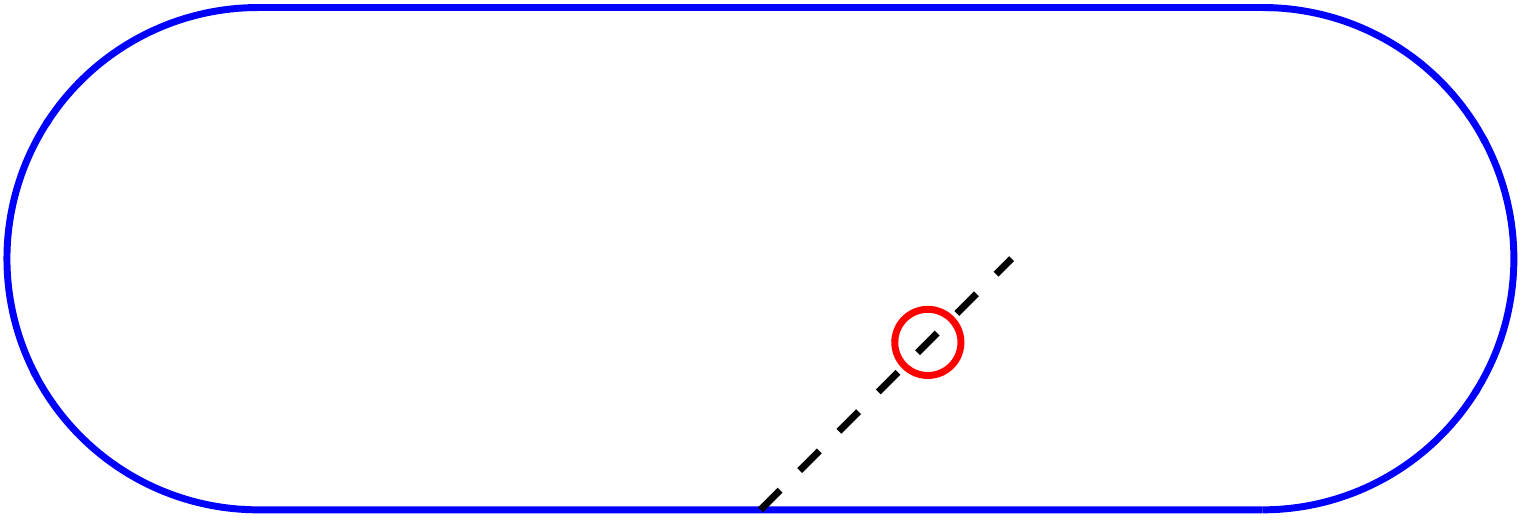}
 \caption{$n_{1}=1$ and $n_{2}=4$, $4312$ triangles.}
\end{subfigure}
\hfill\begin{subfigure}[b]{.45\textwidth}
 \centering
 \includegraphics[width=\textwidth]{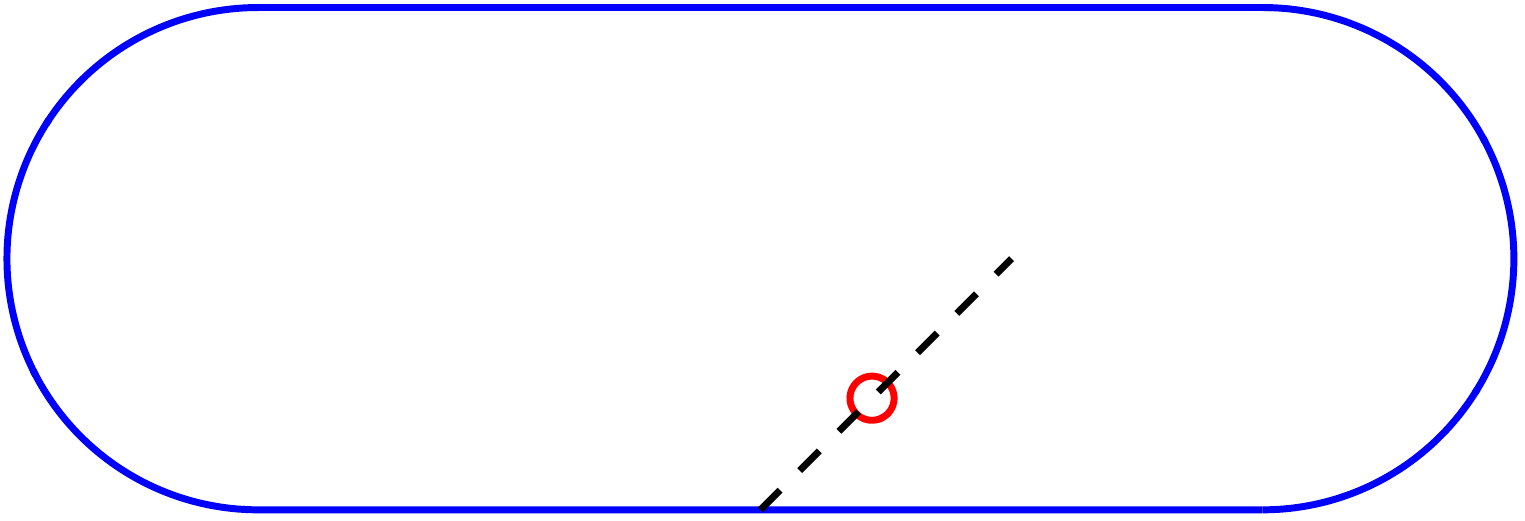}
 \caption{$n_{1}=2$ and $n_{2}=4$, $4364$ triangles.}
\end{subfigure}
\caption{Different computational domains.\label{Figure:description}}
\end{figure}

To approximate the solution $u_{\be}$ of the boundary value problem, we use a $\mathbb{P}^4$
finite element method on an adapted triangular mesh as provided by the Finite Element Library M\'ELINA
(see \cite{Ma07}).  Figures \ref{Figure:energy1}--\ref{Figure:energy3} exhibit the computed square root of the
 energy integral, this is the norm $\|\nabla u_{\be}\|_{\sL^2(\domeps)}$, in the previously
defined configurations.  In Figure \ref{Figure:energy1}, we take $\go=0$ and $\gi=1$, so
the sum in \eqref{Elimit} is $0$ and the limiting energy \eqref{Elimite} is strictly positive (note that with such $g^o$ and $g^i$ the energy coincides with the electrostatic capacity of $\incleps$ in $\dom$).
\begin{figure}[!ht]
\centering
\begin{subfigure}[b]{.45\textwidth}
 \centering
 \includegraphics[width=\textwidth]{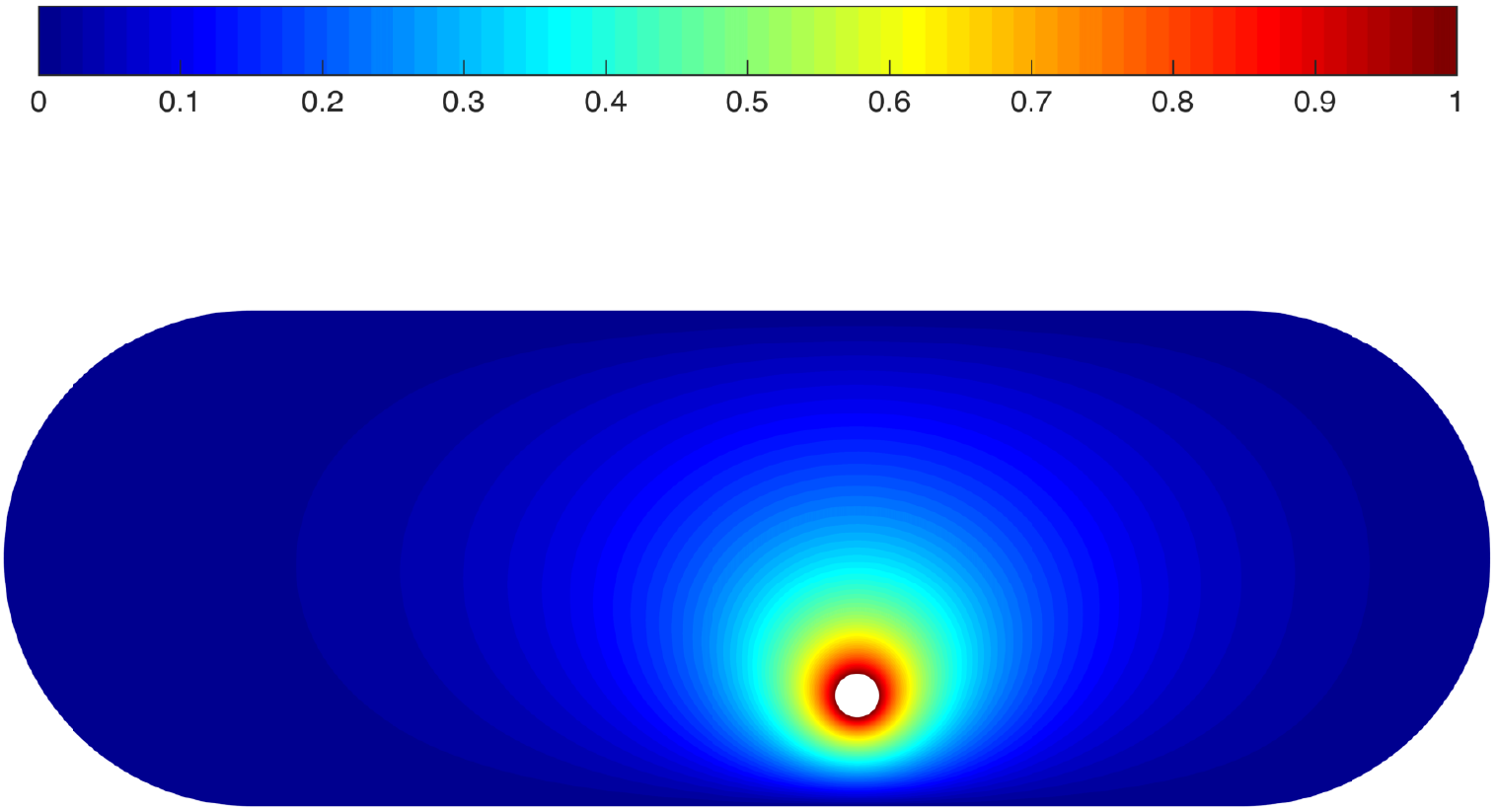}
 \caption{$u_{\be}$ for $n_{1}=2$ and $n_{2}=4$.}
\end{subfigure}
\hfil
\begin{subfigure}[b]{.45\textwidth}
 \centering
\includegraphics[width=\textwidth]{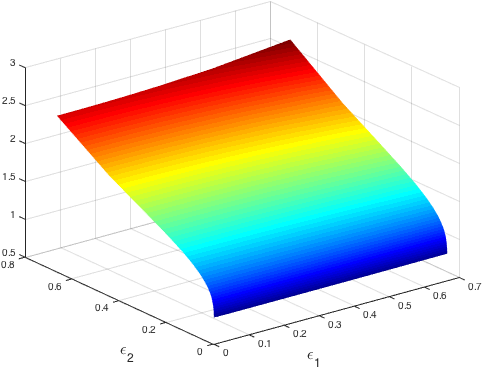}
 \caption{Norm $\|\nabla u_{\be}\|_{\sL^2(\domeps)}$.} 
\end{subfigure}
\caption{Case where $\go=0$ and $\gi=1$.\label{Figure:energy1}}
\end{figure}

\begin{figure}[!ht]
\centering
\begin{subfigure}[b]{.4\textwidth}
 \centering
 \includegraphics[width=\textwidth]{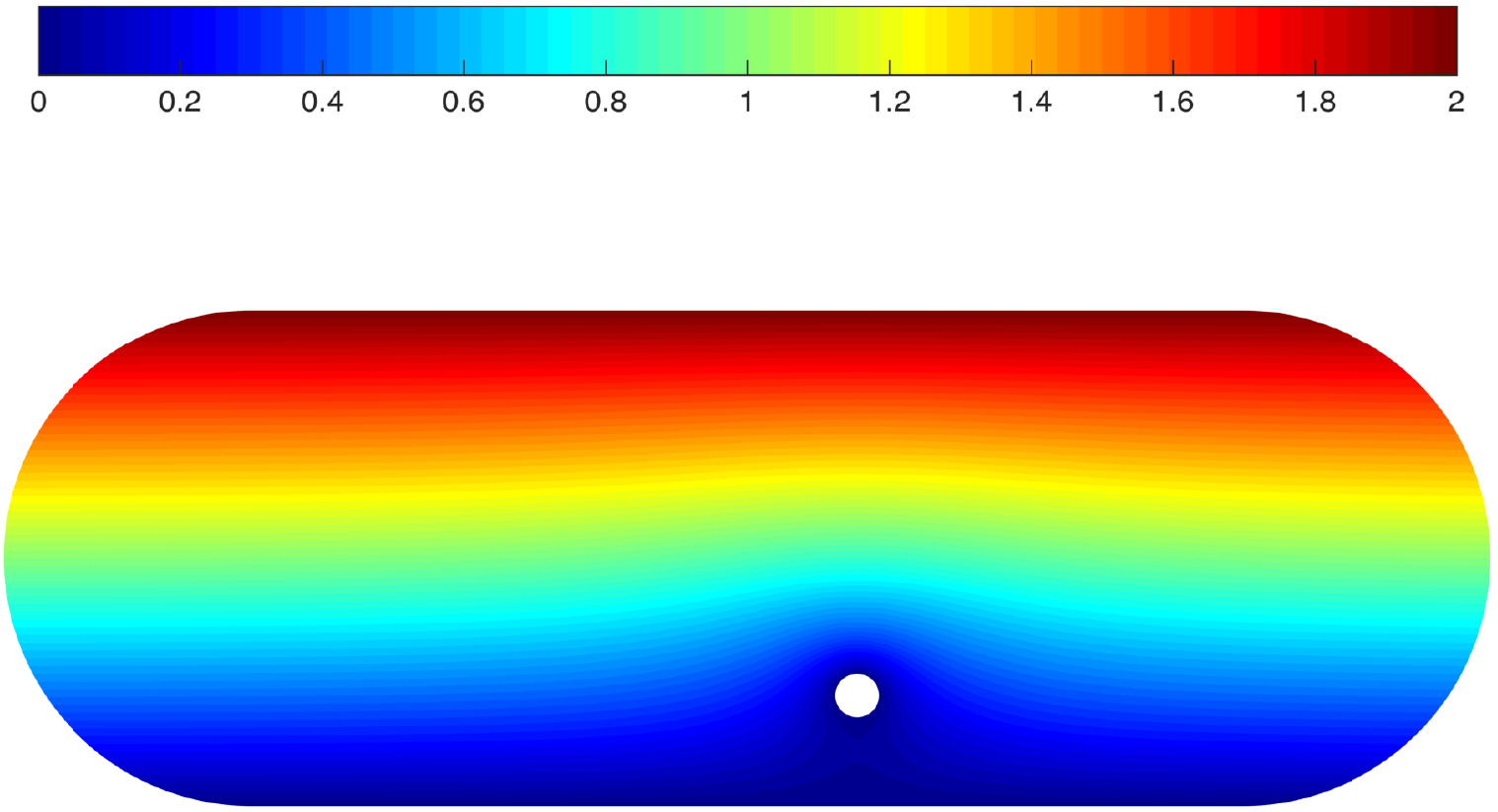}
 \caption{$u_{\be}$ for $n_{1}=2$ and $n_{2}=4$.}
\end{subfigure}
\hfil
\begin{subfigure}[b]{.4\textwidth}
 \centering
\includegraphics[width=\textwidth]{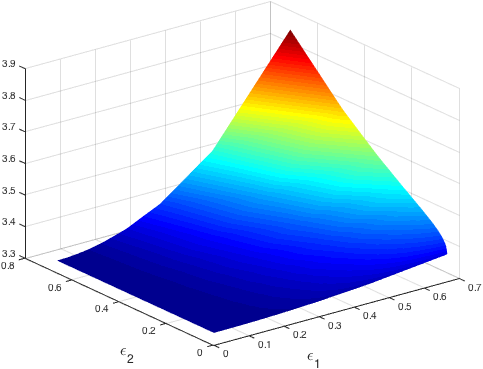}
 \caption{Norm $\|\nabla u_{\be}\|_{\sL^2(\domeps)}$. } 
\end{subfigure}
\caption{Case $\gi=0$ and $\go=x_{2}$.\label{Figure:energy2}}
\end{figure}

\begin{figure}[!ht]
\centering
\begin{subfigure}[b]{.4\textwidth}
 \centering
 \includegraphics[width=\textwidth]{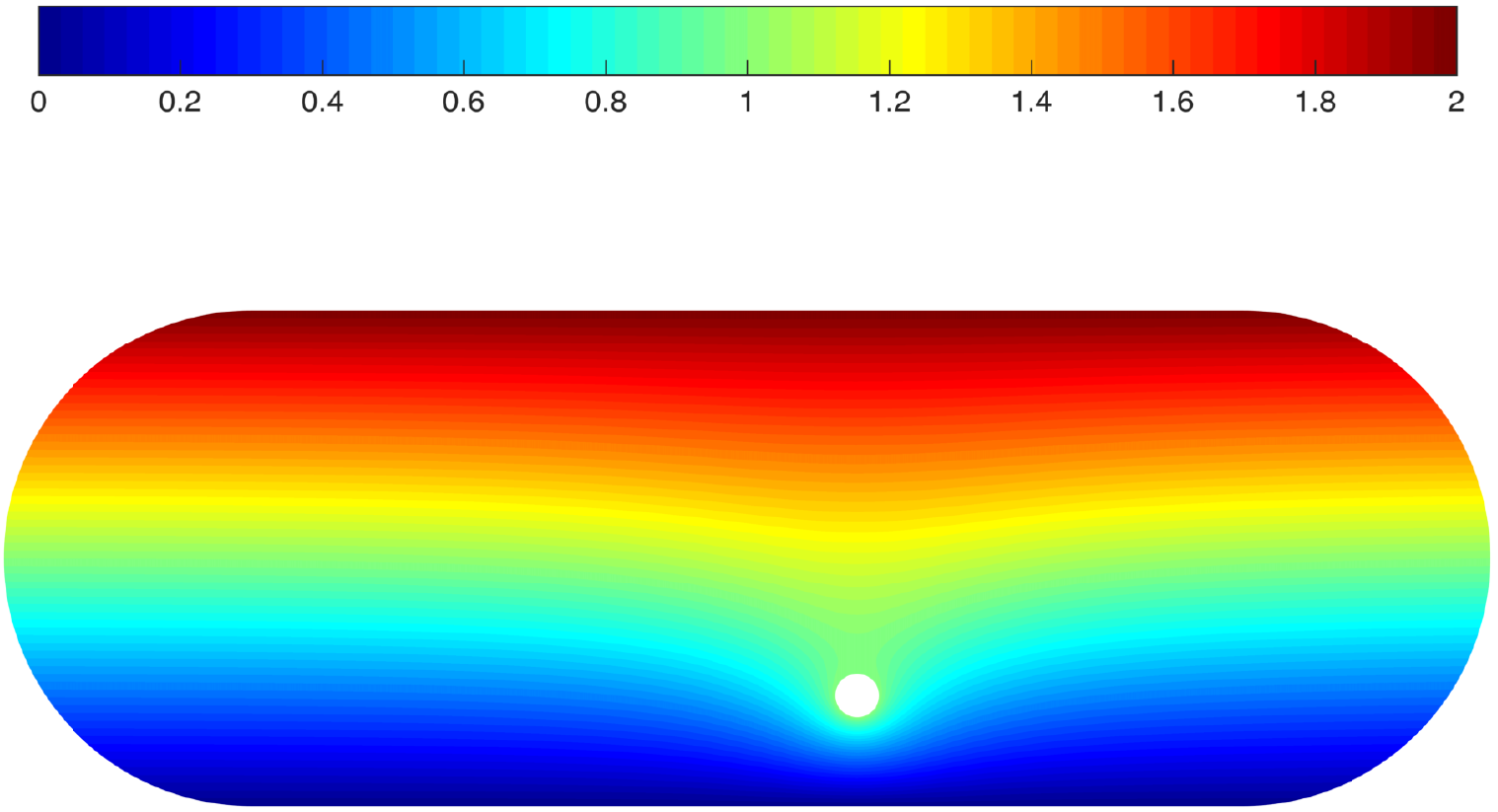}
 \caption{$u_{\be}$ for $n_{1}=2$ and $n_{2}=4$.}
\end{subfigure}
\hfil
\begin{subfigure}[b]{.4\textwidth}
 \centering
\includegraphics[width=\textwidth]{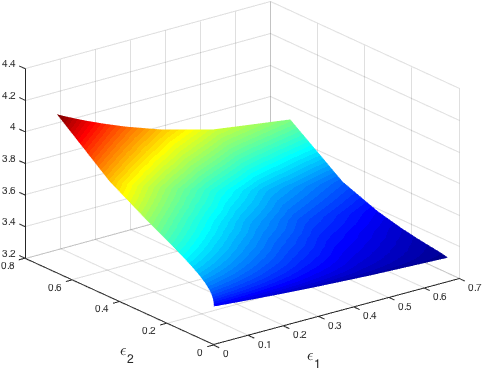}
 \caption{Norm $\|\nabla u_{\be}\|_{\sL^2(\domeps)}$.} 
\end{subfigure}
\caption{Case $\gi=1$ and $\go=x_{2}$.\label{Figure:energy3}}
\end{figure}

To illustrate  the different types of behavior of the energy integral, we now consider
$\go=x_{2}$ and either $\gi=0=\go(0)$ (see Figure \ref{Figure:energy2}),
or $\gi=1\ne \go(0)$ (see Figure \ref{Figure:energy3}).
  Notice that with that choice of $\go$, we have $\|\nabla\uzero\|_{\sL^2(\dom)}=\sqrt{8+\pi}\simeq 3.34$, which is the limiting value observed when $\gi=0$ in Figure \ref{Figure:energy2}.
On the contrary, in the numerical results for $\gi=1$, the energy has a different limiting value
whether  both $\e1$ and $\e2$ tend to $0$ or $\e1$ tends to $0$ with $\e2$ fixed,
in agreement with our expectation when $\gi\neq\go(0)$. When both $\e1$ and $\e2$
tend to $0$,  the limiting value of the energy is the same as in the well-known
case where $\e2\to 0$ with $\e1$  fixed (that is, when the hole shrinks to an interior point of $\dom$).
We notice that in the latter case, the energy appears to converge at a slow
logarithmic rate (see, in particular, Figures \ref{Figure:energy1} and \ref{Figure:energy3}); this is also a well-known fact,
predicted by theoretical analysis (see, {\it e.g.}, Maz'ya, Nazarov, and Plamenevskij \cite{MaNaPl00}).

\subsection{Structure of the paper}

The paper is organised as follows. In Section \ref{prel}, we present some preliminary results in potential
theory and study the layer potentials with integral kernels consisting of the Dirichlet
Green's function of the half space.  Section \ref{n>=3} is devoted to the $n\ge 3$ dimensional case. Here we prove our analyticity result stated in Theorem \ref{Ue1e2}.
In Section  \ref{n=2}, we study the two-dimensional case for $\be\to\b0$. In particular, we  prove Theorem \ref{thm:Ue1gamma-introd}. In Section \ref{e2=1}, we consider the case
where $n=2$ and $\e{1}\to 0$ with $\e{2}=1$ fixed and we prove Theorem \ref{thm:Ue}. Concluding remarks are presented in Section \ref{conclusions}. Some routine technical tools have been placed in the Appendix. Specifically, in \ref{app:pot}
we prove some decay properties of the Green's function and the associated single-layer potential,
and in  \ref{app:CK} we present an extension result based on the Cauchy-Kovalevskaya Theorem.

\section{Preliminaries of potential theory}\label{prel}

In this section, we introduce some technical results and notation.  Most of them deal with the potential theory  constructed with the Dirichlet Green's function of the upper half space. Throughout the section we take
\[
n \in \mathbb{N}\setminus \{0,1\}\, .
\]

\subsection{Classical single and double layer potentials}
As a first step, we  introduce the classical layer potentials for the Laplace equation and thus we introduce the fundamental solution $S_n$ of $\Delta$ defined by
\[
S_n(\px)\equiv
\begin{cases}
\frac{1}{s_{n}}\log|\px|&\text{if }n=2\,,\\[5pt]
 \frac{1}{(2-n)s_n}|\px|^{2-n}&\text{if }n\ge 3\,,
 \end{cases}
 \qquad \forall \px \in \Rn \setminus \{0\}\, ,
\]
where $s_{n}$ is the $(n-1)$-dimensional measure of the boundary of the unit ball in $\R^n$. In the sequel $\domgen$ is a generic open bounded connected subset of $\R^n$ of class $\Ca{1}$.

\begin{defn}[Definition of the layer potentials]
For any $\phi\in \Ca{0}(\partial\domgen)$,  we define
\[
v_{S_n}[\partial\domgen,\phi](\px)\equiv\int_{\partial\domgen}\phi(\py)S_n(\px-\py)\,d\sigma_{\py},\qquad\forall \px\in\Rn\, ,
\]
where $d\sigma$ denotes the area element on $\partial\domgen$. \\
The restrictions of $v_{S_n}[\partial\domgen,\phi]$ to $\overline\domgen$ and to $\Rn\setminus\domgen$ are denoted $v_{S_n}^i[\partial\domgen,\phi]$ and $v_{S_n}^e[\partial\domgen,\phi]$ respectively (the letter `$i$' stands for `interior' while the letter `$e$' stands for `exterior').
\\
For any $\psi\in \Ca{1}(\partial\domgen)$,  we define
\[
w_{S_n}[\partial \domgen,\psi](\px)\equiv-\int_{\partial \domgen}\psi(\py)\;\bn_{\domgen}(\py)\cdot\nabla S_n(\px-\py)\,d\sigma_\py,\qquad\forall \px\in\Rn\,,
\]
where $\bn_\domgen$ denotes the outer unit normal to $\partial\domgen$ and the symbol $\cdot$ denotes the scalar product in $\Rn$.
\end{defn}
To describe the regularity properties of these layer potentials we will need the following definition.
\begin{defn}\label{def.laypot1}
We denote by $\Ca{1}_{\mathrm{loc}}(\Rn\setminus\domgen)$ the space of functions on $\Rn\setminus\domgen$ whose restrictions to $\overline{\cO }$ belong to $\Ca{1}(\overline{\cO })$ for all open bounded subsets $\cO $ of $\Rn\setminus\domgen$.\\
$\Ca{0}_\#(\partial\domgen)$ denotes the subspace of $\Ca{0}(\partial\domgen)$ consisting of the functions $\phi$ with $\int_{\partial\domgen}\phi\, d\sigma=0$.
\end{defn}

Let us now present some well known regularity  properties of the single and double layer potentials.
\begin{prop}[Regularity of layer potentials]\label{clreg}
If $\phi\in\Ca{0}(\partial\domgen)$, then  the function $v_{S_n}[\partial\domgen,\phi]$ is continuous from $\Rn$ to $\mathbb{R}$. Moreover,  the restrictions  $v_{S_n}^i[\partial\domgen,\phi]$ and $v_{S_n}^e[\partial\domgen,\phi]$ belong to $\Ca{1}(\overline{\domgen})$ and to $\Ca{1}_{\mathrm{loc}}(\Rn\setminus\domgen)$, respectively. \\
If $\psi\in\Ca{1}(\partial\domgen)$, then  the restriction $w_{S_n}[\partial\domgen,\psi]_{|\domgen}$ extends to a function $w_{S_n}^i[\partial\domgen,\psi]$ of $\Ca{1}(\overline{\domgen})$ and the restriction $w_{S_n}[\partial\domgen,\psi]_{|\Rn\setminus\overline{\domgen}}$ extends to a function $w_{S_n}^e[\partial\domgen,\psi]$ of $\Ca{1}_{\mathrm{loc}}(\Rn\setminus\domgen)$.
\end{prop}

In the next Proposition \ref{cljump} we  recall the classical jump formulas (see, {\it e.g.}, Folland \cite[Chap.~3]{Fo95}).

\begin{prop}[Jump relations of layer potentials]\label{cljump}
For any $\px\in\partial\domgen$, $\psi\in \Ca{1}(\partial\domgen)$, and $\phi\in \Ca{0}(\partial\domgen)$, we have
\begin{align*}
w_{S_n}^{\sharp}[\partial\domgen,\psi](\px)&=\frac{\ssei}{2}\psi(\px)+w_{S_n}[\partial\domgen,\psi](\px)\,,\\
\bn_\dom(\px)\cdot\nabla v_{S_n}^\sharp[\partial\domgen,\phi](\px)&=-\frac{\ssei}{2}\phi(\px)+\int_{\partial\domgen}\phi(\py)\bn_{\dom}(\px)\cdot\nabla S_n(\px-\py)\,d\sigma_\py\,,
\end{align*}
where $\sharp = i,e$ and $\ssi=1$, $\sse=-1$.
\end{prop}

We will exploit the following classical result of potential theory.
\begin{lem}\label{viso}
The map $\begin{array}{rcl} 
& \\
\Ca{0}_\#(\partial\domgen)\times\mathbb{R}&\to&\Ca{1}(\partial\domgen)\\
(\phi,\xi)&\mapsto & v_{S_n}[\partial\domgen,\phi]_{|\partial\domgen}+\xi
 \end{array}$
is an isomorphism.\\ 
Moreover, if $n\ge 3$, then the map 
$\begin{array}{rcl} & \\
\Ca{0}(\partial\domgen)& \to &\Ca{1}(\partial\domgen)\\
\phi&\mapsto &v_{S_n}[\partial\domgen,\phi]_{|\partial\domgen}
\end{array}$ is an isomorphism.
\end{lem}

\subsection{Green's function for the upper half space and associated layer potentials}
As mentioned above, a  key tool for the analysis of problem \eqref{bvpe} are layer potentials constructed with the Dirichlet Green's function of the upper half space instead of the classical fundamental solution $S_n$.  Transforming problem \eqref{bvpe}  by means of these layer potentials will lead us to a system of integral equations with no integral equation on $\partial_0 \dom$, which is the part of the boundary of $\dO$ where the inclusion $\incleps$ collapses for $\be=\bf 0$.

Let us begin by introducing some notation. We denote by $\msigma$ the reflexion with respect to the hyperplane $\partial\Rnp$, so that
\[
\msigma(\px)\equiv(x_{1},\dots,x_{n-1},-x_{n})\,,\qquad \forall \px=(x_{1},\dots,x_{n})\in\R^n\,.
\]
Then we denote by ${\G}$ the Green's function defined by
\[
{\G}(\px,\py)\equiv S_n(\px-\py)-S_n(\msigma(\px)-\py),\quad \forall (\px,\py) \in \Rn\times\Rn\text{ with }\py\neq \px\text{ and }\py\neq\msigma(\px).
\]

We observe that
\begin{equation}\label{Gsym}
 {\G}(\px,\py)={\G}(\py,\px),\quad\forall(\px,\py)\in\Rn\times\Rn\text{ with }\py\neq \px\text{ and }\py\neq\msigma(\px),
\end{equation}
and
\begin{equation}\label{Gzero}
{\G}(\px,\py)=0,\quad\forall(\px,\py)\in\partial\Rnp\times\Rn\text{ with }\py\neq \px\text{ and }\py\neq\msigma(\px).
\end{equation}
We denote by the symbols $\nabla_{\px}{\G}(\px,\py)$ and $\nabla_{\py}{\G}(\px,\py)$  the gradient of the function $\px \mapsto {\G}(\px,\py)$ and of the function $\py\mapsto {\G}(\px,\py)$, respectively. If $\domgen$ is a subset of $\Rn$, we find convenient to set $\msigma(\domgen)\equiv\{\px\in\Rn\;|\;\msigma(\px)\in\domgen\}$.
We now introduce analogs of the classical layer potentials of Definition \ref{def.laypot1} obtained by replacing $S_{n}$ by the Green's function ${\G}$. In the sequel, $\domgen_+$ denotes an open bounded connected set contained in $\Rnp$ and of class $\Ca{1}$.

\begin{defn}[Definition of layer potentials derived by $G$] \label{def.SDG+}
For any $\phi \in \Ca{0}(\partial\domgen_+)$, we define
\[
v_{{\G}}[\partial\domgen_+,\phi](\px)\equiv\int_{\partial\domgen_+}\phi(\py) {\G}(\px,\py)\,d\sigma_\py,\qquad\forall \px\in\Rn\, .
\]
The restrictions of $v_{{\G}}[\partial\domgen_+,\phi]$ to $\overline\domgen_+$ and $\overline{\Rnp\setminus\domgen_+}$ are denoted $v_{{\G}}^i[\partial\domgen_+,\phi]$ and $v_{{\G}}^e[\partial\domgen_+,\phi]$ respectively.\\
For any subset $\Gamma$ of the boundary $\partial\domgen_+$ and  for any $\psi\in \Ca{1}(\partial\domgen_+)$,  we define
\[
w_{{\G}}[\Gamma,\psi](\px)\equiv \int_{\Gamma}\psi(\py)\;\bn_{\domgen_+}(\py)\cdot\nabla_\py\, {\G}(\px,\py)\,d\sigma_\py,\qquad\forall \px\in\Rn\, .
\]
\end{defn}

By the definition of ${\G}$, we easily obtain the equalities
\[
v_{{\G}}[\partial \domgen_+,\phi](\px)=v_{S_n}[\partial \domgen_+,\phi](\px)-v_{S_n}[\partial \domgen_+,\phi](\msigma(\px))\,,\qquad\forall \px\in\Rn\\, \ \forall \phi \in \Ca{0}(\partial \domgen_+)\, ,
\]
and
\[
w_{{\G}}[\partial \domgen_+,\psi](\px)=w_{S_n}[\partial \domgen_+,\psi](\px)-w_{S_n}[\partial \domgen_+,\psi](\msigma({\px}))\,,\qquad\forall \px\in\Rn\,,\ \forall \psi \in \Ca{1}(\partial \domgen_+)\, .
\]
Thus one deduces by Propositions \ref{clreg} and \ref{cljump} the regularity properties and jump formulas for $v_{{\G}}[\partial \domgen_+,\phi]$ and $w_{{\G}}[\partial \domgen_+,\psi]$. 
\begin{prop}[Regularity and jump relations for the layer potentials derived by $G$]\label{jumps}
Let $\phi\in \Ca{0}(\partial \domgen_+)$ and $\psi\in \Ca{1}(\partial \domgen_+)$. Then
\begin{itemize}
\item the functions $v_{{\G}}[\partial \domgen_+,\phi]$ and $w_{{\G}}[\partial \domgen_+,\psi]$ are harmonic in $\domgen_+$, $\msigma({\domgen_+})$, and $\Rn\setminus\overline{\domgen_+\cup\msigma({\domgen_+})}$;
\item the function $v_{{\G}}[\partial\domgen_+,\phi]$ is continuous from $\Rn$ to $\mathbb{R}$ and the restrictions $v_{{\G}}^i[\partial\domgen_+,\phi]$ and $v_{{\G}}^e[\partial\domgen_+,\phi]$ belong to $\Ca{1}(\overline{\domgen_+})$ and to $\Ca{1}_{\mathrm{loc}}(\overline{\Rnp\setminus\domgen_+})$, respectively;
\item the restriction $w_{{\G}}[\partial\domgen_+,\psi]_{|\dom}$ extends to a function $w_{{\G}}^i[\partial\domgen_+,\psi]$ of $\Ca{1}(\overline{\domgen_+})$ and the restriction $w_{{\G}}[\partial\domgen_+,\psi]_{|\Rnp\setminus\overline{\domgen_+}}$ extends to a function $w_{{\G}}^e[\partial\domgen_+,\psi]$ of $\Ca{1}_{\mathrm{loc}}(\overline{\Rnp\setminus\domgen_+})$.
\end{itemize}
The jump formulas for the double layer potential are (with $\sharp = i,e$, $\ssi=1$, $\sse=-1$)
\begin{align*}
w_{{\G}}^\sharp [\partial\domgen_+,\psi](\px)&=\frac{\ssei}{2}\psi(\px)+w_{{\G}}[\partial\domgen_+,\psi](\px),&\forall \px\in\partial_+\domgen_+\,,\\
w_{{\G}}^i[\partial\domgen_+,\psi](\px)&=\psi(\px),&\forall \px\in\partial_0\domgen_+\,.
\end{align*}
Moreover, we have
\begin{align}
v_{{\G}}[\partial\domgen_+,\phi](\px)=0,&\qquad\forall \px\in\partial\Rnp\,,\label{eq.vG+dO}\\
w^e_{{\G}}[\partial\domgen_+,\psi](\px)=0,&\qquad\forall \px\in\partial\Rnp\setminus\partial_0\domgen_+\,.\nonumber
\end{align}
Here above, $\partial_0 \domgen_+\equiv\partial \domgen_+ \cap\partial\Rnp$ and $\partial_+\domgen_+\equiv\partial \domgen_+\cap\Rnp$.
\end{prop}

In the following lemma we show how the layer potentials with kernel $G$ introduced in Definition \ref{def.SDG+} allow to prove a corresponding Green-like representation formula.

\begin{lem}[Green-like representation formula in $\domgen_+$]\label{i-green}
Let $u^i\in \Ca{1}(\overline{\domgen_+})$ be such that $\Delta u^i=0$ in $\domgen_+$. Then we have
\begin{equation}\label{eq:i-green}
w_{{\G}}[\partial\domgen_+,u^i_{|\partial\domgen_+}]-v_{{\G}}[\partial\domgen_+,\bn_{\domgen_+}\cdot\nabla u^i_{|\partial\domgen_+}]=
\begin{cases}
u^i & \mbox{ in }\domgen_+,\\
0 
&\mbox{ in }\Rn\setminus\overline{\domgen_+\cup\msigma(\domgen_+)}\,.
\end{cases}\end{equation}
\end{lem}
\begin{proof}  Let us first consider $\px\in\domgen_+$. 
By the Green's representation formula (see, {\it e.g.}, Folland \cite[Chap.~2]{Fo95}), we have
\begin{equation}\label{eq:i-green:3}
u^i(\px)=-\int_{\partial\domgen_+}\bn_{\domgen_+}(\py) \cdot \nabla S_n(\px-\py)\, u^i(\py) \, d\sigma_\py-\int_{\partial\domgen_+}S_n(\px-\py)\;\bn_{\domgen_+}(\py)\cdot \nabla u^i(\py)\, d\sigma_\py {\Be,\quad} \forall \px \in \domgen_+\, .
\end{equation}
On the other hand, we note that if $\px \in \domgen_+$ is fixed, then the function $\py \mapsto S_n(\msigma({\px})-\py)$ is of class $\sC^{1}(\overline{\domgen_+})$ and harmonic in $\domgen_+$. Therefore, by the Green's identity, we have
\begin{equation}\label{eq:i-green:4}
0=\int_{\partial\domgen_+}\bn_{\domgen_+}(\py) \cdot \nabla S_n(\msigma({\px})-\py)\, u^i(\py)\, d\sigma_\py+\int_{\partial\domgen_+}S_n(\msigma({\px})-\py)\;\bn_{\domgen_+}(\py)\cdot\nabla u^i(\py)\, d\sigma_\py \qquad \forall \px \in \domgen_+\, .
\end{equation}
Then, by summing equalities \eqref{eq:i-green:3} and \eqref{eq:i-green:4} we deduce the validity of \eqref{eq:i-green} in $\domgen_+$.\\
Let us now consider any fixed $\px\in\Rn\setminus\overline{\domgen_+\cup\msigma(\domgen_+)}$. We observe that the functions $\py\mapsto S_n(\px-\py)$ and $\py\mapsto S_n(\msigma(\px)-\py)$ are harmonic on $\domgen_+$. Accordingly ${\G}(\px,\cdot)$ is an harmonic function in $\domgen_+$. Then a standard argument based on the divergence theorem shows that
\[
\int_{\partial\domgen_+} u^i(\py)\; \bn_{\domgen_+}(\py)\cdot\nabla_\py {\G}(\px,\py)-{\G}(\px,\py)\; \bn_{\domgen_+}(\py)\cdot\nabla u^i(\py)\; d\sigma_\py=0\,.
\]
\end{proof}

\vspace{\baselineskip}

\subsection{Mapping properties of the single layer potential $v_{{\G}}[\dO,\cdot]$}\label{ss:potom}

In order to analyze the $\be$-dependent boundary value problem \eqref{bvpe}, we are going to exploit the layer potentials with kernel derived by $G$ in the case when $\domgen=\domeps$. Since $\dOeps = \dO \cup \partial \incleps$, we need to consider layer potentials integrated on $\dO$ and on $\partial \incleps$. In this section,  we will investigate some properties of the single layer potential supported on the boundary of the set $\dom$ which  satisfies the assumptions \eqref{Omega}, \eqref{DOmega}, and \eqref{D+Omega}.

First of all, as one can easily see, the single layer potential $v_{{\G}}[\dO,\phi]$ does not depend on the values of the density $\phi$ on $\partial_0 \dom$. In other words, it takes into account only $\phi_{|\dpO}$. For this reason, it is convenient to introduce a quotient Banach space.

\begin{defn}\label{def:C+0}
We denote by $\Ca{0}_+(\dO)$ the quotient Banach space
\[
\Ca{0}(\dO)/\{\phi\in \Ca{0}(\dO)\;|\; \phi_{|\dpO}=0\}\,.
\]
\end{defn}

Then we can prove that the  single layer potential  map
$$\begin{array}{rcl}
\Ca{0}_+(\dO) &\to& \Ca{1}(\dpO)\\ \phi &\mapsto &v_{{\G}}[\dO,\phi]_{|\dpO}\end{array}$$
is well defined and  one-to-one. Namely we have the following.

\begin{prop}[Null space of the single layer potential derived by $G$]\label{one-to-one}
Let $\phi\in \Ca{0}(\dO)$. Then $v_{{\G}}[\dO,\phi]_{|\dpO}=0$ if and only if $\phi_{|\dpO}=0$.
\end{prop}
\begin{proof}  Let $\phi\in \Ca{0}(\dO)$ be such that $\phi_{|\dpO}=0$. As a consequence,
\[
v_{{\G}}[\dO, \phi](\px)=\int_{\dpO}{\G}(\px,\py)\phi_{|\dpO}(\py)\, d\sigma_\py=0 \qquad\forall \px\in\dO\, .
\]
Let now assume that $v_{{\G}}[\dO,\phi]_{|\dpO}=0$. With \eqref{eq.vG+dO}, we have in particular $v_{{\G}}[\dO,\phi]_{|\d0O}=0$ and then $v_{{\G}}[\dO,\phi]_{|\dO}=0$. By the uniqueness of the solution of the Dirichlet problem we deduce that $v_{{\G}}[\dO, \phi]=0$ in $\overline{ \dom}$. By the harmonicity at infinity of $v_{{\G}}[\dO, \phi]$ (cf.~Lemma \ref{vGatinfty}), by equality  $v_{{\G}}[\dO, \phi]_{|\partial \Rnp \cup \dO}=0$,
 and by a standard energy argument based on the divergence theorem, we deduce that
$\nabla v_{{\G}}[\dO, \phi]=0$ in $\Rnp \setminus \overline{\dom}$,
and that accordingly $ v_{{\G}}[\dO, \phi]$ is constant in $\Rnp \setminus\dom$. Since
$v_{{\G}}[\dO, \phi]=0$ on $\dO$, we have
$v_{{\G}}[\dO, \phi]=0$ in $\Rnp$. Then, for the normal derivative of $ v_{{\G}}[\dO, \phi]$  on $\dpO$ we have the following jump formulas:
\[
\bn_{\dom}(\px)\cdot\nabla v_{{\G}}^\sharp[\dO,\phi](\px)=-\frac{\ssei}{2}\phi(\px)+\int_{\dpO}\phi(\py)\bn_{\dom}(\px)\cdot\nabla_\px {\G}(\px,\py)\,d\sigma_\py,\qquad\forall \px\in\dpO\,,
\]
with $\sharp = i,e$, $\ssi=1$, $\sse=-1$. It follows that
\[
\phi= \bn_\dom\cdot\nabla v^e_{{\G}}[\dO,\phi]-\bn_\dom\cdot\nabla v^i_{{\G}}[\dO,\phi]=0 \qquad \text{on $\dpO$}\, ,
\]
 and thus the proof is complete.
\end{proof}

\vspace{\baselineskip}

\medskip

By the previous Proposition \ref{one-to-one} one readily verifies the validity of the following Proposition \ref{V} where we introduce the image space $\sV^{1,\alpha}(\dpO)$ of $v_{{\G}}[\dO,\cdot]_{|\dpO}$.

\begin{prop}[Image of the single layer potential derived by $G$]\label{V}
Let $\sV^{1,\alpha}(\dpO)$ denote the vector space
\[
\sV^{1,\alpha}(\dpO)=\left\{v_{{\G}}[\dO,\phi]_{|\dpO},\ \forall\phi\in \Ca{0}_+(\dO)\right\}.
\]
Let $\|\cdot\|_{\sV^{1,\alpha}(\dpO)}$ be the norm on $\sV^{1,\alpha}(\dpO)$ defined by
\[
\|f\|_{\sV^{1,\alpha}(\dpO)}\equiv\|\phi\|_{\Ca{0}_+(\dO)}
\]
for all $(f,\phi)\in \sV^{1,\alpha}(\dpO) \times \Ca{0}_+(\dO)$ such that $f=v_{{\G}}[\dO,\phi]_{|\dpO}$. Then the following statements hold.
\begin{enumerate}
\item[(i)] $\sV^{1,\alpha}(\dpO)$ endowed with the norm $\|\cdot\|_{\sV^{1,\alpha}(\dpO)}$ is a Banach space.
\item[(ii)] The operator $v_{{\G}}[\dO,\cdot]_{|\dO}$ is an homeomorphism from $\Ca{0}_+(\dO)$ to $\sV^{1,\alpha}(\dpO)$.
\end{enumerate}
\end{prop}

\subsubsection{Characterization of the image of the single layer potential}

 We wish now to characterize the functions of $\sV^{1,\alpha}(\dpO)$,  that is the set of the elements of $\Ca{1}(\dO)$ that can be represented as  $v_{{\G}}[\dO,\phi]_{|\dpO}$  for some $\phi\in \Ca{0}_+(\dO)$. We do so in the following Proposition \ref{f=ue}.

\begin{prop}\label{f=ue}
Let $f\in \Ca{1}(\dpO)$.  Then $f$ belongs to $\sV^{1,\alpha}(\dpO)$ if and only if
$ f=u^e_{|\dpO}$,
where $u^e$ is a function of $\Ca{1}_{\mathrm{loc}}(\overline{\Rnp\setminus\dom})$ such that
\begin{equation}\label{eq:f=ue}
\begin{cases}
\Delta u^e=0\quad\mbox{ in }\Rnp\setminus\overline{\dom},\\
u^e=0 \quad\mbox{ on }\partial\Rnp\setminus\d0O,\\
\lim_{\px \to \infty} \frac{1}{|\px|}u^{e}(\px)=0,\\
\lim_{\px \to \infty} \frac{\px}{|\px|}\cdot \nabla u^{e}(\px)=0.
\end{cases}
\end{equation}
\end{prop}

\medskip

\begin{proofof}{Proposition \ref{f=ue}} We divide the proof in three steps.

\medskip

\noindent {\it $\bullet$ First step: Green-like representation formulas in $\Rnp\setminus\overline{\dom}$.} As a first step, we prove  a representation  formula for harmonic functions in the set $\Rnp\setminus\overline{\dom}$. 

\begin{lem}\label{e-green}
Let $u^e\in \Ca{1}_{\mathrm{loc}}(\overline{\Rnp\setminus\dom})$ be such that
$$\begin{cases}
\Delta u^e=0\quad\mbox{ in }\Rnp\setminus\overline{\dom},\\
\lim_{|\px| \to \infty} \frac{1}{|\px|} u^{e}(\px)=0,\\
\lim_{|\px| \to \infty} \frac{\px}{|\px|}\cdot \nabla u^{e}(\px)=0.
\end{cases}$$
Then we have
\begin{multline*}
-w_{{\G}}[\dpO,u^e_{|\dpO}](\px)+v_{{\G}}[\dO,\bn_{\dom}\cdot\nabla u^e_{|\dO}](\px)+\frac{2x_n}{s_n}\int_{\partial\Rnp \setminus\d0O} \frac{u^e(\py)}{|\px-\py|^n}\, d\sigma_\py
\\
=\begin{cases}
u^e(\px) & \forall \px\in\Rnp\setminus\overline{\dom},\\
0 &  \forall \px\in\dom.
\end{cases}
\end{multline*}
\end{lem}
\begin{proof}  Let $R>\max_{\px\in\overline{\dom}}|\px|$. Let $\dom^{e,+}_R\equiv \Rnp\cap\Bn{R}\setminus\overline{\dom}$. Let $\px\in \dom^{e,+}_R$. Let $r>0$ and $\cB(\px,r)\subseteq \dom^{e,+}_R$. By Lemma \ref{i-green} we have
\begin{equation}\label{eq1}
\begin{split}
u^e(\px)& =w_{{\G}}[\partial\cB(\px,r),u^e_{|\partial\cB(\px,r)}](\px)-v_{{\G}}[\partial\cB(\px,r),\bn_{\cB(\px,r)}\cdot\nabla u^e_{|\partial\cB(\px,r)}](\px)\\
& =\int_{\partial\cB(\px,r)} u^e(\py)\; \bn_{\cB(\px,r)}(\py)\cdot\nabla_\py {\G}(\px,\py)-{\G}(\px,\py)\; \bn_{\cB(\px,r)}(\py)\cdot\nabla u^e(\py)\; d\sigma_\py\,.
\end{split}
\end{equation}
Then we observe that ${\G}(\px,\cdot)$ is a harmonic function in $\dom^{e,+}_R\setminus\overline{\cB(\px,r)}$ and thus by the divergence theorem we have
\[
\begin{split}
0&=\int_{\dom^{e,+}_R\setminus\overline{\cB(\px,r)}}u^e(\py)\Delta_\py{\G}(\px,\py)-{\G}(\px,\py)\Delta u^e(\py)\, d\px\\
&=-\int_{\partial\cB(\px,r)} u^e(\py)\; \bn_{\cB(\px,r)}(\py)\cdot\nabla_\py {\G}(\px,\py) -{\G}(\px,\py)\; \bn_{\cB(\px,r)}(\py)\cdot\nabla u^e(\py)\; d\sigma_\py\\
&\quad+\int_{\dO^{e,+}_R} u^e(\py)\; \bn_{\dom^{e,+}_R}(\py)\cdot\nabla_\py {\G}(\px,\py) -{\G}(\px,\py)\; \bn_{\dom^{e,+}_R}(\py)\cdot\nabla u^e(\py)\; d\sigma_\py\\
&=-\int_{\partial\cB(\px,r)} u^e(\py)\; \bn_{\cB(\px,r)}(\py)\cdot\nabla_\py {\G}(\px,\py)-{\G}(\px,\py)\; \bn_{\cB(\px,r)}(\py)\cdot\nabla u^e(\py)\; d\sigma_\py\\
&\quad-\int_{\dpO} u^e(\py)\; \bn_{\dom}(\py)\cdot\nabla_\py {\G}(\px,\py) -{\G}(\px,\py)\; \bn_{\dom}(\py)\cdot\nabla u^e(\py)\; d\sigma_\py\\
&\quad+\int_{\partial_+\Bn{R}} u^e(\py)\; \bn_{\Bn{R}}(\py)\cdot\nabla_\py {\G}(\px,\py)-{\G}(\px,\py)\; \bn_{\Bn{R}}(\py)\cdot\nabla u^e(\py)\; d\sigma_\py\\
&\quad-\int_{\partial\Rnp\cap\Bn{R}\setminus\d0O} u^e(\py)\; \partial_{y_n} {\G}(\px,\py)-{\G}(\px,\py)\; \partial_{y_n} u^e(\py)\; d\sigma_\py.
\end{split}\]
Using Definition~\ref{def.SDG+} and the fact that ${\G}(\px,\py)=0$ and $\partial_{y_n}{\G}(\px,\py)=-2x_ns_n^{-1}|\px-\py|^{-n}$ for all $\py\in\partial\Rnp$, we deduce
\[
\begin{split}
0&=-\int_{\partial\cB(\px,r)} u^e(\py)\; \bn_{\cB(\px,r)}(\py)\cdot\nabla_\py {\G}(\px,\py)-{\G}(\px,\py)\; \bn_{\cB(\px,r)}(\py)\cdot\nabla u^e(\py)\; d\sigma_\py\\
&\quad-w_{{\G}}[\dpO,u^e_{|\dpO}](\px)+v_{{\G}}[\dpO,\bn_{\dom}\cdot\nabla u^e_{|\dpO}](\px)\\
&\quad+\int_{\partial_+\Bn{R}} u^e(\py)\; \bn_{\Bn{R}}(\py)\cdot\nabla_\py {\G}(\px,\py)-{\G}(\px,\py)\; \bn_{\Bn{R}}(\py)\cdot\nabla u^e(\py)\; d\sigma_\py\\
&\quad+\frac{2x_n}{s_n}\int_{\partial\Rnp\cap\Bn{R}\setminus\d0O} \frac{u^e(\py)}{|\px-\py|^n}\, d\sigma_\py.
\end{split}
\]
Then we observe that the maps $\py\mapsto |\py|^{n-1} {\G}(\px,\py)$ and
$\py\mapsto |\py|^n \nabla_\py {\G}(\px,\py)$ are bounded at infinity (see Lemma \ref{Gatinfty}).
 Thus, by taking the limit as $R\to\infty$ we obtain
\begin{equation}\label{eq2}
\begin{split}
0&=-\int_{\partial\cB(x,r)} u^e(\py)\; \bn_{\cB(\px,r)}(\py)\cdot\nabla_\py {\G}(\px,\py)-{\G}(\px,\py)\; \bn_{\cB(\px,r)}(\py)\cdot\nabla u^e(\py)\; d\sigma_\py\\
&\quad-w_{{\G}}[\dpO,u^e_{|\dpO}](\px)+v_{{\G}}[\dO,\bn_{\dom}\cdot\nabla u^e_{|\dO}](\px)+\frac{2x_n}{s_n}\int_{\partial\Rnp\setminus\d0O} \frac{u^e(\py)}{|\px-\py|^n}\, d\sigma_\py\,.
\end{split}
\end{equation}
Then by summing \eqref{eq1} and \eqref{eq2} we show the validity of the first equality in the statement. The proof of the second equality is similar and accordingly omitted.
\end{proof}

\vspace{\baselineskip}

\medskip

Incidentally, we observe that under the assumptions of Lemma \ref{e-green} the integral
\[
\int_{\partial\Rnp \setminus\d0O} \frac{u^e(\py)}{|\px-\py|^n}\, d\sigma_\py
\]
exists finite for all $\px \in \mathbb{R}_+^n \setminus \dO$.

\medskip

\noindent{\it $\bullet$ Second step: representation in terms of single layer potentials plus an extra term}. In the following Proposition \ref{representation}, we introduce a representation formula for a suitable family of functions of $\Ca{1}(\dO)$. More precisely, we show that the restriction to $\dpO$ of a function $f \in \Ca{1}(\dO)$ which satisfies certain assumptions can be written as the sum of a single layer potential with kernel $G$ plus an extra term.

\begin{prop}\label{representation}
Let $f\in \Ca{1}(\dO)$ with $f_{|\d0O}=0$. Assume that there exists a function $u^e\in \Ca{1}_{\mathrm{loc}}(\overline{\Rnp\setminus\dom})$ such that
$$\begin{cases}
\Delta u^e=0\quad\mbox{ in }\Rnp\setminus\overline{\dom},\\
u^e=f\quad\mbox{ on }{\dpO},\\
\lim_{\px \to \infty} \frac{1}{|\px|}u^{e}(\px)=0,\\
\lim_{\px \to \infty} \frac{\px}{|\px|}\cdot \nabla u^{e}(\px)=0.
\end{cases}
$$
Then there exists $\phi\in \Ca{0}(\dO)$ such that
\begin{equation}\label{representation.eq0}
v_{{\G}}[\dO,\phi](\px)+\frac{2x_n}{s_n}\int_{\partial\Rnp\setminus\d0O} \frac{u^e(\py)}{|\px-\py|^n}\, d\sigma_\py=f(\px)\,,\qquad\forall \px\in\dpO\,.
\end{equation}
\end{prop}
\begin{proof}  Let $u^i\in \Ca{1}(\overline{\dom})$ be the solution of the Dirichlet problem with boundary datum $f$. By Lemma \ref{i-green} we have
\[
0=w_{{\G}}[\dO,u^i_{|\dO}](\px)-v_{{\G}}[\dO,\bn_{\dom}\cdot\nabla u^i_{|\dO}](\px)\,,\qquad\forall \px\in\Rnp\setminus\overline{\dom}\,.
\]
Since $u^i_{|\d0O}=f_{|\d0O}=0$ we deduce that
\begin{equation}\label{representation.eq1}
0=w_{{\G}}[\dpO,f_{|\dpO}](\px)-v_{{\G}}[\dO,\bn_{\dom}\cdot\nabla u^i_{|\dO}](\px)\,,\qquad\forall \px\in\Rnp\setminus\overline{\dom}\,.
\end{equation}
By Lemma \ref{e-green} we have
\begin{equation}\label{representation.eq2}
u^e(\px)=-w_{{\G}}[\dpO,f_{|\dpO}](\px)+v_{{\G}}[\dO,\bn_{\dom}\cdot\nabla u^e_{|\dO}](\px)+\frac{2x_n}{s_n}\int_{\partial\Rnp \setminus\d0O} \frac{u^e(\py)}{|\px-\py|^n}\, d\sigma_\py
\end{equation}
 for all $\px\in\Rnp\setminus\overline{\dom}$. Then by taking the sum of \eqref{representation.eq1} and \eqref{representation.eq2} and by the continuity properties of the (Green) single layer potential one verifies that the proposition holds with
\[
\phi=\bn_{\dom}\cdot\nabla u^e_{|\dO}-\bn_{\dom}\cdot\nabla u^i_{|\dO}\,.
\]
\end{proof}

\noindent{\it $\bullet$ Last step: vanishing of the extra term in \eqref{representation.eq0}.} In order to understand what can be represented just by means of  the single layer potential, the final step is to understand when such an extra term vanishes. So let  $f\in \Ca{1}(\dpO)$ be such that
$ f=u^e_{|\dpO}$,
where $u^e$ is a function of $\Ca{1}_{\mathrm{loc}}(\overline{\Rnp\setminus\dom})$ such that \eqref{eq:f=ue} holds. Then
\[
\frac{2x_n}{s_n}\int_{\partial\Rnp\setminus\d0O} \frac{u^e(\py)}{|\px-\py|^n}\, d\sigma_\py=0\,, \qquad\forall \px\in\dpO\, ,
\]
and thus \eqref{representation.eq0} implies that $f \in \sV^{1,\alpha}(\dpO)$. Conversely, if $f \in \sV^{1,\alpha}(\dpO)$ then there exists $\phi\in \Ca{0}_+(\dO)$ such that $f=v_{{\G}}[\dO,\phi]_{|\dpO}$ and the function  $u^e \equiv v_{{\G}}[\dO,\phi]_{|\overline{\Rnp\setminus\dom}}$ satisifies \eqref{eq:f=ue}. This concludes the proof of Proposition \ref{f=ue}.
\end{proofof}

\vspace{\baselineskip}

\medskip

Now that Proposition \ref{f=ue} is proved,  we observe that if $u^e$ is as in Proposition \ref{representation}, then
\begin{equation}\label{extra_lim}
\lim_{t\to 0^+}\frac{2t}{s_n}\int_{\partial\Rnp \setminus\d0O} \frac{u^e(\py)}{|\px+t \se_n-\py|^n}\, d\sigma_\py=u^e(\px),\qquad\forall \px\in\partial\Rnp\setminus\d0O\,,
\end{equation}
where $\se_n$ denotes the vector $(0,\dots,0,1)\in\Rn$.  The limit in \eqref{extra_lim} can be computed  by exploiting  known results in potential theory (see~Cialdea \cite[Thm.~1]{Ci95}).
A consequence of \eqref{extra_lim} is that the second term in the left hand side of \eqref{representation.eq0} vanishes on $\dpO$ only if $u^e_{|\partial\Rnp\setminus\d0O}=0$. Namely, we have the following

\begin{prop}\label{extra_prop}
Let $u^e$ be as in Proposition \ref{representation}.  Then we have
\begin{equation}\label{extra_eq1}
\frac{2x_n}{s_n}\int_{\partial\Rnp\setminus\d0O} \frac{u^e(\py)}{|\px-\py|^n}\, d\sigma_\py=0,\qquad\forall \px\in\dpO
\end{equation}
 if and only if
 \begin{equation}\label{extra_eq2}
 u^e_{|\partial\Rnp\setminus\d0O}=0\,.
 \end{equation}
\end{prop}
\begin{proof} 
One immediately verifies that \eqref{extra_eq2} implies \eqref{extra_eq1}. To prove that  \eqref{extra_eq1} implies \eqref{extra_eq2}, we denote by $U^+$ the function of $\px\in\Rn\setminus(\partial\Rnp\setminus\d0O)$ defined by the left hand side of \eqref{extra_eq1}. Then, we observe that, by the properties of integral operators with real analytic kernel and no singularity, $U^+$ is harmonic in  $\Rn\setminus(\partial\Rnp\setminus\d0O)$ and vanishes on $\d0O$. Thus, \eqref{extra_eq1} implies that $U^+=0$ on the whole of $\dO$ and by  the uniqueness of the solution of the Dirichlet problem we have that $U^+=0$ on $\dom$. By the identity principle for analytic functions it  follows that $U^+=0$ on $\Rn\setminus(\partial\Rnp\setminus\d0O)$ and thus, by \eqref{extra_lim}, we have
\[
u^e(\px)=\lim_{t\to 0^+}U^+(\px+t \se_n)=0\qquad\forall \px\in\partial\Rnp\setminus\d0O.
\]
\end{proof}

\vspace{\baselineskip}

In Remark \ref{extra_rem} here below we observe that a function $u^e$ which satisfies the conditions of Proposition \ref{representation} actually exists and that the second term in the left hand side of \eqref{representation.eq0} cannot be in general omitted.

\begin{rmk}\label{extra_rem} Let $f\in \Ca{1}(\dO)$ with $f_{|\d0O}=0$ and let $u_\#\in \Ca{1}_{\mathrm{loc}}(\Rn\setminus\dom)$ be the unique solution of the Dirichlet problem in $\Rn\setminus\overline{\dom}$ with boundary datum $f$ which satisfies the decay condition $\lim_{\px\to\infty}u_\#(\px)=0$ if $n \geq 3$ and such that $u_\#$ is bounded if $n=2$ ({\it i.e.}, $u_\#$ is harmonic at $\infty$).  Then the function $u^e_{\#}\equiv u_{\#|\overline{\Rnp\setminus\dom}}$ satisfies the conditions of Proposition \ref{representation}. In addition, $u^e_{\#|\partial\Rnp\setminus\d0O}=0$ only if $f=0$, and thus the corresponding second term in the left hand side of \eqref{representation.eq0} is $0$ only if $f=0$ (cf. Proposition \ref{extra_prop}). The latter fact can be proved by observing that if $u^e_{\#|\partial\Rnp\setminus\d0O}=0$, then $u_{\#|\partial\Rnp\setminus\d0O}=0$ and thus $u_{\#|\partial\Rnp}=0$ (because $u_{\#|\d0O}=f_{|\d0O}=0$ by our assumptions on $f$). Then, by the decay properties of $u_\#$ and by the divergence theorem we have
\[
\int_{\Rnm}|\nabla u_\#|^2\,d\px=\lim_{R\to\infty}\left(\int_{\partial\Bn{R}\cap\Rnm} u\, \bn_{\partial\Bn{R}}\cdot\nabla u\, d\sigma+\int_{\partial\Rnp\cap\Bn{R}} u\, \partial_{x_n} u\, d\sigma\right)=0\,.
\]
It follows that $u_{\#|\Rnm}=0$, which in turn implies that $u_\#=0$ by the identity principle of real analytic functions. Hence $f=u_{\#|\dO}=0$.
\end{rmk}

\subsection{Extending functions from $\Ca{k}(\overline{\dpO})$ to $\Ca{k}(\dO)$}
We will need to pass from functions defined on $\dpO$ to functions defined on $\dO$, and viceversa.
The restriction operator from $\Ca{k}(\dO)$ to $\Ca{k}(\overline{\dpO})$  is linear and continuous for  $k=0,1$.  On the other hand, we have the following extension result.
\begin{lem}
\label{ext}
There exist linear and continuous extension operators $E^{k,\alpha}$ from $\Ca{k}(\overline{\dpO})$ to $\Ca{k}(\dO)$, for $k=0,1$.
\end{lem}
 A proof can be effected by arguing as in Troianiello \cite[proof of Lem.~1.5, p.~16]{Tr87} and by exploiting condition \eqref{D+Omega}.
We observe that as a consequence of Lemma \ref{ext} we can identify $\Ca{0}_+(\dO)$ and $\Ca{0}(\overline{\dpO})$.

\section{Asymptotic behavior of $\ueps$ in dimension $n\geq3$}\label{n>=3}

In this section, we investigate the asymptotic behavior of the solution of problem \eqref{bvpe} as $\be \to \bf 0$. In the whole Section \ref{n>=3}, the dimension $n$ is assumed to be greater than or equal to $3$. Namely,
\[
n\in\mathbb{N}\setminus  \{0,1,2\}\, .
\]

 Our strategy is here to reformulate the problem as an equation $\fL[\be ,\bm]=0$ where $\fL$ is a real analytic function and to use the implicit function theorem.

\subsection{Defining the operator $\fL$}

 Let $\be \in ]\b0,\bea[$.
 We start from the Green-like representation formula  of Lemma \ref{i-green}. By applying it to  the solution $\ueps$ of \eqref{bvpe}, we can write:
 \[
 \begin{split}
 \ueps=&w_{G}^i[\dOeps,u_{\be|\dOeps}]-v_G^i[\dOeps,\bn_{\domeps}\cdot\nabla u_{\be|\dOeps}]\\
 =&w_{G}^i[\dO,\go]-w_{G}^e\Big[\partial\incleps,\gi\big(\frac{\cdot-\e1\p}{\e1\e2}\big)\Big]-v_G^i[\dO,\bn_{\dom}\cdot\nabla u_{\be|\dO}]+v_G^e[\partial \incleps,\bn_{\incleps}\cdot\nabla u_{\be|\partial \incleps}]\,.
 \end{split}
 \]
By adding and subtracting $v_{G}^i[\dO,\bn_{\dom}\cdot \nabla u_{0|\dO}]$ we get
 \begin{equation}\label{eq:uepsrep}
\begin{split}
 \ueps=&w_{G}^i[\dO,\go]-v_{G}^i[\dO,\bn_{\dom}\cdot \nabla u_{0|\dO}]-w_{G}^e\Big[\partial\incleps,\gi\big(\frac{\cdot-\e1\p}{\e1\e2}\big)\Big]\\
 &-v_G^i[\dO,\bn_{\dom}\cdot\nabla u_{\be|\dO}-\bn_{\dom}\cdot \nabla u_{0|\dO}]+v_G^e[\partial \incleps,\bn_{\incleps}\cdot\nabla u_{\be|\partial \incleps}]\, .
 \end{split}
 \end{equation}
 Then we note that
 \[
 \uzero=w_{G}^i[\dO,\go]-v_{G}^i[\dO,\bn_{\dom}\cdot \nabla u_{0|\dO}]
 \]
 and we think to the functions
 \[
 \bn_{\dom}\cdot\nabla u_{\be|\dO}-\bn_{\dom}\cdot \nabla u_{0|\dO}\, , \qquad \bn_{\incleps}\cdot\nabla u_{\be|\partial \incleps}
 \]
as to unknown densities which have to be determined in order to solve problem \eqref{bvpe}.  Accordingly, inspired by \eqref{eq:uepsrep} and by the rule of change of variables in integrals, we look for a solution of problem \eqref{bvpe} in the form
\begin{equation}\label{eq:rep:sol}
\begin{split}
\uzero (\px)-\e1 ^{n-1}\e2 ^{n-1}\int_{\dincl} \bn_{\incl}(\ps)&\cdot (\nabla_\py{\G})(\px,\e1 \p+\e1 \e2 \ps) \gi(\ps)\, d\sigma_\ps-\int_{\dpO}{\G}(\px,\py)\mu_1(\py)\, d\sigma_\py\\
& +\e1 ^{n-2}\e2 ^{n-2}\int_{\dincl}{\G}(\px,\e1 \p +\e1 \e2 \ps)\mu_2(\ps)\, d\sigma_\ps\, , \qquad \px \in \domeps,
\end{split}
\end{equation}
where the pair $(\mu_1,\mu_2)\in \Ca{0}_+(\dO)\times \Ca{0}(\dincl)$ has to be determined. {We set $\bm\equiv(\mu_{1},\mu_{2})$ and $\sB_1\equiv\Ca{0}_+(\dO)\times \Ca{0}(\dincl)$.}
Since the function in \eqref{eq:rep:sol} is harmonic in $\domeps$ for all $\bm\in\sB_{1}$,
we just need to choose $\bm\in\sB_1$  such that  the boundary conditions are satisfied. By the jump properties of the layer potentials derived by $G$, this is equivalent to ask that  $\bm\in\sB_1$ solves
\begin{equation}
\label{eq:int:N0}
\fL[\be ,\bm]=0,
\end{equation}
where   $\fL[\be ,\bm]\equiv (\fL_1[\be ,\bm],\fL_2[\be ,\bm])$ is defined  by
\[
\begin{split}
\fL_1[\be ,\bm](\px)&\equiv v_{{\G}}[\dO,\mu_1](\px)\\
&\quad  -\e1 ^{n-2}\e2 ^{n-2}\int_{\dincl}{\G}(\px,\e1 \p+\e1 \e2 \ps)\, \mu_2(\ps)\, d\sigma_\ps\\
&\quad +\e1 ^{n-1}\e2 ^{n-1}\int_{\dincl}\bn_\incl(\ps)\cdot (\nabla_\py{\G})(\px,\e1 \p+\e1 \e2 \ps)\, \gi(\ps)\, d\sigma_\ps\qquad\forall \px\in\dpO\,,\\
\fL_2[\be,\bm](\pt)&\equiv v_{S_n}[\dincl,\mu_2](\pt)-\e2 ^{n-2}\int_{\dincl} S_n(-2p_n\se_n+\e2 (\msigma(\pt)-\ps))\, \mu_2(\ps)\, d\sigma_\ps\\
&\quad -\int_{\dpO} {\G}(\e1 \p+\e1 \e2 \pt, \py)\, \mu_1(\py)\, d\sigma_\py\\
&\quad  -\e2 ^{n-1}\int_{\dincl}\bn_\incl(\ps)\cdot\nabla S_n(-2p_n \se_n+\e2 (\msigma(\pt)-\ps))\, \gi(\ps)\, d\sigma_\ps\\
&\quad +U_0  (\e1 \p+\e1 \e2 \pt)-w_{S_n}[\dincl,\gi](\pt)-\frac{\gi(\pt)}2\qquad \qquad\qquad\qquad\forall \pt\in\dincl\,,
\end{split}
\]
with $p_n$ as in \eqref{p} (note that $p_n>0$ by the membership of $\p$ in $\Rnp$) and
$U_0$ as in Proposition \ref{Ub}.

\subsection{Real analyticity of the operator $\fL$}

By the equivalence of the boundary value problem \eqref{bvpe} and the functional equation \eqref{eq:int:N0}, we can deduce results for the map $\be\mapsto\ueps$ by studying the dependence of $\bm$ upon $\be$ in \eqref{eq:int:N0}. To do so, we plan to apply the implicit function theorem for real analytic maps  and, as a first step, we wish to prove that the operator $\fL$ is real analytic.

\begin{prop}[Real analyticity of $\fL$]\label{prop.N}
The map
\[
\begin{array}{rcl}
 \bJea\times \sB_1& \to& \sV^{1,\alpha}(\dpO)\times \Ca{1}(\dincl)\\
 (\be,\bm)& \mapsto & \fL[\be,\bm]
 \end{array}
 \]
 is real analytic.
\end{prop}
\begin{proof}  We split the proof component by component.

\paragraph{Study of $\fL_{1}$} Here we prove that $\fL_{1}$ is real analytic from $\bJea\times \sB_1$ to $\sV^{1,\alpha}(\dpO)$.

\noindent\emph{First step: the range of $\fL_1$ is a subset of $\sV^{1,\alpha}(\dpO)$.}  Let  $U^e[\be ,\bm]$ denote the function from $\overline{\Rnp\setminus\dom}$ to $\mathbb{R}$ defined by
\[
\begin{split}
U^e[\be ,\bm](\px)&\equiv v^e_{{\G}}[\dO,\mu_1](\px)\\
&\quad  - \e1 ^{n-2}\e2 ^{n-2}\int_{\dincl}{\G}(\px,\e1 \p+\e1 \e2 \ps)\, \mu_2(\ps)\, d\sigma_\ps\\
&\quad +\e1 ^{n-1}\e2 ^{n-1}\int_{\dincl}\bn_\incl(\ps)\cdot(\nabla_\py {\G})(\px,\e1 \p+\e1 \e2 \ps)\, \gi(\ps)\, d\sigma_\ps\,.\\
\end{split}
\]
Then, by the properties of the (Green) single layer potential and by the properties of integral operators with real analytic kernel and no singularity one verifies that $U^e[\be,{\bm}]\in \Ca{1}_{\mathrm{loc}}(\overline{\Rnp\setminus\dom})$. In addition, one has
$$
\left\{
\begin{array}{ll}
 \Delta U^e[\be,\bm]=0& \text{in }\Rnp\setminus\overline{\dom} ,\\
 U^e[\be,\bm]=0& \text{on }{\partial\Rnp\setminus\d0O}, \\
 \lim_{\px\to\infty}U^e[\be,\bm](\px)=0,\\
\lim_{\px\to\infty}\frac{\px}{|\px|}\cdot\nabla U^e[\be,\bm](\px)=0
 \end{array}
 \right.$$
Thus $U^e[\be,\bm]$ satisfies the conditions of Proposition \ref{f=ue}. Accordingly, we conclude that
$\fL_1[\be,\bm]=U^e[\be,\bm]_{|\dpO}$ belongs to $\sV^{1,\alpha}(\dpO)$.

\medskip
\noindent\emph{Second step: $\fL_1$ is real analytic.}  We decompose $\fL_{1}$ and study each part separately.
\begin{itemize}
\item By the definition of $\sV^{1,\alpha}(\dpO)$ in Proposition \ref{V}, one readily verifies that the map $\mu_1 \mapsto v_{{\G}}[\dO,\mu_1]_{|\dpO}$ is linear and continuous from $\Ca{0}_+(\dO)$ to $\sV^{1,\alpha}(\dpO)$ and therefore real analytic.
\item We now consider the map which takes $(\be ,\mu_2)$ to the function $\mathfrak{f}[\be,\mu_2](\px)$ of $\px\in\overline{\dpO}$ defined by
\[
\mathfrak{f}[\be,\mu_2](\px)\equiv\int_{\dincl}{\G}(\px,\e1 \p+\e1 \e2 \ps)\, \mu_2(\ps)\, d\sigma_\ps\qquad\forall \px\in\overline{\dpO}.
\]
We wish to prove that $\mathfrak{f}$ is real analytic from $\bJea\times \Ca{0}(\dO)$ to $\sV^{1,\alpha}(\dpO)$ by showing that there is a real analytic function
\[
\phi:\bJea\times \Ca{0}(\dO) \to \Ca{0}_+(\dO)
\]
such that
\begin{equation}\label{N1.eq02}
\mathfrak{f}[\be,\mu_2]=v_{{\G}}[\dO,\phi[\be,\mu_2]]_{|\dpO}
\end{equation}
for all $(\be,\mu_2)\in\bJea\times \Ca{0}(\dO)$. Then the real analyticity of  $\mathfrak{f}$ follows by the definition of $\sV^{1,\alpha}(\dpO)$ in Proposition \ref{V}.

We will obtain $\phi$ as the sum of two real analytic terms. To find the first one we observe that, by the properties of integral operators with real analytic kernel and no singularity, the map $(\be,\mu_2)\mapsto \mathfrak{f}[\be,\mu_2]$ is real analytic from $\bJea\times \Ca{0}(\dincl)$ to $\Ca{1}(\overline{\dpO})$ (see Lanza de Cristoforis and Musolino \cite[Prop.~4.1 (ii)]{LaMu13}). Then, by the extension Lemma \ref{ext}, we deduce that the composed map
\[
\begin{array}{rcl}
\bJea\times \Ca{0}(\dO)& \to &\Ca{1}(\dO)\\
(\be,\mu_2) &\mapsto& E^{1,\alpha}\circ \mathfrak{f}[\be,\mu_2]\end{array}
\]
is real analytic.
Let now $u^i[\be,\mu_2]$ denote the unique solution of the Dirichlet problem for the Laplace equation in $\dom$ with boundary datum $E^{1,\alpha}\circ \mathfrak{f}[\be,\mu_2]$. As is well-known, the map from $\Ca{1}(\dO)$ to $\Ca{1}(\overline{\dom})$ which takes a function $\psi$ to the unique solution of the Dirichlet problem for the Laplace equation in $\dom$ with boundary datum $\psi$ is linear and continuous. It follows that the map from $\bJea\times \Ca{0}(\dincl)$ to $\Ca{1}(\overline{\dom})$ which takes $(\be,\mu_2)$ to $u^i[\be,\mu_2]$ is real analytic. Thus the map
\begin{equation}\label{N1.eq1}
\begin{array}{rcl}
\bJea\times \Ca{0}(\dincl)& \to &\Ca{0}_+(\dO)\\
(\be,\mu_2)& \mapsto &\bn_\dom\cdot\nabla u^i[\be,\mu_2]_{|\dO}
\end{array}
\end{equation}
is real analytic.

The function in \eqref{N1.eq1} is the first term in the sum that gives $\phi$. To obtain the second term we define
\[
u^e[\be ,\mu_2](\px)\equiv \int_{\dincl}{\G}(\px,\e1 \p+\e1 \e2 \ps)\, \mu_2(\ps)\, d\sigma_\ps\,,\qquad\forall \px\in\overline{\Rnp\setminus\dom}\,.
\]
Then, by standard properties of integral operators with real analytic kernels and no singularity one verifies that the map from $\bJea\times \Ca{0}(\dincl)$ to $\Ca{0}(\overline{\dpO})$ which takes $(\be,\mu_2)$ to
\[
\bn_\dom\cdot\nabla u^e[\be,\mu_2](\px)=\bn_\dom(\px)\cdot\int_{\dincl}\nabla_\px {\G}(\px,\e1 \p+\e1 \e2 \ps)\, \mu_2(\ps)\, d\sigma_\ps\qquad\forall \px\in\overline{\dpO}
\]
 is real analytic (see Lanza de Cristoforis and Musolino \cite[Prop.~4.1 (ii)]{LaMu13}). Thus, by the extension Lemma \ref{ext}, we can show that the map
\begin{equation}\label{N1.eq2}
\begin{array}{rcl}
\bJea\times \Ca{0}(\dincl)& \to &\Ca{0}_+(\dO)\\
(\be,\mu_2)& \mapsto &\bn_\dom\cdot\nabla u^e[\be,\mu_2]_{|\dO}
\end{array}
\end{equation}
is real analytic.

We are now ready to show that  $\phi$  is given by the difference of the function in \eqref{N1.eq2} and the one in \eqref{N1.eq1}. To do so, we begin by observing that $u^i[\be,\mu_2]_{|\dpO}=\mathfrak{f}[\be,\mu_2]$. Then, by Lemma \ref{i-green} we have
\begin{equation}\label{representation.eq13D}
0=w_{{\G}}[\dpO,\mathfrak{f}[\be ,\mu_2]](\px)-v_{{\G}}[\dpO,\bn_{\dom}\cdot\nabla u^i[\be ,\mu_2]_{|\dpO}](\px)
\end{equation}
for all $ \px\in\R^n_+\setminus\overline{\dom}$. Moreover,  the function  $u^e[\be,\mu_2]$ belongs to  $\Ca{1}_{\mathrm{loc}}(\overline{\mathbb{R}^n_+\setminus\dom})$ and one verifies that
\begin{equation}\label{L1n=3.eq3}
\left\{
\begin{array}{ll}
\Delta u^e[\be,\mu_2]=0& \text{ in } \mathbb{R}^n_{+}\setminus\overline{\dom} ,\\
u^e[\be,\mu_2](\px)=0 &\text{ on }\partial\mathbb{R}^n_{+}\setminus\d0O, \\
\lim_{\px\to\infty}u^e[\be,\mu_2](\px)=0\,,&\\
\lim_{\px \to \infty} \frac{\px}{|\px|}\cdot \nabla u^{e}[\be,\mu_2](\px)=0\,&
 \end{array}
 \right.
 \end{equation}
 (see Lemma \ref{vGatinfty}). In addition, by the definitions of $\mathfrak{f}[\be,\mu_2]$ and  $u^e[\be,\mu_2]$ one sees that 
 \begin{equation}\label{L1n=3.eq4}
 u^e[\be,\mu_2]_{|\overline{\dpO}}=\mathfrak{f}[\be,\mu_2]\,.
 \end{equation}
 Then by  \eqref{L1n=3.eq3} and by  Lemma \ref{e-green}  we deduce that
 \begin{equation}\label{representation.eq23D}
u^e[\be,\mu_2](\px)=-w_{{\G}}[\dpO,\mathfrak{f}[\be,\mu_2]](\px)+v_{{\G}}[\dpO,\bn_{\dom}\cdot\nabla u^e[\be,\mu_2]_{|\dpO}](\px)
\end{equation}
for all $ \px\in\R^n_+\setminus\overline{\dom}$.  Now, taking the sum of \eqref{representation.eq13D} with \eqref{representation.eq23D} we obtain
\[
u^e[\be,\mu_2](\px)=v_{{\G}}[\dpO,\bn_{\dom}\cdot\nabla u^e[\be,\mu_2]_{|\dpO}](\px)-v_{{\G}}[\dpO,\bn_{\dom}\cdot\nabla u^i[\be ,\mu_2]_{|\dpO}](\px)
\]
for all $ \px\in\R^n_+\setminus\overline{\dom}$.  Then, by \eqref{L1n=3.eq4} and 
by the continuity properties of the (Green) single layer potential in $\Rn$ we get 
\[
\mathfrak{f}[\be,\mu_2](\px)=v_{{\G}}[\dpO,\bn_{\dom}\cdot\nabla u^e[\be,\mu_2]_{|\dpO}](\px)-v_{{\G}}[\dpO,\bn_{\dom}\cdot\nabla u^i[\be ,\mu_2]_{|\dpO}](\px)
\]
 for all $ \px\in\overline{\dpO}$. Hence, \eqref{N1.eq02} holds with
\[
\phi[\be,\mu_2]=\bn_\dom\cdot\nabla u^e[\be,\mu_2]_{|\dpO}-\bn_\dom\cdot\nabla u^i[\be,\mu_2]_{|\dpO}\,.
\]
To show that ${{\fL}_1}$ is real analytic it remains to observe that, since the maps in \eqref{N1.eq1} and \eqref{N1.eq2} are both real analytic, $\phi$ is real analytic from $\bJea\times \Ca{0}(\dO)$ to $\Ca{0}_+(\dO)$.

\item Finally, we have to consider the function which takes $\be$ to the function $\mathfrak{g}[\be]$ defined on $\overline{\dpO}$ by
\[
\mathfrak{g}[\be](\px)\equiv\int_{\dincl}\bn_\incl(\ps)\cdot(\nabla_\py {\G})(\px,\e1 \p+\e1 \e2 \ps)\, \gi(\ps)\, d\sigma_\ps\qquad\forall \px\in\overline{\dpO}\,.
\]
By arguing as we have done above for $\mathfrak{f}[\be,\mu_2]$, we can verify that the map $\be\mapsto \mathfrak{g}[\be]$ is real analytic from $\bJea$ to $\sV^{1,\alpha}(\dpO)$ .
\end{itemize}
This proves the analyticity of $\fL_{1}$.
\paragraph{Study of $\fL_{2}$} The analyticity of $\fL_{2}$ from $\bJea\times \sB_1$ to $\Ca{1}(\dincl)$
is a consequence of: 
 \begin{itemize}
 \item The real analyticity of $U_0$  (see also assumption \eqref{Ubassumption});
 \item The mapping properties of the single layer potential (see Lanza de Cristoforis and Rossi \cite[Thm.~3.1]{LaRo04} and Miranda \cite{Mi65}) and of the integral operators with real analytic kernels and no singularity (see Lanza de Cristoforis and Musolino \cite[Prop.~4.1 (ii)]{LaMu13}).
 \end{itemize}
\end{proof}

\vspace{\baselineskip}

\subsection{Functional analytic representation theorems}

To investigate problem \eqref{bvpe} for  $\be$ close to $\bf 0$, we consider in the following Proposition \ref{00} the equation in \eqref{eq:int:N0} for $\be=\bf 0$.

\begin{prop}\label{00}
There exists a unique pair of functions $\bm^*\equiv(\mu^*_{1},\mu^*_{2})\in \sB_1$ such that
\[
\fL[{\bf 0},{\bm^*}]=0,
\]
and we have
\[
\mu^*_{1}=0\qquad\mbox{ and }\qquad v_{S_n}[\dincl,\mu^*_{2}]_{|\dincl}=-\go(0) + w_{S_n}[\dincl,\gi]_{|\dincl} + \frac{\gi}2\,.
\]
\end{prop}
\begin{proof}  First of all, we observe that for all $\bm\in\sB_1$,
we have
\[\begin{cases}
\fL_1[{\bf 0},\bm](\px)= v_{{\G}}[\dO,\mu_1](\px),&\forall \px\in\dpO\,,\\
\fL_2[{\bf 0},\bm](\pt)= v_{S_n}[\dincl,\mu_2](\pt)  +  \go (0)-w_{S_n}[\dincl,\gi](\pt) - \frac{\gi(\pt)}2 ,&\forall \pt\in\dincl\,.
\end{cases}
\]
By Proposition \ref{one-to-one}, the unique function in $\Ca{0}_+(\dO)$ such that $v_{{\G}}[\dO,\mu_1]=0$ on $\dpO$ is $\mu_1=0$. On the other hand, by classical potential theory {and Lemma \ref{viso}}, there exists a unique function $\mu_2 \in \Ca{0}(\dincl)$ such that
\[
v_{S_n}[\dincl,\mu_2](\pt)=-\go(0)+w_{S_n}[\dincl,\gi](\pt)+\frac{\gi(\pt)}2 \qquad\forall \pt\in\dincl.
\]
The validity of the proposition is proved.
\end{proof}

\vspace{\baselineskip}

\medskip

We are now ready to study the dependence of the solution of \eqref{eq:int:N0} upon $\be$. Indeed, by exploiting the implicit function theorem for real analytic maps (see Deimling \cite[Thm.~15.3]{De85}) one proves the following.

\begin{thm}\label{thm:an}
There exist $0<\be^*<\bea$, an open neighborhood $\cU _*$ of ${\bm^*}\in\sB_1$ and a real analytic map ${\mathrm{M}\equiv({\mathrm{M}}_1,{\mathrm{M}}_2)}$ from $\bJe*$ to $\cU_*$ such that the set of zeros of $\fL$ in $\bJe*\times\cU_*$ coincides with the graph of ${\mathrm{M}}$.
\end{thm}
\begin{proof} 
The partial differential of $\fL$ with respect to $\bm$ evaluated at $({\bf 0},{\bm^*})$ is delivered by
\[
\begin{split}
&\partial_{\bm}\fL_1[{\bf 0},{\bm^*}](\bar\bm)=v_{{\G}}[\dO,\bar\mu_1]_{|\dpO}\,,\\
&\partial_{\bm}\fL_2[{\bf 0},{\bm^*}](\bar\bm)=v_{S_n}[\dincl,\bar\mu_2]_{|\dincl}\,,
\end{split}
\]
for all $\bar\bm\in\sB_1$.
Then by  Proposition \ref{V} and by the properties of the (classical) single layer potential (cf.~Lemma \ref{viso}) we deduce that $\partial_{\bm}\fL[{\bf 0},{\bm^*}]$ is an isomorphism from $\sB_1$ to $\sV^{1,\alpha}(\dpO)\times \Ca{1}(\dincl)$. Then the theorem follows by the implicit function theorem (see \cite[Thm.~15.3]{De85}) and by Proposition \ref{prop.N}.
\end{proof}

\vspace{\baselineskip}

\subsubsection{Macroscopic behavior}

In the following remark, we exploit the maps ${\mathrm{M}}_1$ and ${\mathrm{M}}_2$ of Theorem \ref{thm:an}  in the representation of the solution $\ueps$.

\begin{rmk}[Representation formula in the macroscopic variable]\label{e1e2}
Let the assumptions of Theorem \ref{thm:an} hold. Then
\[
\begin{split}
\ueps (\px)=&\uzero(\px)-\e1 ^{n-1}\e2 ^{n-1}\int_{\dincl} \bn_{\incl}(\ps)\cdot (\nabla_\py{\G})(\px,\e1 \p+\e1 \e2 \ps) \gi(\ps)\, d\sigma_\ps\\
&-\int_{\dpO}{\G}(\px,\py){{\mathrm{M}}}_1[\be](\py)\, d\sigma_\py\\
& + \e1 ^{n-2}\e2 ^{n-2}\int_{\dincl}{\G}(\px,\e1 \p +\e1 \e2 \ps){{\mathrm{M}}}_2[\be](\ps)\, d\sigma_\ps\, \qquad \forall \px \in \domeps,
\end{split}
\]
for all $\be \in \bIe*$.
\end{rmk}

As a consequence of Remark \ref{e1e2} one can prove that for all fixed $\px\in\dom$ the function $\ueps (\px)$ can be written in terms of a convergent power series of $\be$ for $\e1 $ and $\e2 $ positive and small. If $\dom'$ is an open subset of $\dom$ such that $0\notin\overline{\dom'}$, then a similar result holds for the restriction ${\ueps}_{|\overline{\dom'}}$, which describes the `macroscopic' behavior of $\ueps$ far from the hole. Namely, we are now ready to prove Theorem \ref{Ue1e2}.

\medskip

\begin{proofof}{Theorem \ref{Ue1e2}} 
Let $\be^*$ be as in Theorem \ref{thm:an}. We take $\be'\in\bIe{*}$ small enough so that $\overline{\incleps }\cap\overline{\dom'}=\emptyset$ for all $\be\in \bJep$. Then we define
\[
\begin{split}
{\fU}_{\dom'}[\be](\px)&\equiv \uzero (\px)-\e1 ^{n-1}\e2 ^{n-1}\int_{\dincl} \bn_{\incl}(\ps)\cdot (\nabla_\py{\G})(\px,\e1 \p+\e1 \e2 \ps) \gi(\ps)\, d\sigma_\ps\\
&\quad -\int_{\dpO}{\G}(\px,\py){{\mathrm{M}}}_1[\be](\py)\, d\sigma_\py\\
&\quad  + \e1 ^{n-2}\e2 ^{n-2}\int_{\dincl}{\G}(\px,\e1 \p +\e1 \e2 \ps){{\mathrm{M}}}_2[\be](\ps)\, d\sigma_\ps,
\end{split}
\]
for all $\px\in\overline{\dom'}$ and for all $\be\in\bJep$. Then, by Theorem \ref{thm:an} and by a standard argument   (see the study of $\fL_{2}$ in the proof of Proposition \ref{prop.N}) one deduces that ${\fU}_{\dom'}$ is real analytic from $\bJep$ to $\Ca{1}(\overline{\dom'})$. The validity of \eqref{Ue1e2.eq1} follows by Remark \ref{e1e2} and the validity of \eqref{Ue1e2.eq2} can be deduced by Proposition \ref{00}, by Theorem \ref{thm:an}, and by a straightforward computation.
\end{proofof}

\vspace{\baselineskip}

\subsubsection{Microscopic behavior}

By Remark \ref{e1e2} and by the rule of change of variable in integrals we obtain here below a representation of the solution $\ueps$ in proximity of the perforation.

\begin{rmk}[Representation formula in the microscopic variable]\label{micro_e1e2}
Let the assumptions of Theorem \ref{thm:an} hold. Then
\[
\begin{split}
\ueps (\e1\p+\e1\e2 \pt)&=\uzero(\e1\p+\e1\e2 \pt)-w^e_{S_n}[\dincl,\gi](\pt)\\
&\quad -\e2 ^{n-1}\int_{\dincl}\bn_\incl(\ps)\cdot\nabla S_n(-2p_n \se_n+\e2 (\msigma(\pt)-\ps))\, \gi(\ps)\, d\sigma_\ps\\
&\quad  - \int_{\dpO} {\G}(\e1 \p+\e1 \e2 \pt, \py)\, {{\mathrm{M}}}_1[\be](\py)\, d\sigma_\py\\
&\quad +v_{S_n}[\dincl,{{\mathrm{M}}}_2[\be]](\pt)-\e2 ^{n-2}\int_{\dincl} S_n(-2p_n\se_n+\e2 (\msigma(\pt)-\ps))\, {{\mathrm{M}}}_2[\be](\ps)\, d\sigma_\ps
\end{split}
\]
for all $\pt\in\mathbb{R}^n\setminus\omega$ and all $\be=(\e1,\e2) \in \bIe*$ such that $\e1\p+\e1\e2 \pt\in \overline\domeps$.
\end{rmk}

Then we can prove the following theorem, where we characterize the `microscopic' behavior of $\ueps$ close to hole, {\it i.e.}  $\ueps(\e1\p+\e1\e2\,\cdot\,)$ as $\be \to 0$.

\begin{thm}\label{Ve1e2}
Let the assumptions of Theorem \ref{thm:an} hold. Let $\incl'$ be an open bounded subset of $\mathbb{R}^n\setminus\overline\incl$. Let $\be''$ be such that $0<\be''<\be^*$ and $(\e1\p+\e1\e2\overline{\incl'})\subseteq\Bn{r_1}$ for all $\be\in  \bJes$.
Then there exists  a real analytic map $\fV_{\incl'}$ from $\bJes$ to $\Ca{1}(\overline{\incl'})$ such that
\begin{equation}\label{Ve1e2.eq1}
{\ueps(\e1\p+\e1\e2\,\cdot\,)}_{ |\overline{\incl'}}={\fV}_{\incl'}[\be]\qquad\forall \be\in\bIes\,.
\end{equation}
Moreover we have
\begin{equation}\label{Ve1e2.eq2}
{\fV}_{\incl'}[\mathbf{0}]=v_{0|\overline{\incl'}}
\end{equation}
where $v_0\in \Ca{1}_{\mathrm{loc}}(\mathbb{R}^n\setminus\incl)$ is the unique solution of
\[
\left\{
\begin{array}{ll}
\Delta v_0=0&\text{in }\mathbb{R}^n\setminus\incl\,,\\
v_0=\gi&\text{on }\dincl\\
\lim_{\pt\to\infty}v_0(\pt)=\go(0)\,.&
\end{array}
\right.
\]
\end{thm}
\begin{proof}   We define
\[
\begin{split}
{\fV}_{\incl'}[\be](\pt)&\equiv U_0(\e1\p+\e1\e2 \pt)-w^e_{S_n}[\dincl,\gi](\pt)\\
&\quad -\e2 ^{n-1}\int_{\dincl}\bn_\incl(\ps)\cdot\nabla S_n(-2p_n \se_n+\e2 (\msigma(\pt)-\ps))\, \gi(\ps)\, d\sigma_\ps\\
&\quad  - \int_{\dpO} {\G}(\e1 \p+\e1 \e2 \pt, \py)\, {{\mathrm{M}}}_1[\be](\py)\, d\sigma_\py\\
&\quad +v_{S_n}[\dincl,{{\mathrm{M}}}_2[\be]](\pt)-\e2 ^{n-2}\int_{\dincl} S_n(-2p_n\se_n+\e2 (\msigma(\pt)-\ps))\, {{\mathrm{M}}}_2[\be](\ps)\, d\sigma_\ps
\end{split}
\]
for all $\pt\in\overline{\incl'}$ and for all $\be\in\bJes$. Then, by Proposition \ref{Ub}, by Theorem \ref{thm:an}, and by a standard argument (see the study of $\fL_2$ in the proof of Proposition  \ref{prop.N})   one deduces that ${\fV}_{\incl'}$ is real analytic from $\bJes$ to $\Ca{1}(\overline{\incl'})$. The validity of \eqref{Ve1e2.eq1} follows by Remark \ref{micro_e1e2}. By a  straightforward computation and by Proposition \ref{00} one verifies that
\begin{equation}\label{Ve1e2.eq3}
{\fV}_{\incl'}[\mathbf{0}](\pt)= \go(0)-w^e_{S_n}[\dincl,\gi](\pt)+v_{S_n}[\dincl,{{\mathrm{M}}}_2[\mathbf{0}]](\pt),
\end{equation}
for all $\pt\in\overline{\incl'}$. Then, by Proposition \ref{00} and by the jump properties of the double layer potential we deduce that the right hand side of \eqref{Ve1e2.eq3} equals $\gi$ on $\partial\incl$. Hence, by the decaying properties at $\infty$ of the single and double layer potentials and by the uniqueness of the solution of the exterior Dirichlet problem, we deduce the validity of \eqref{Ve1e2.eq2}.
\end{proof}

\vspace{\baselineskip}

\subsubsection{Energy integral}

We now turn to study the behavior of the energy integral $\int_{\domeps}\left|\nabla\ueps\right|^2\,d\px$ by representing it in terms of a real analytic function. In Theorem \ref{Enge3} here below we consider the case when $\go=0$.

\begin{thm}\label{Enge3}
Let $\go=0$. Let the assumptions of Theorem \ref{thm:an} hold.  Then there exist $\be^{\fG}\in \bIe{*}$ and a real analytic map $\fG$ from  $\bJe{\fG}$ to $\mathbb{R}$ such that
\begin{equation}\label{Enge3.eq1}
\int_{\domeps}\left|\nabla\ueps\right|^2\,d\px=\e1^{n-2}\e2^{n-2}{\fG}(\be)\qquad\forall  \be\in\bIe{\fG}
\end{equation}
and
\begin{equation}\label{Enge3.eq2}
{\fG}(\mathbf{0})=\int_{\mathbb{R}^n\setminus\incl}\left|\nabla v_0\right|^2\,d\px\,.
\end{equation}
\end{thm}
\begin{proof} 
We observe that by the divergence theorem and by \eqref{bvpe} we have
\begin{equation}\label{Emge3.eq3}
\begin{split}
\int_{\domeps}\left|\nabla \ueps\right|^2\,d\px&=\int_{\dO}\ueps\;\bn_\Omega\cdot\nabla \ueps\,d\sigma-\int_{\dincl_\be}\ueps\;\bn_{\incleps}\cdot\nabla\ueps \,d\sigma
\\
&=-\int_{\dincl_\be}\gi\Big(\frac{\px-\e1 \p}{\e1 \e2 }\Big)\bn_{\incleps}(\px)\cdot\nabla\ueps(\px) \,d\sigma_\px\,.
\end{split}
\end{equation}
Then, we take  $\incl'$ as in Theorem \ref{Ve1e2} which in addition satisfies the condition $\partial\incl\subseteq\overline{\incl'}$. We set $\be^{\fG}\equiv\be''$ with $\be''$ as in Theorem \ref{Ve1e2} and we define
\[
{\fG}(\be)\equiv-\int_{\partial\incl} \gi\; \bn_\incl\cdot \nabla {\fV}_{\incl'}[\be]\,d\sigma\qquad\forall\be\in\bJe{\fG}\,.
\]
By Theorem  \ref{Ve1e2} and by standard calculus in Banach spaces it follows that $\fG$ is real analytic from $\bJe{\fG}$ to $\mathbb{R}$. By \eqref{Emge3.eq3} and by the rule of change of variable in integrals one shows the validity of  \eqref{Enge3.eq1}. Finally, the validity of \eqref{Enge3.eq2} follows by \eqref{Ve1e2.eq2} and by the divergence theorem. 
\end{proof}

\vspace{\baselineskip}

\medskip

We now consider the case when $\go\neq 0$. To do so, we need the following technical Lemma \ref{wnear0} which can be proved by the properties of integral operators with harmonic kernel (and no singularity).

\begin{lem}\label{wnear0}
Let $\cO $ be an open subset of $\mathbb{R}^n$ such that $\overline{\cO }\cap \overline{(\dpO\cup\msigma(\dpO))}=\emptyset$. Then  $w_G[\dpO,\psi]$ is harmonic on $\cO $  for all  $\psi\in \Ca{1}(\partial\dom)$.
\end{lem}

\begin{thm}\label{Enge3g}
Let the assumptions of Theorem \ref{thm:an} hold. Then there exist $\be^{\fE}\in \bIe{*}$ and a real analytic map $\fE$ from  $\bJe{\fE}$ to $\mathbb{R}$ such that
\begin{equation}\label{Enge3g.eq1}
\int_{\domeps}\left|\nabla\ueps\right|^2\,d\px={\fE}(\be)\qquad\forall  \be\in\bIe{\fE}
\end{equation}
and
\begin{equation}\label{Enge3g.eq2}
{\fE}(\mathbf{0})=\int_{\dom}\left|\nabla u_0\right|^2\,d\px\,.
\end{equation}
\end{thm}
\begin{proof}  As in the proof of Theorem \ref{Enge3} we begin by noting that, by the divergence theorem and by \eqref{bvpe}, we have
\begin{equation}\label{Enge3g.eq3}
\int_{\domeps}\left|\nabla \ueps\right|^2\,d\px=\int_{\dO}\go\;\bn_\Omega\cdot\nabla \ueps\,d\sigma-\int_{\dincl_\be}\gi\Big(\frac{\px-\e1 \p}{\e1 \e2 }\Big)\bn_{\incleps}(\px)\cdot\nabla\ueps(\px) \,d\sigma_\px
\end{equation}
for all $\be\in\bIea$. Then, by Remark \ref{e1e2} we have
\begin{equation}\label{Enge3g.eq4}
\int_{\partial\dom}\go\;\bn_\Omega\cdot\nabla u_{\be}\,d\sigma=I_{1,\be}+\e1^{n-1}\e2^{n-1}I_{2,\be}+\e1^{n-2}\e2^{n-2}I_{3,\be}
\end{equation}
with
\begin{align}
\label{Enge3g.eq5}
&I_{1,\be}=\int_{\dO}\go\;\bn_\Omega\cdot\nabla\uzero \,d\sigma- \int_{\dO}\go\;\bn_\Omega\cdot\nabla v_{{\G}}\bigl[\dO,\mathrm{M}_1[\be]\bigr]\,d\sigma_\px\,,\\
\nonumber
&I_{2,\be}= -\int_{\dO}\go(\px)\;\bn_\Omega(\px)\cdot\nabla_\px\int_{\dincl}\bn_\incl(\ps)\cdot(\nabla_\py {\G})(\px,\e1 \p+ \e1\e2 \ps)\, \gi(\ps)\, d\sigma_{\ps}\,d\sigma_\px\,,\\
\nonumber
&I_{3,\be}=  \int_{\dO}\go(\px)\;\bn_\Omega(\px)\cdot\nabla_\px\int_{\dincl}{\G}(\px,\e1 \p+ \e1\e2 \ps)\,\mathrm{M}_2[\be](\ps)\,d\sigma_{\ps}\,d\sigma_\px
\end{align}
for all $\be \in \bIe*$.
By the Fubini's theorem and by \eqref{Gsym}  it follows that
\[
\begin{split}
&I_{2,\be}= -\int_{\partial\incl}\gi(\ps)\, \bn_\incl(\ps)\cdot\nabla_\py\left(\int_{\dO}\go(\px)\;\bn_\dom(\px)\cdot\nabla_\px \left({\G}(\py,\px)\right)\,d\sigma_\px\right)_{\py=\e1 \p+ \e1\e2 \ps} d\sigma_{\ps}\,,\\
&I_{3,\be}= \int_{\partial\incl}\mathrm{M}_2[\be](\ps)\int_{\dO}\go(\px)\bn_\Omega(\px)\;\cdot\nabla_\px\left({\G}(\e1 \p+ \e1\e2 \ps,\px)\right)\,d\sigma_\px\,d\sigma_{\ps}\,.
\end{split}
\]
Then, by the definition of the double layer potential derived by $G$ (cf.~Definition \ref{def.SDG+}) and by \eqref{Gzero}, it follows that
\begin{equation}\label{Enge3g.eq6}
\begin{split}
&I_{2,\be}= -\int_{\partial\incl}\gi(\ps)\, \bn_\incl(\ps)\cdot\nabla w_G[\dpO,\go](\e1 \p+ \e1\e2\ps)\,d\sigma_\ps\,,\\
&I_{3,\be}= \int_{\partial\incl}\mathrm{M}_2[\be](\ps)\, w_G[\dpO,\go](\e1 \p+ \e1\e2 \ps)\,d\sigma_{\ps}
\end{split}
\end{equation}
for all $\be \in \bIe*$.

Now we choose a specific domain   $\incl'$ which satisfies the conditions  in Theorem \ref{Ve1e2} and which in addition contains the boundary of $\incl$ in its closure, namely such that $\partial\incl\subseteq\overline{\incl'}$. Then, for such $\incl'$,  we take $\be^{\fE}\equiv\be''$ with $\be''$ as in Theorem \ref{Ve1e2}. By \eqref{Ve1e2.eq1} and by a change of variable in the integral, we have
\begin{equation}\label{Enge3g.eq7}
\int_{\dincl_\be}\gi\Big(\frac{\px-\e1 \p}{\e1 \e2 }\Big)\bn_{\incleps}(\px)\cdot\nabla\ueps(\px) \,d\sigma_\px=\e1^{n-2}\e2^{n-2}\int_{\partial\incl} \gi\; \bn_\incl\cdot \nabla \fV_{\incl'}[\be]\,d\sigma
\end{equation}
for all $\be\in \bIe{\fE}$.

Then we define
\begin{align*}
&\fE_1(\be)\equiv\int_{\partial\Omega} \go\; \bn_\dom\cdot \nabla (u_0-v_G[\partial_+\dom,\mathrm{M}_1[\be]])\,d\sigma\,,\\
&\fE_2(\be)\equiv-\int_{\partial\incl} \gi(\ps)\,\bn_{\incl}(\ps)\cdot\nabla w_G[\partial_+\Omega,\go](\eps_1\p+\eps_1\eps_2 \ps)\, d\sigma_{\ps}\,,\\
&\fE_3(\be)\equiv\int_{\partial\incl}\mathrm{M}_2[\be](\ps)\, w_G[\partial_+\Omega,\go](\e1\p+\e1\e2 \ps)\, d\sigma_{\ps}\,,\\
&\fE_4(\be)\equiv-\int_{\partial\incl} \gi\; \bn_\incl\cdot \nabla \fV_{\incl'}[\be]\,d\sigma
\end{align*}
and
\begin{equation}\label{Enge3g.eq8}
\fE(\be)\equiv \fE_1(\be)+\e1^{n-1}\e2^{n-1} \fE_2(\be)+\e1^{n-2}\e2^{n-2}\fE_3(\be)+\e1^{n-2}\e2^{n-2} \fE_4(\be)\,
\end{equation}
for all $\be \in]-\be^\fE,\be^\fE[$. Now the validity \eqref{Enge3g.eq1} follows by \eqref{Enge3g.eq3}--\eqref{Enge3g.eq8}. In addition, by Theorems \ref{thm:an}  and  \ref{Ve1e2}, by Lemma \ref{wnear0}, and by a standard argument (see in the proof of  Proposition \ref{prop.N} the study of $\fL_2$),  we can prove that the $\fE_i$'s are real analytic from $]-\be^\fE,\be^\fE[$ to $\mathbb{R}$. Hence $\fE$ is real analytic from  $]-\be^\fE,\be^\fE[$ to $\mathbb{R}$.

To prove \eqref{Enge3g.eq2} we observe that by Proposition \ref{00} and Theorem \ref{thm:an}, we have ${{\mathrm{M}}}_1[\mathbf{0}]=0$. Thus \eqref{Enge3g.eq8} implies that ${\fE}(\mathbf{0})=\int_{\partial\dom}\go \bn_\dom\cdot \nabla u_0\, d\sigma$ and \eqref{Enge3g.eq2} follows by the divergence theorem.
\end{proof}

\vspace{\baselineskip}

\section{Asymptotic behavior of $\ueps$ in dimension $n=2$ for $\be$ close to $\b0$}\label{n=2}

When studying singular perturbation problems in perforated domains in the plane, it  is expected to see some logarithmic terms  in the description of the  perturbation. Such logarithmic terms  do not appear in dimension higher than or equal to three  and are generated by the specific behavior of the fundamental solution upon rescaling. Indeed,
\[
S_2(\varepsilon \pt)=S_2(\pt)+\frac{1}{2\pi}\log \varepsilon
\]
for all $\varepsilon>0$, and for the Green's function $G$ we have
\begin{multline}\label{n=2.eq1}
{\G}(\e1 \p+\e1 \e2 \pt,\e1 \p+\e1 \e2 \ps)\\
={S_2}(\pt-\ps)+\frac{1}{2\pi}\log \e1 \e2 -{S_2}(-2p_2\se_2+\e2(\msigma(\pt)-\ps))-\frac{1}{2\pi}\log \e1\, ,
\end{multline}
 for all $\be=(\e{1},\e{2}) \in \R^2_+$.
To handle the logarithimic terms, we need a representation formula for harmonic functions in $\domeps$ which is different from the one  that we have exploited in the case of dimension $\geq 3$.

First of all we note that, if  $\be \in \bIea$, then the sets $\domeps$ and $\incleps$ satisfy the same assumption \eqref{Omega}, \eqref{DOmega}, and \eqref{D+Omega} as $\dom$. Accordingly, we can apply the results of Subsection \ref{ss:potom} with $\dom$ replaced by $\domeps$ or $\incleps$.

In that spirit, we denote by $v_{{\G}}[\partial \incleps,1]$ the single layer potential with density function identically equal to $1$ on $\partial \incleps$:
\[
v_{{\G}}[\partial \incleps,1](\px)\equiv\int_{\partial \incleps}{\G}(\px,\py)\,d\sigma_\py,\qquad\forall \px\in \mathbb{R}^2\,.
\]
 We set $\sB_\be\equiv\Ca{0}_+(\dO)\times \Ca{0}_{\#}(\partial \incleps)$ (cf.~Definitions \ref{def.laypot1} and \ref{def:C+0}). Then we have the following proposition.

\begin{prop}\label{Vg}
Let $\be \in \bIea$ and $\rho\in\mathbb{R}\setminus\{0\}$. Then  the map
$$\begin{array}{rcl}
\sB_{\be}\times\mathbb{R}& \to &\sV^{1,\alpha}(\partial_+ \domeps)\\
(\bfi,\xi)& \mapsto &v_{{\G}}[\dO,\phi_1]_{|\partial_+ \domeps}+v_{{\G}}[\partial \incleps,\phi_2]_{|\partial_+ \domeps}+ \ \xi\ (\rho v_{{\G}}[\partial \incleps,1]_{|\partial_+ \domeps})
\end{array}$$
is an isomorphism.
\end{prop}
\begin{proof} 
We have
\[
v_{{\G}}[\dO,\phi_1]_{|\partial_+ \domeps}+v_{{\G}}[\partial \incleps,\phi_2]_{|\partial_+ \domeps}+(\rho v_{{\G}}[\partial \incleps,1]_{|\partial_+ \domeps}) \xi=v_{{\G}}[\partial \domeps,\phi]_{|\partial_+ \domeps}
\]
with
\[
\phi(\px)\equiv\left\{
\begin{array}{ll}
\phi_1(\px)&\forall \px\in\dO\,,\\
\phi_2(\px)+\rho\xi&\forall \px\in\partial \incleps\,.
\end{array}
\right.
\]
Then the statement follows by the definition of $\sV^{1,\alpha}$ as the image of the single layer potential derived by $G$ (cf.~Proposition \ref{V}).
\end{proof}

\vspace{\baselineskip}

\noindent Now, by Proposition \ref{Vg} and by the representation formula stated in Lemma \ref{i-green}  we have the  following Proposition \ref{f=WVg} where we show a suitable way to write a function of $\Ca{1}(\partial \domeps)$ as a sum of layer potentials derived by $G$.

\begin{prop}\label{f=WVg}
Let $\be \in \bIea$. Let $f\in \Ca{1}(\partial \domeps)$. Let $\rho\in\mathbb{R}\setminus\{0\}$. Then there exists a unique pair $(\bfi,\xi)=((\phi_1,\phi_2),\xi)\in
\sB_{\eps}\times\mathbb{R}$ such that
\[
f=w^i_{{\G}}[\partial \domeps,f]_{|\partial \domeps}+v_{{\G}}[\dO,\phi_1]_{|\partial \domeps}+v_{{\G}}[\partial \incleps,\phi_2]_{|\partial \domeps}+(\rho v_{{\G}}[\partial \incleps,1]_{|\partial \domeps}) \xi\,.
\]
\end{prop}

\subsection{Defining the operator $\fM$}\label{sub:L}
Let $\be \in \bIea$. By the previous  Proposition \ref{f=WVg}, we can look for solutions of problem \eqref{bvpe} in the form
\begin{equation}\label{rep}
w^i_{{\G}}[\partial \domeps,u_{\be|\partial \domeps}]+v_{{\G}}[\dO,\phi_1]+v_{{\G}}[\partial \incleps,\phi_2]+(\rho v_{{\G}}[\partial \incleps,1]) \xi
\end{equation}
for a suitable $(\bfi,\xi)\in \sB_{\eps}\times\mathbb{R}$. We split the integral on $\partial \domeps$ as the sum of integrals on $\dO$ and on $\partial \incleps$, we add and subtract $v_{G}^i[\dO,\bn_{\dom}\cdot \nabla u_{0|\dO}]$, and we obtain
\[
\begin{split}
w^i_{{\G}}[\dO,\go]-v_{G}^i[\dO,\bn_{\dom}\cdot \nabla &u_{0|\dO}]-w^e_{{\G}}\Big[\partial\incleps,\gi\big(\frac{\cdot-\e1\p}{\e1\e2}\big)\Big]\\&+v_{{\G}}[\dO,\phi_1+\bn_{\dom}\cdot \nabla u_{0|\dO}]+v_{{\G}}[\partial \incleps,\phi_2]+(\rho v_{{\G}}[\partial \incleps,1]) \xi\, .
\end{split}
\]
Then we note that
\[
 \uzero=w_{G}^i[\dO,\go]-v_{G}^i[\dO,\bn_{\dom}\cdot \nabla u_{0|\dO}]\,.
\]
 By taking $\rho=(\e1 \e2 \log(\e1 \e2 ))^{-1}$ and by performing a change of variable in the integrals over $\partial\incleps$, we deduce that the solutions of \eqref{bvpe} can be written  in the form
\begin{equation}\label{soln2}
\begin{split}
&\uzero (\px)-\e1 \e2 \int_{\dincl}\bn_\incl(\ps)\cdot(\nabla_\py {\G})(\px,\e1 \p+\e1 \e2 \ps)\, \gi(\ps)\, d\sigma_{\ps} + v_{{\G}}[\dO,\mu_1](\px)\\
&\quad + \int_{\dincl}{\G}(\px,\e1 \p+\e1 \e2 \ps)\,\mu_2(\ps)\,d\sigma_{\ps}
 + \frac{\xi}{\log(\e1 \e2 )}\int_{\dincl}{\G}(\px,\e1 \p+\e1 \e2 \ps)\,d\sigma_{\ps},\qquad\forall \px\in\domeps
\end{split}
\end{equation}
provided that  $(\mu_1,\mu_2,\xi)\in \Ca{0}_+(\dO)\times \Ca{0}_\#(\dincl)\times\R$ is chosen in such a way that the boundary conditions of \eqref{bvpe} are satisfied.

Now define $\sB\equiv \Ca{0}_+(\dO)\times \Ca{0}_\#(\dincl)$. We can verify that the (extension to $\overline{\domeps }$ of the) harmonic function in \eqref{soln2} solves problem \eqref{bvpe} if and only if the pair $(\bm,\xi)\in\sB\times\R$ solves
\begin{equation}\label{Lmu1mu2xi=0}
{\fM}[\be ,\tfrac1{\log(\e1 \e2)},\tfrac{\log\e1}{\log(\e1 \e2 )},\bm,\xi]=\b0\, ,
\end{equation}
 where $\fM[\be,\bde,\bm,\xi]\equiv (\fM_1[\be,\bde,\bm,\xi],\fM_2[\be,\bde,\bm,\xi])$ is defined for all $(\be,\bde,\bm,\xi)\in \bJea\times\mathbb{R}^2\times \sB\times\mathbb{R}$ by
\[\begin{split}
{{\fM}_1}[\be,\bde,\bm,\xi](\px)&\equiv v_{{\G}}[\dO,\mu_1](\px)\\
&\quad +\int_{\dincl}{\G}(\px,\e1 \p+\e1 \e2 \ps)\, \mu_2(\ps)\, d\sigma_{\ps}\\
&\quad +\delta_1\, \xi\int_{\dincl}{\G}(\px,\e1 \p+\e1 \e2 \ps)\, d\sigma_{\ps}\\
&\quad  - \e1 \e2 \int_{\dincl}\bn_\incl(\ps)\cdot(\nabla_\py {\G})(\px,\e1 \p+\e1 \e2 \ps)\, \gi(\ps)\, d\sigma_{\ps},\qquad\forall \px\in\dpO\,,
\end{split}\]
\[\begin{split}
{{\fM}_2}[\be,\bde,\bm,\xi](\pt)&\equiv v_{{S_2}}[\dincl,\mu_2](\pt)+\rho_\incl(1-\delta_2)\,\xi\\
&\quad  - \int_{\dincl} {S_2}(-2p_2\se_2+\e2 (\msigma(\pt)-\ps))\, \mu_2(\ps)\, d\sigma_{\ps}\\
&\quad +\delta_1\xi\int_{\dincl}\left({S_2}(\pt-\ps)-{S_2}(-2p_2\se_2+\e2 (\msigma(\pt)-\ps))\right)d\sigma_{\ps}\\
&\quad +\int_{\dpO} {\G}(\e1 \p+\e1 \e2 \pt, \py)\, \mu_1(\py)\, d\sigma_\py\\
&\quad  - \e2 \int_{\dincl}\bn_\incl(\ps)\cdot\nabla {S_2}(-2p_2\se_2+\e2 (\msigma(\pt)-\ps))\, \gi(\ps)\, d\sigma_{\ps}\\
&\quad  + U_0  (\e1 \p+\e1 \e2 \pt)-w_{{S_2}}[\dincl,\gi](\pt) - \frac{\gi(\pt)}2,\qquad \qquad\qquad\qquad\forall \pt\in\dincl\,,
\end{split}
\]
with
\[
\rho_\incl\equiv \frac{1}{2\pi}\int_{\dincl}d\sigma\,.
\]

Thus, to find the solution $\ueps$ of problem \eqref{bvpe} it suffices to find a solution of the system of integral equations  \eqref{Lmu1mu2xi=0} and, to study the asymptotic behavior of $\ueps$, we are now reduced to analyze the behavior of the solutions of \eqref{Lmu1mu2xi=0}.

 We incidentally observe that the dependence of equations \eqref{Lmu1mu2xi=0} upon the quotient  \eqref{eq:gam0} is generated by the presence of the term $(\rho v_{{\G}}[\partial \incleps,1]) \xi$ in the representation \eqref{rep}. Other geometric settings may lead to different integral equations which may not depend on  \eqref{rep}. For example,  in the   toy problem of Subsection \ref{toy} we don't have the exterior boundary $\dpO$ and, by  Lemma \ref{e-green}, we can write the solution as the sum of a double and a single layer potential supported on $\partial \incleps$. As we have mentioned at the end of Subsection \ref{beto0sec}, the expression \eqref{utoy} of such solution does not display a dependence on the quotient \eqref{eq:gam0}.

\subsection{Real analyticity of the operator ${\fM}$}

We are going to apply the implicit function theorem for real analytic maps to equation \eqref{Lmu1mu2xi=0} (see Deimling \cite[Thm.~15.3]{De85}).  As a first step, we prove that  ${\fM}$ defines a real analytic nonlinear operator between suitable Banach spaces.

\begin{prop}[Real analyticity of ${\fM}$]\label{Ln=2}
The map ${\fM}$ defined by
$$
 \begin{array}{rcl}
 \bJea\times \R^2 \times \sB\times \R& \to& \sV^{1,\alpha}(\dpO)\times \Ca{1}(\dincl)\\
(\be,\bde,\bm,\xi)& \mapsto & {\fM}[\be,\bde,{\bm},\xi]
 \end{array}
 $$
 is real analytic.
\end{prop}
\begin{proof} 
 We split the proof component by component.

\paragraph{Study of ${{\fM}_1}$} First we prove that ${{\fM}_1}$ is real analytic. \\
\noindent\emph{First step: the range of ${{\fM}_1}$ is a subset of $\sV^{1,\alpha}(\dpO)$.} 
Let $(\be,\bde,\bm,\xi)\in\bJea\times\mathbb{R}^2\times \sB\times\mathbb{R}$. Let  $U^e[\be,\bde,\bm,\xi]$ denote the function defined by
\[
\begin{split}
U^e[\be,\bde,\bm,\xi](\px)&\equiv v^e_{{\G}}[\dO,\mu_1](\px)\\
&\quad +\int_{\dincl}{\G}(\px,\e1 \p+\e1 \e2 \ps)\, \mu_2(\ps)\, d\sigma_{\ps}\\
&\quad +\delta_1\, \xi\int_{\dincl}{\G}(\px,\e1 \p+\e1 \e2 \ps)\, d\sigma_{\ps}\\
&\quad  - \e1 \e2 \int_{\dincl}\bn_\incl(\ps)\cdot(\nabla_\py {\G})(\px,\e1 \p+\e1 \e2 \ps)\, \gi(\ps)\, d\sigma_{\ps},
\quad\forall \px\in\overline{\mathbb{R}^2_+\setminus\dom}\,.\\
\end{split}
\]
The function $U^e[\be,\bde,\bm,\xi]$ belongs to $\in \Ca{1}_{\mathrm{loc}}(\overline{\mathbb{R}^2_+\setminus\dom})$ by the properties of the (Green) single layer potential and by the properties of integral operators with real analytic kernel and no singularity. In addition, one verifies that
\[
\begin{cases}
\Delta U^e[\be,\bde,\bm,\xi]=0& \text{ in }\mathbb{R}^2_+\setminus\overline{\dom} , \\
U^e[\be,\bde,\bm,\xi]=0& \text{ on }\partial\mathbb{R}^2_+\setminus\d0O ,\\
\lim_{\px\to\infty}U^e[\be,\bde,\bm,\xi](\px)=0,\\
\lim_{\px\to\infty}\frac{\px}{|\px|}\cdot\nabla U^e[\be,\bde,\bm,\xi](\px)=0
 \end{cases}
\]
(see also Lemma \ref{vGatinfty}). Then, by the characterization of $\sV^{1,\alpha}$ in Proposition \ref{f=ue},  we conclude that
${\fM}_1[\be,\bde,\bm,\xi]=U^e[\be,\bde,\bm,\xi]_{|\partial_+\Omega}\in \sV^{1,\alpha}(\dpO)$.

\noindent\emph{Second step: ${\fM}_1$ is real analytic.}  We observe that
\[
{\fM}_1[\be,\bde,\bm,\xi]=v_{{\G}}[\dO,\mu_1]_{|\dpO}+\mathfrak{f}[\be,\delta_1,\mu_2,\xi]
\]
where
\[
\begin{split}
\mathfrak{f}[\be,\delta_1,\mu_2,\xi](\px)
&\equiv\int_{\dincl}{\G}(\px,\e1 \p+\e1 \e2 \ps)\, \mu_2(\ps)\, d\sigma_{\ps}\\
&\quad+\delta_1\, \xi\int_{\dincl}{\G}(\px,\e1 \p+\e1 \e2 \ps)\, d\sigma_{\ps}\\
&\quad  - \e1 \e2 \int_{\dincl}\bn_\incl(\ps)\cdot(\nabla_\py {\G})(\px,\e1 \p+\e1 \e2 \ps)\, \gi(\ps)\, d\sigma_{\ps},\qquad\forall \px\in\overline{\dpO}\,.
\end{split}
\]
Since that the map which takes $\mu_1$ to $v_{{\G}}[\dO,\mu_1]_{|\dpO}$ is linear and continuous from $\Ca{0}_+(\dO)$ to $\sV^{1,\alpha}(\dpO)$, it is real analytic. Then, to prove that ${\fM}_1$ is real analytic we have to show that the map which takes $(\be,\delta_1,\mu_2,\xi)$ to $\mathfrak{f}[\be,\delta_1,\mu_2,\xi](\px)$ is real analytic from $ \bJea\times \R \times \Ca0_\#(\partial\incl)\times \R$ to $\sV^{1,\alpha}(\dpO)$. To that end, we  will show that there is a real analytic map
\[
\phi:\bJea\times\mathbb{R}\times \Ca{0}_{\#}(\dincl)\times\mathbb{R} \to \Ca{0}_+(\dO)
\]
such that
\begin{equation}\label{Fphi}
\mathfrak{f}[\be ,\delta_1,\mu_2,\xi]=v_{{\G}}[\dO,\phi[\be ,\delta_1,\mu_2,\xi]]_{|\dpO}
\end{equation}
for all $(\be ,\delta_1,\mu_2,\xi)\in\bJea\times\mathbb{R}\times \Ca{0}_{\#}(\dincl)\times\mathbb{R}$. Then the real analyticity of $\mathfrak{f}$ will follow from the definition of  the Banach space $\sV^{1,\alpha}(\dpO)$
in Proposition \ref{V}.

We will obtain such map $\phi$ as the sum of two real analytic terms. To construct the first one, we begin by observing that  $\mathfrak{f}$ is real analytic from $\bJea\times\mathbb{R}\times \Ca{0}_{\#}(\dincl)\times\mathbb{R}$ to $\Ca{1}(\overline{\dpO})$ by the properties of integral operators with real analytic kernel and no singularities (see Lanza de Cristoforis and Musolino \cite[Prop.~4.1 (ii)]{LaMu13}). Then, by the extension Lemma \ref{ext}, the composed map
\[
\begin{array}{rcl}
\bJea\times\mathbb{R}\times \Ca{0}_{\#}(\dincl)\times\mathbb{R}& \to &\Ca{1}(\dO)\\
(\be ,\delta_1,\mu_2,\xi)& \mapsto & E^{1,\alpha}\circ\mathfrak{f}[\be ,\delta_1,\mu_2,\xi]
\end{array}
\]
is real analytic.
Then we denote by $u^i[\be ,\delta_1,\mu_2,\xi]$ the unique solution of the Dirichlet problem in $\dom$ with boundary datum $E^{1,\alpha}\circ\mathfrak{f}[\be ,\delta_1,\mu_2,\xi]$. Since the map from $\Ca{1}(\dO)$ to $\Ca{1}(\overline{\dom})$ which takes a function $\psi$ to the unique solution of the Dirichlet problem in $\dom$ with boundary datum $\psi$ is linear and continuous, the map from $\bJea\times\mathbb{R}\times \Ca{0}_{\#}(\dincl)\times\mathbb{R}$ to $\Ca{1}(\overline{\dom})$ which takes $(\be ,\delta_1,\mu_2,\xi)$ to $u^i[\be ,\delta_1,\mu_2,\xi]$ is real analytic. In particular we have that
\begin{equation}\label{L1n=2.eq1}
\mbox{the map }
\begin{array}{rcl}\\
\bJea\times\mathbb{R}\times \Ca{0}_{\#}(\dincl)\times\mathbb{R}& \to &\Ca{0}_+(\dO)\\
(\be ,\delta_1,\mu_2,\xi)& \mapsto &\bn_\dom\cdot\nabla u^i[\be ,\delta_1,\mu_2,\xi]_{|\dO}
\end{array}
\mbox{is real analytic.}
\end{equation}

The map in \eqref{L1n=2.eq1} will be the first term in the sum which gives $\phi$. To obtain the second term,  we begin by taking
\[
\begin{split}
u^e[\be ,\delta_1,\mu_2,\xi](\px)&\equiv \int_{\dincl}{\G}(\px,\e1 \p+\e1 \e2 \ps)\, \mu_2(\ps)\, d\sigma_{\ps}\\
&\quad +\delta_1\, \xi\int_{\dincl}{\G}(\px,\e1 \p+\e1 \e2 \ps)\, d\sigma_{\ps}\\
&\quad  - \e1 \e2 \int_{\dincl}\bn_\incl(\ps)\cdot(\nabla_\py {\G})(\px,\e1 \p+\e1 \e2 \ps)\, \gi(\ps)\, d\sigma_{\ps},\qquad\forall \px\in\overline{\mathbb{R}^2_+\setminus\dom}\,.
\end{split}
\]
By standard properties of integral operators with real analytic kernels and no singularity  (see Lanza de Cristoforis and Musolino \cite[Prop.~4.1 (ii)]{LaMu13}), we have that the map from $\bJea\times\mathbb{R}\times \Ca{0}_{\#}(\dincl)\times\mathbb{R}$ to $\Ca{0}(\overline{\dpO})$ which takes $(\be ,\delta_1,\mu_2,\xi)$ to
\[
\begin{split}
&\bn_\dom\cdot\nabla u^e[\be ,\delta_1,\mu_2,\xi](\px)\\
&\qquad =\bn_\dom(\px)\cdot\int_{\dincl}\nabla_\px {\G}(\px,\e1 \p+\e1 \e2 \ps)\, \mu_2(\ps)\, d\sigma_{\ps}\\
&\qquad\quad +\delta_1\, \xi\, \bn_\dom(\px)\cdot\int_{\dincl}\nabla_\px {\G}(\px,\e1 \p+\e1 \e2 \ps)\, d\sigma_{\ps}\\
&\qquad\quad  - \e1 \e2 \sum_{j,k=1}^2(\bn_\dom(\px))_j\int_{\dincl}(\bn_\incl(\ps))_k(\partial_{x_j}\partial_{y_k} {\G})(\px,\e1 \p+\e1 \e2 \ps)\, \gi(\ps)\, d\sigma_{\ps},\qquad\forall \px\in\overline{\dpO}
\end{split}
\]
 is real analytic. Since by the extension Lemma \ref{ext} we can identify $\Ca{0}_+(\dO)$ with $\Ca{0}(\overline{\dpO})$, we deduce  that
\begin{equation}\label{L1n=2.eq2}
\mbox{the map} \begin{array}{rcl}\\
\bJea\times\mathbb{R}\times \Ca{0}_{\#}(\dincl)\times\mathbb{R}& \to& \Ca{0}_+(\dO)\\
(\be ,\delta_1,\mu_2,\xi)& \mapsto &\bn_\dom\cdot\nabla u^e[\be ,\delta_1,\mu_2,\xi]_{|\dO}
 \end{array}
 \mbox{is real analytic.}
\end{equation}

We now show that the maps in \eqref{L1n=2.eq1} and  \eqref{L1n=2.eq2} provide the two terms for the construction of $\phi$. First, we observe that $u^i[\be ,\delta_1,\mu_2,\xi]_{|\dpO}=\mathfrak{f}[\be ,\delta_1,\mu_2,\xi]$,  and thus by the representation formula in Lemma \ref{i-green} we have
\begin{equation}\label{representation.eq12D}
0=w_{{\G}}[\dpO,\mathfrak{f}[\be ,\delta_1,\mu_2,\xi]](\px)-v_{{\G}}[\dpO,\bn_{\dom}\cdot\nabla u^i[\be ,\delta_1,\mu_2,\xi]_{|\dpO}](\px)
\end{equation}
for all $ \px\in\R^2_+\setminus\overline{\dom}$. In addition, one verifies that $u^e[\be ,\delta_1,\mu_2,\xi]\in \Ca{1}_{\mathrm{loc}}(\overline{\mathbb{R}^2_+\setminus\dom})$ and that
\begin{equation}\label{L1n=2.eq3}
\left\{
\begin{array}{ll}
\Delta u^e[\be ,\delta_1,\mu_2,\xi]=0& \text{ in } \mathbb{R}^2_{+}\setminus\overline{\dom} ,\\
u^e[\be ,\delta_1,\mu_2,\xi](\px)=0 &\text{ on }\partial\mathbb{R}^2_{+}\setminus\d0O, \\
\lim_{\px\to\infty}u^e[\be ,\delta_1,\mu_2,\xi](\px)=0\,,&\\
\lim_{\px \to \infty} \frac{\px}{|\px|}\cdot \nabla u^{e}[\be ,\delta_1,\mu_2,\xi](\px)=0&
 \end{array}
 \right.
 \end{equation}
 (see also  Lemma \ref{vGatinfty}).  Then, by  \eqref{L1n=2.eq3}, by equality $u^e[\be ,\delta_1,\mu_2,\xi]_{|\dpO}=\mathfrak{f}[\be ,\delta_1,\mu_2,\xi]$, and by the exterior representation formula in Lemma \ref{e-green} we have
\begin{equation}\label{representation.eq22D}
u^e[\be ,\delta_1,\mu_2,\xi](\px)=-w_{{\G}}[\dpO,\mathfrak{f}[\be ,\delta_1,\mu_2,\xi]](\px)+v_{{\G}}[\dpO,\bn_{\dom}\cdot\nabla u^e[\be ,\delta_1,\mu_2,\xi]_{|\dpO}](\px)
\end{equation}
 for all $\px\in\R^2_+\setminus\overline{\dom}$. Then, by taking the sum of \eqref{representation.eq12D} and \eqref{representation.eq22D} and by the continuity properties of the (Green) single layer potential  we obtain that \eqref{Fphi} holds with
\[
\phi[\be ,\delta_1,\mu_2,\xi]=\bn_\dom\cdot\nabla u^e[\be ,\delta_1,\mu_2,\xi]_{|\dpO}-\bn_\dom\cdot\nabla u^i[\be ,\delta_1,\mu_2,\xi]_{|\dpO}\,.
\]
In addition,  by \eqref{L1n=2.eq1} and \eqref{L1n=2.eq2}, $\phi$ is real analytic from $\bJea\times\mathbb{R}\times \Ca{0}_{\#}(\dincl)\times\mathbb{R}$ to $\Ca{0}_+(\dO)$. The analyticity of  ${{\fM}_1}$ is now proved.

 \paragraph{Study of ${{\fM}_2}$} The analyticity of  the map ${{\fM}_2}$ from $\bJea\times\mathbb{R}^2\times {\sB}\times\mathbb{R}$ to $\Ca{1}(\dincl)$ can be proved by arguing as for ${\fL_2}$ in the proof of Proposition \ref{prop.N}.
\end{proof}

\vspace{\baselineskip}

\subsection{Functional analytic representation theorems}

\subsubsection{Analysis of \eqref{Lmu1mu2xi=0} via the implicit function theorem}

In this subsection, we study equation  \eqref{Lmu1mu2xi=0} around a  singular pair $(\be,\bde)=(\b0,(0,\lambda))$, with $\lambda\in[0,1[$. As a first step,
we investigate equation \eqref{Lmu1mu2xi=0} for $(\be,\bde)=(\b0,(0,\lambda))$.

\begin{prop}\label{000l}
Let $\lambda\in[0,1[$. There exists a unique $({\bm^*},\xi^*)\in {\sB}\times\mathbb{R}$ such that
\[
{\fM}[{\b0},(0,\lambda),{\bm^*},\xi^*]=\b0
\]
and we have
\[
\mu_1^*=0
\]
and
\[
v_{{S_2}}[\dO,\mu_2^*]_{|\dincl}+\rho_\incl(1-\lambda)\,\xi^*= - \go(0)+w_{{S_2}}[\dincl,\gi]_{|\dincl} + \frac{\gi}2\,.
\]
\end{prop}
\begin{proof}  First of all, we observe that for all $({\bm},\xi)\in {\sB}\times\mathbb{R}$, we have
\[\begin{cases}
{{\fM}_1}[{\b0}, (0,\lambda),{\bm},\xi](\px)=& v_{{\G}}[\dO,\mu_1](\px),\qquad\forall \px\in\dpO\,,\\
{{\fM}_2}[{\b0}, (0,\lambda),{\bm},\xi](\pt)=& v_{{S_2}}[\dincl,\mu_2](\pt)+\rho_\incl(1-\lambda)\, \xi\\ & \quad +v_{{\G}}[\dO,\mu_1](0) +\go  (0)-w_{S_2}[\dincl,\gi](\pt) - \frac{\gi(\pt)}2 ,\qquad\forall \pt\in\dincl\,.
\end{cases}
\]
 By Proposition \ref{V} (ii), the unique function in $\Ca{0}_+(\dO)$ such that $v_{{\G}}[\dO,\mu_1]=0$ on $\dpO$ is $\mu_1=0$. On the other hand, by classical potential theory, there exists a unique pair $(\mu_2,\xi) \in \Ca{0}_{\#}(\dincl)\times\mathbb{R}$ such that (cf.~Lemma \ref{viso})
\[
v_{{S_2}}[\dincl,\mu_2](\pt) +\rho_\incl(1-\lambda)\, \xi=  -\go  (0)+w_{S_2}[\dincl,\gi](\pt) + \frac{\gi(\pt)}2 ,\qquad\forall \pt\in\dincl.
\]
Now the validity of the proposition is proved.
\end{proof}

\vspace{\baselineskip}

Then, by the implicit function theorem for real analytic maps (see Deimling \cite[Thm.~15.3]{De85}) we deduce the following theorem.

\begin{thm}\label{analytic}
Let $\lambda\in[0,1[$. Let $(\bm^*,\xi^*)$ be as in Proposition \ref{000l}. Then there exist $\be^*\in\bIea$, an open neighborhood $\cV_{\lambda}$ of $(0,\lambda)$ in $\mathbb{R}^2$, an open neighborhood $\cU^*$ of $(\bm^*,\xi^*)$ in $\sB\times\mathbb{R}$, and a real analytic map $\Phi\equiv(\Phi_1,\Phi_2,\Phi_3)$ from $\bJe*\times \cV_\lambda$ to $\cU^*$ such that the set of zeros of ${\fM}$ in $\bJe*\times \cV_\lambda\times\cU^*$ coincides with the graph of $\Phi$.
\end{thm}
\begin{proof}  The partial differential of ${\fM}$ with respect to $(\bm,\xi)$ evaluated at $({\b0},(0,\lambda),\bm^*,\xi^*)$ is delivered by
\[
\begin{split}
&\partial_{(\bm,\xi)}{\fM}_1[{\b0},(0,\lambda),\bm^*,\xi^*]({\boldsymbol\phi},\zeta)=v_{{\G}}[\dO,\phi_1]_{|\dpO}\,,\\
&\partial_{(\bm,\xi)}{\fM}_2[{\b0},(0,\lambda),\bm^*,\xi^*]({\boldsymbol\phi},\zeta)=v_{S_2}[\dincl,\phi_2]_{|\dincl}+\rho_\incl(1-\lambda)\,\zeta\,,
\end{split}
\]
for all $({\boldsymbol\phi},\zeta)\in \sB\times\mathbb{R}$. Then by Proposition \ref{V} and by the properties of the single layer potential we deduce that $\partial_{(\bm,\xi)}{\fM}[{\b0},(0,\lambda),\bm^*,\xi^*]$ is an isomorphism from $\sB\times\mathbb{R}$ to $\sV^{1,\alpha}(\dpO)\times \Ca{1}(\dincl)$. Then the theorem follows by the implicit function theorem (see Deimling \cite[Thm.~15.3]{De85}) and by Proposition \ref{Ln=2}.
\end{proof}

\vspace{\baselineskip}

\subsubsection{Macroscopic behavior}

Since  $\log\e1 /\log(\e1\e2 )$ has no limit when $\be\in\bIea$ tends to $\b0$, we have to introduce a specific curve of parameters $\be$. Then, we take a function $\eta\mapsto\be(\eta)$  from $]0,1[$ to  $\bIea$ such that assumptions \eqref{eq:gam1} and \eqref{eq:gam2} hold (cf. Theorem \ref{thm:Ue1gamma-introd}). In the following Remark \ref{macro_e1}, we provide a convenient representation for the solution $u_{\be(\eta)}$.

\begin{rmk}[Representation formula in the macroscopic variable]
\label{macro_e1}
Let the assumptions of Theorem \ref{analytic} hold. Let  $\eta\mapsto\be(\eta)$ be a  function from $]0,1[$ to  $\bIea$ such that assumptions \eqref{eq:gam1} and \eqref{eq:gam2} hold. Let $\eta\mapsto\bde(\eta)$ be as in \eqref{eq:bfdelta}.  Then
\[
\begin{split}
&u_{\be(\eta)}(\px)=\uzero (\px)-\e1(\eta)\e2(\eta)\int_{\dincl}\bn_\incl(\ps)\cdot(\nabla_\py {\G})(\px,\e1(\eta) \p+ \e1(\eta)\e2(\eta) \ps)\, \gi(\ps)\, d\sigma_{\ps}\\
& + v_{{\G}}\bigl[\dO,\Phi_1\bigl[\be(\eta),\bde(\eta)\bigr]\bigr](\px)\\
& + \int_{\dincl}{\G}(\px,\e1(\eta) \p+ \e1(\eta)\e2(\eta) \ps)\,\Phi_2\bigl[\be(\eta),\bde(\eta)\bigr](\ps)\,d\sigma_{\ps}\\
& +\delta_1(\eta)\Phi_3\bigl[\be(\eta),\bde(\eta)\bigr]\int_{\dincl}{\G}(\px,\e1(\eta) \p+  \e1(\eta)\e2(\eta) \ps)\,d\sigma_{\ps}
\end{split}
\]
for all $\px\in\dom_{\be(\eta)}$ and for all $\eta \in ]0,1[$ such that $(\be(\eta),\bde(\eta))\in \bIe*\times \cV_\lambda$.
\end{rmk}

As a consequence of this representation formula, $u_{\be(\eta)}(\px)$ can be written as a converging power series of four real variables evaluated at $\bigl(\be(\eta),\bde(\eta)\bigr)$ for $\eta$ positive and small.  A similar result holds for the restrictions $u_{\be(\eta)|\overline{\dom'}}$ to any open subset  $\dom'$ of $\dom$ such that $0\notin\overline{\dom'}$. Namely, we are now in the position to prove Theorem \ref{thm:Ue1gamma-introd}.

\medskip

\begin{proofof}{Theorem \ref{thm:Ue1gamma-introd}} 
Let $\be^*$ and  $\cV_\lambda$ be as in Theorem \ref{analytic}. We take $\be'\in]\b0,\be^*[$ such that \eqref{Ue1gamma.eq0} holds true. Then, we define
\[
\begin{split}
\fU_{\dom'}[\be,\bde](\px)
&\equiv \uzero (\px)-\e1 \e2 \int_{\dincl}\bn_\incl(\ps)\cdot(\nabla_\py {\G})(\px,\e1 \p+\e1 \e2 \ps)\, \gi(\ps)\, d\sigma_{\ps}\\
&\quad + v_{{\G}}[\dO,\Phi_1[\be,\bde]](\px)\\
&\quad + \int_{\dincl}{\G}(\px,\e1 \p+\e1 \e2 \ps)\,\Phi_2[\be,\bde](\ps)\,d\sigma_{\ps}\\
&\quad + \delta_1 \Phi_3[\be,\bde]\int_{\dincl}{\G}(\px,\e1 \p+\e1 \e2 \ps)\,d\sigma_{\ps}
\end{split}
\]
for all $\px\in\overline{\dom'}$ and for all $(\be,\bde)\in ]-\be',\be'[\times \cV_{\lambda}$. By Theorem \ref{analytic} and by a standard argument (see in the proof of Proposition { \ref{prop.N}} the argument used to study $\fL_{2}$), we can show that  $\fU_{\dom'}$ is real analytic from $ ]-\be',\be'[\times \cV_{\lambda}$ to $\Ca{1}(\overline{\dom'})$.   The validity of \eqref{Ue1gamma.eq1} follows by Remark \ref{macro_e1} and the validity of \eqref{Ue1gamma.eq2} is deduced by Proposition \ref{000l}, by Theorem \ref{analytic}, and by a straightforward computation. 
\end{proofof}

\subsubsection{Microscopic behavior}

We now  present  a representation formula of the rescaled function $u_{\be(\eta)}(\e1(\eta)\p+\e1(\eta)\e2(\eta) \cdot)$.

\begin{rmk}[Representation formula in the microscopic variable]\label{micro_e1}
Let the assumptions of Theorem \ref{analytic} hold. Let  $\eta\mapsto\be(\eta)$ be a  function from $]0,1[$ to  $\bIea$ such that assumptions \eqref{eq:gam1} and \eqref{eq:gam2} hold. Let $\eta\mapsto\bde(\eta)$ be as in \eqref{eq:bfdelta}. Then
\[
\begin{split}
u_{\be(\eta)}(\e1(\eta)\p&+\e1(\eta)\e2(\eta) \pt)=\uzero(\e1(\eta)\p+\e1(\eta)\e2(\eta) \pt)-w^e_ {S_2}[\dincl,\gi](\pt)\\
& -\e2(\eta) \int_{\dincl}\bn_\incl(\ps)\cdot\nabla {S_2}\biggl(-2p_2 \se_2+\e2(\eta) (\msigma(\pt)-\ps)\biggr)\, \gi(\ps)\, d\sigma_{\ps}\\
& +\int_{\dpO} {\G}(\e1(\eta) \p+\e1(\eta)\e2(\eta) \pt, \py)\, \Phi_1\bigl[\be(\eta),\bde(\eta)\bigr](\py)\, d\sigma_\py\\
& +v_{S_2}\bigl[\dincl,\Phi_2\bigl[\be(\eta),\bde(\eta)\bigr]\bigr](\pt)\\
& -\int_{\dincl} {S_2}\biggl(-2p_2\se_2+\e2(\eta) (\msigma(\pt)-\ps)\biggr)\, \Phi_2\bigl[\be(\eta),\bde(\eta)\bigr](\ps)\, d\sigma_{\ps}\\
& +\rho_\incl\bigl(1-\delta_2(\eta)\bigr)\Phi_3\bigl[\be(\eta),\bde(\eta)\bigr]\\
& +\delta_1(\eta)\int_{\dincl}\left({S_2}(\pt-\ps)-{S_2}\biggl(-2p_2\se_2+\e2(\eta) (\msigma(\pt)-\ps)\biggr)\right)\,d\sigma_{\ps} \Phi_3\bigl[\be(\eta),\bde(\eta)\bigr],
\end{split}
\]
for all $\pt\in\mathbb{R}^2\setminus\omega$ and  for all $\eta \in ]0,1[$ such that $(\be(\eta),\bde(\eta))\in \bIe*\times \cV_\lambda$ and such that $\e1(\eta)\p+\e1(\eta)\e2(\eta) \pt\in \overline{\dom_{\be(\eta)}}$.
\end{rmk}

 In the following Theorem \ref{Ve1gamma}, we show that $u_{\be(\eta)}(\e1(\eta)\p+\e1(\eta)\e2(\eta)\, \cdot\,)$ for $\eta$ close to $0$ can be expressed as a real analytic map evaluated at $\bigl(\be(\eta),\bde(\eta)\bigr)$.

\begin{thm}\label{Ve1gamma}
Let the assumptions of Theorem \ref{analytic} hold.
Let $\incl'$ be an open bounded subset of $\mathbb{R}^2\setminus\overline\incl$ and let  $\be{''}\in]\b0,\be^*[$ be such that
\[
(\e1\p+\e1\e2\overline{\incl'})\subseteq\Bn{r_1}\qquad\forall \be\in ]-\be{''},\be{''}[\,.
\]
Then there is  a real analytic map
\[
\fV_{\incl'}: ]-\be'',\be''[\times \cV_{\lambda}\to\Ca{1}(\overline{\dom'})\,
\]
such that
 \begin{equation}\label{Ve1gamma.eq1}
u_{\be(\eta)}(\e1(\eta)\p+\e1(\eta)\e2(\eta)\, \cdot\,)_{|\overline{\incl'}}=\fV_{\incl'}\bigl[\be(\eta),\bde(\eta)\bigr]\,,\qquad\forall\eta \in]0,\eta''[\,.
\end{equation}
The equality in \eqref{Ve1gamma.eq1} holds for all parametrizations $\eta\mapsto\be(\eta)$ from $]0,1[$ to  $\bIea$  which satisfy  \eqref{eq:gam1} and \eqref{eq:gam2}. The function $\eta\mapsto\bde(\eta)$ is defined as in \eqref{eq:bfdelta}.\\
At the singular point $(\b0, (0,\lambda))$ we have
\begin{equation}\label{Ve1gamma.eq2}
\fV_{\incl'}[\b0,(0,\lambda)]=v_{0|\overline{\incl'}}
\end{equation}
where $v_0\in \Ca{1}_{\mathrm{loc}}(\mathbb{R}^2\setminus\incl)$ is the unique solution of \eqref{v0}.
\end{thm}
\begin{proof}   We define
\[
\begin{split}
\fV_{\incl'}[\be,\bde](\pt)&\equiv U_0(\e1\p+\e1\e2 \pt)-w^e_{S_2}[\dincl,\gi](\pt)\\
&\quad -\e2 \int_{\dincl}\bn_\incl(\ps)\cdot\nabla {S_2}(-2p_2 \se_2+\e2 (\msigma(\pt)-\ps))\, \gi(\ps)\, d\sigma_{\ps}\\
&\quad +\int_{\dpO} {\G}(\e1 \p+\e1\e2 \pt, \py)\, \Phi_1[\be,\bde](\py)\, d\sigma_\py\\
&\quad +v_{S_2}[\dincl,\Phi_2[\be,\bde]](\pt)\\
&\quad -\int_{\dincl} {S_2}(-2p_2\se_2+\e2 (\msigma(\pt)-\ps))\, \Phi_2[\be,\bde](\ps)\, d\sigma_{\ps}\\
&\quad +\rho_\incl(1-\delta_2)\Phi_3[\be,\bde]\\
&\quad +\delta_1\int_{\dincl}\left({S_2}(\pt-\ps)-{S_2}(-2p_2\se_2+\e2 (\msigma(\pt)-\ps))\right)\,d\sigma_{\ps}\; \Phi_3[\be,\bde]
\end{split}
\]
for all $\pt\in\overline{\incl'}$ and for all $(\be, \bde)\in]-\be'',\be''[\times  \cV_{\lambda}$. Then,  by Proposition \ref{Ub} and by a standard argument (see the study of $\fL_{2}$ in the proof of Proposition  \ref{prop.N}) we verify that $\fV_{\incl'}$ is real analytic from $]-\be'',\be''[\times  \cV_{\lambda}$ to $\Ca{1}(\overline{\incl'})$. The validity of \eqref{Ve1gamma.eq1} follows by Remark \ref{micro_e1}. By a  straightforward computation and by Proposition \ref{000l} one verifies that
\begin{equation}\label{Ve1gamma.eq3}
\begin{split}
\fV_{\incl'}[\mathbf{0},(0,\lambda)](\pt)&= \go(0)-w^e_{S_2}[\dincl,\gi](\pt)+v_{S_2}[\dincl,\Phi_2[\mathbf{0},(0,\lambda)]](\pt)+\rho_\incl(1-\lambda)\,\Phi_3[\mathbf{0},(0,\lambda)]
\end{split}
\end{equation}
for all $\pt\in\overline{\incl'}$. Then, we deduce that the right hand side of \eqref{Ve1gamma.eq3} equals $\gi$ on $\partial\incl$ by Proposition \ref{000l} and by the jump properties of the double layer potential. Hence, by the decaying properties at $\infty$ of the single and double layer potentials and by the uniqueness of the solution of the exterior Dirichlet problem, we deduce the validity of \eqref{Ve1gamma.eq2}.\end{proof}

\vspace{\baselineskip}

\subsubsection{Energy integral}

We turn to consider the behavior of the energy integral $\displaystyle\int_{\dom_{\be(\eta)}}\left|\nabla u_{\be(\eta)}\right|^2\,d\px$ for $\eta$ close to $0$.

\begin{thm}\label{Enge2}
Let the assumptions of Theorem \ref{analytic} hold.  Then there exist $\be^{\fE}\in \bIe{*}$ and a real analytic function
\[
\fE:\bJe{\fE}\times\cV_{\lambda}\to\mathbb{R}
\]
such that
\begin{equation}\label{Enge2.eq1}
\int_{\dom_{\be(\eta)}}\left|\nabla u_{\be(\eta)}\right|^2\,d\px=\fE\bigl(\be(\eta), \bde(\eta)\bigr)\,,\qquad\forall\eta\in]0,\eta^\fE[\,,
\end{equation}
where the latter equality holds for all functions $\eta\mapsto\be(\eta)$ from $]0,1[$ to  $\bIea$  which satisfy  \eqref{eq:gam1}  and \eqref{eq:gam2} and with $\eta\mapsto\bde(\eta)$ as in \eqref{eq:bfdelta}, and for all $\eta^\fE \in ]0,1[$ such that
\[
\bigl(\be(\eta),\bde(\eta)\bigr)\in \bIe{\fE}\times\cV_\lambda\,,\qquad\forall \eta\in]0,\eta^\fE[\,.
\]
In addition,
\begin{equation}\label{Enge2.eq2}
\fE(\b0,(0,\lambda))=\int_{\Omega}|\nabla u_0|^2\,d\px+\int_{\mathbb{R}^2\setminus\incl}\left|\nabla v_0\right|^2\,d\px\,.
\end{equation}
\end{thm}

\begin{proof}  By the divergence theorem and by \eqref{bvpe} we have
\begin{equation}\label{Enge2.eq7}
\begin{split}
\int_{\domeps}\left|\nabla \ueps\right|^2\,d\px&=\int_{\dO}\ueps\;\bn_\Omega\cdot\nabla \ueps\,d\sigma-\int_{\dincl_\be}\ueps\;\bn_{\incleps}\cdot\nabla\ueps \,d\sigma\\
&=\int_{\dO}\go\;\bn_\Omega\cdot\nabla \ueps\,d\sigma-\int_{\dincl_\be}\gi\Big(\frac{\px-\e1 \p}{\e1 \e2 }\Big)\bn_{\incleps}(\px)\cdot\nabla\ueps(\px) \,d\sigma_\px
\end{split}
\end{equation}
for all $\be\in\bIea$. Then we take a function $\eta\mapsto\be(\eta)$  from $]0,1[$ to  $\bIea$ and a function $\eta\mapsto\bde(\eta)$ from $]0,1[$ to  $\R^2$ which satisfy \eqref{eq:gam1} -- \eqref{eq:bfdelta}. By Remark \ref{macro_e1} we have
\begin{equation}\label{Enge2.eq8}
\int_{\partial\dom}\go\;\bn_\Omega\cdot\nabla u_{\be(\eta)}\,d\sigma=I_{1,\eta}+\e1(\eta)\e2(\eta)I_{2,\eta}+I_{3,\eta}+\delta_1(\eta)I_{4,\eta}
\end{equation}
 for all $\eta \in ]0,1[$ such that $(\be(\eta),\bde(\eta))\in \bIe*\times \cV_\lambda$, where
\begin{align*}
&I_{1,\eta}=\int_{\dO}\go\;\bn_\Omega\cdot\nabla\uzero \,d\sigma+ \int_{\dO}\go\;\bn_\Omega\cdot\nabla v_{{\G}}\bigl[\dO,\Phi_1\bigl[\be(\eta),\bde(\eta)\bigr]\bigr]\,d\sigma\,,\\
\nonumber
&I_{2,\eta}= -\int_{\dO}\go(\px)\;\bn_\Omega(\px)\cdot\nabla_\px\int_{\dincl}\bn_\incl(\ps)\cdot(\nabla_\py {\G})(\px,\e1(\eta) \p+ \e1(\eta)\e2(\eta) \ps)\, \gi(\ps)\, d\sigma_{\ps}\,d\sigma_\px\,,\\
\nonumber
&I_{3,\eta}=  \int_{\dO}\go(\px)\;\bn_\Omega(\px)\cdot\nabla_\px\int_{\dincl}{\G}(\px,\e1(\eta) \p+ \e1(\eta)\e2(\eta) \ps)\,\Phi_2\bigl[\be(\eta),\bde(\eta)\bigr](\ps)\,d\sigma_{\ps}\,d\sigma_\px\,,\\
\nonumber
&I_{4,\eta}=  \Phi_3\bigl[\be(\eta),\bde(\eta)\bigr]\int_{\partial\dom}\go(\px)\;\bn_\Omega(\px)\cdot\nabla_\px\int_{\dincl}{\G}(\px,\e1(\eta) \p+  \e1(\eta)\e2(\eta) \ps)\,d\sigma_{\ps}\,d\sigma_\px\,.
\end{align*}
By the Fubini's theorem and by \eqref{Gsym}  it follows that
\[
\begin{split}
&I_{2,\eta}= -\int_{\partial\incl}\gi(\ps)\, \bn_\incl(\ps)\cdot\nabla_\py\left(\int_{\dO}\go(\px)\;\bn_\dom(\px)\cdot\nabla_\px\left({\G}(\py,\px)\right)\,d\sigma_\px\right)_{\py=\e1(\eta) \p+ \e1(\eta)\e2(\eta) \ps} d\sigma_{\ps}\,,\\
&I_{3,\eta}= \int_{\partial\incl}\Phi_2\bigl[\be(\eta),\bde(\eta)\bigr](\ps)\int_{\dO}\go(\px)\bn_\Omega(\px)\;\cdot\nabla_\px\left({\G}(\e1(\eta) \p+ \e1(\eta)\e2(\eta) \ps,\px)\right)\,d\sigma_\px\,d\sigma_{\ps}\,,\\
&I_{4,\eta}= \delta_1(\eta)\Phi_3\bigl[\be(\eta),\bde(\eta)\bigr]\int_{\partial\incl}\int_{\dO}\go(\px)\;\bn_\Omega(\px)\cdot(\nabla_\px\,{\G})(\px,\e1(\eta) \p+  \e1(\eta)\e2(\eta) \ps)\,d\sigma_\px\,d\sigma_{\ps}\,,
\end{split}
\]
and, by the definition of the double layer potential derived by $G$ (cf.~Definition \ref{def.SDG+}) and by \eqref{Gzero}, we deduce that
\begin{equation}\label{Enge2.eq10}
\begin{split}
&I_{2,\eta}= -\int_{\partial\incl}\gi(\ps)\, \bn_\incl(\ps)\cdot\nabla w_G[\dpO,\go](\e1(\eta) \p+ \e1(\eta)\e2(\eta) \ps)\,d\sigma_\ps\,,\\
&I_{3,\eta}= \int_{\partial\incl}\Phi_2\bigl[\be(\eta),\bde(\eta)\bigr](\ps)\, w_G[\dpO,\go](\e1(\eta) \p+ \e1(\eta)\e2(\eta) \ps)\,d\sigma_{\ps}\,,\\
&I_{4,\eta}=\Phi_3\bigl[\be(\eta),\bde(\eta)\bigr]\int_{\partial\incl}w_G[\dpO,\go](\e1(\eta) \p+ \e1(\eta)\e2(\eta) \ps)\,d\sigma_{\ps}\,,
\end{split}
\end{equation}
 for all $\eta \in ]0,1[$ such that $(\be(\eta),\bde(\eta))\in \bIe*\times \cV_\lambda$.

Now we choose a specific domain   $\incl'$ which satisfies the conditions  in Theorem \ref{Ve1gamma} and which in addition contains the boundary of $\incl$ in its closure, namely such that $\partial\incl\subseteq\overline{\incl'}$. Then, for such $\incl'$,  we take $\be^{\fE}\equiv\be''$ with $\be''$ as in Theorem \ref{Ve1gamma}. By \eqref{Ve.eq1} and by a change of variable in the integral, we have
\begin{equation}\label{Enge2.eq11}
\int_{\dincl_\be}\gi\Big(\frac{\px-\e1(\eta) \p}{\e1(\eta) \e2(\eta) }\Big)\bn_{\omega_{\be(\eta)}}(\px)\cdot\nabla u_{\be(\eta)}(\px) \,d\sigma_\px=\int_{\partial\incl} \gi\; \bn_\incl\cdot \nabla \fV_{\incl'}[\be(\eta),\bde(\eta)]\,d\sigma,
\end{equation}
for all $\eta \in ]0,1[$ such that $(\be(\eta),\bde(\eta))\in \bIe{\fE}\times \cV_\lambda$.

Then we define
\begin{align*}
&\fE_1(\be,\bde)\equiv\int_{\partial\Omega} \go\; \bn_\dom\cdot \nabla (u_0+v_G[\partial_+\dom,\Phi_1[\be,\bde]])\,d\sigma\,,\\
&\fE_2(\be,\bde)\equiv-\int_{\partial\incl} \gi(\ps)\,\bn_{\incl}(\ps)\cdot\nabla w_G[\partial_+\Omega,\go](\eps_1\p+\eps_1\eps_2 \ps)\, d\sigma_{\ps}\,,\\
&\fE_3(\be,\bde)\equiv\int_{\partial\incl}\Phi_2[\be,\bde](\ps)\, w_G[\partial_+\Omega,\go](\e1\p+\e1\e2 \ps)\, d\sigma_{\ps}\,,\\
&\fE_4(\be,\bde)\equiv\Phi_3[\be,\bde]\int_{\partial\incl}  w_G[\partial_+\Omega,\go](\e1\p+\e1\e2 \ps)\, d\sigma_{\ps}\,,\\
&\fE_5(\be,\bde)\equiv-\int_{\partial\incl} \gi\; \bn_\incl\cdot \nabla \fV_{\incl'}[\be,\bde]\,d\sigma
\end{align*}
and
\begin{equation}\label{Enge2.eq3}
\fE(\be,{\bde})\equiv \fE_1(\be,\bde)+\eps_1\eps_2 \fE_2(\be,\bde)+\fE_3(\be,\bde)+\delta_1 \fE_4(\be,\bde)+\fE_5(\be,\bde)\,
\end{equation}
for all $(\be,{\bde}) \in]-\be^\fE,\be^\fE[\times \cV_{\lambda}$.
Now the validity \eqref{Enge2.eq1} follows by \eqref{Enge2.eq7}--\eqref{Enge2.eq11}. In addition, by Theorems \ref{analytic}  and  \ref{Ve1gamma}, by Lemma \ref{wnear0}  (which holds also for $n=2$), and by a standard argument (see in the proof of  Proposition \ref{prop.N} the study of $\fL_2$),  we can prove that the $\fE_i$'s are real analytic from $]-\be^\fE,\be^\fE[\times \cV_{\lambda}$ to $\mathbb{R}$. Hence $\fE$ is real analytic from  $]-\be^\fE,\be^\fE[\times \cV_{\lambda}$ to $\mathbb{R}$.

To complete the proof we have to verify \eqref{Enge2.eq2}. We begin by observing  that $\Phi_1[\b0,(0,\lambda)]=0$ by Proposition \ref{000l} and Theorem \ref{analytic}. Thus
\begin{equation}\label{Enge2.eq4}
\fE_1(\b0,(0,\lambda))=\int_{\partial\dom}\go\; \bn_\dom\cdot \nabla u_0\, d\sigma=\int_{\Omega}|\nabla u_0|^2\,d\px\,.
\end{equation}
By Lemma \ref{jumps}, we have $w_G[\partial_+\Omega,\go](0)=\go(0)$. Since $\Phi_2[\b0,(0,\lambda)]$ belongs to $\Ca{1}_{\#}(\partial\omega)$, we compute
\begin{equation}\label{Enge2.eq5}
 \fE_3(\b0,(0,\lambda))= \go(0)\,  \int_{\partial\incl}\Phi_2[\b0,(0,\lambda)]\, d\sigma=0\,.
\end{equation}
Then, by \eqref{Ve1gamma.eq2} and by the divergence theorem,   we have
\begin{equation}\label{Enge2.eq6}
\fE_4(\b0,(0,\lambda))=-\int_{\partial\incl} \gi\; \bn_\incl\cdot \nabla v_0\,d\sigma=\int_{\mathbb{R}^2\setminus\incl}\left|\nabla v_0\right|^2\,d\px\,.
\end{equation}
We conclude by \eqref{Enge2.eq3} -- \eqref{Enge2.eq6}.
\end{proof}

\vspace{\baselineskip}

Finally, in the following Theorem \ref{Fluxe2} we consider the total flux on $\partial\Omega$.

\begin{thm}\label{Fluxe2}
Let the assumptions of Theorem \ref{analytic} hold.   Then there exist $\be^{\fF}\in \bIe{*}$ and a real analytic function
\[
\fF:\bJe{{\fF}}\times\cV_{\lambda}\to\mathbb{R}
\]
such that
\[
\int_{\partial\dom}\bn_\Omega\cdot\nabla u_{\be(\eta)}\,d\sigma={\fF}\bigl(\be(\eta), \bde(\eta)\bigr)\,,\qquad\forall \eta\in]0,\eta^\fF[
\]
where the latter equality holds for all functions $\eta\mapsto\be(\eta)$ from $]0,1[$ to  $\bIea$  which satisfy  \eqref{eq:gam1}  and \eqref{eq:gam2} and with $\eta\mapsto\bde(\eta)$ as in \eqref{eq:bfdelta}, and for all $\eta^\fF\in ]0,1[$ such that
\begin{equation}\label{Fluxe2.eq1}
\bigl(\be(\eta),\bde(\eta)\bigr)\in \bIe{\fF}\times\cV_\lambda\,,\qquad\forall \eta\in]0,\eta^\fF[\,.
\end{equation}
 Furthermore,
\[
{\fF}(\b0,(0,\lambda))=0\,.
\]
\end{thm}
\begin{proof} 
Let $\eta\mapsto\be(\eta)$  from $]0,1[$ to  $\bIea$ and $\eta\mapsto\bde(\eta)$ from $]0,1[$ to  $\R^2$ which satisfy \eqref{eq:gam1} -- \eqref{eq:bfdelta}. Then by the divergence theorem we have
\[
\int_{\partial\dom}\bn_\Omega\cdot\nabla u_{\be(\eta)}\,d\sigma=\int_{\partial\incleps}\bn_{\incleps}\cdot\nabla u_{\be(\eta)}\,d\sigma
\]
 for all $\eta \in ]0,1[$ such that $(\be(\eta),\bde(\eta))\in \bIe*\times \cV_\lambda$. Then we take $\incl'$ which satisfies the conditions  in Theorem \ref{Ve1gamma} and such that $\partial\incl\subseteq\overline{\incl'}$. Then, for such $\incl'$,  we take $\be^{\fF}\equiv\be''$ with $\be''$ as in Theorem \ref{Ve1gamma} and we deduce that
\[
\int_{\partial\dom}\bn_\Omega\cdot\nabla u_{\be(\eta)}\,d\sigma=\int_{\partial\incl}\bn_{\incl}\cdot\nabla \fV_{\incl'}[\be(\eta),\bde(\eta)]\,d\sigma
\]
for all $\eta \in ]0,1[$ such that $(\be(\eta),\bde(\eta))\in \bIe{\fF}\times \cV_\lambda$. Accordingly, we define
\[
\fF(\be,\bde)\equiv\int_{\partial\incl}\bn_{\incl}\cdot\nabla \fV_{\incl'}[\be,\bde]\,d\sigma\,,\qquad\forall(\be,\bde)\in\bJe{{\fF}}\times\cV_{\lambda}\,.
\]
Then the equality \eqref{Fluxe2.eq1} holds true. By Theorem \ref{Ve1gamma}, one deduces  that $\fF$ is real analytic from $\bJe{{\fF}}\times\cV_{\lambda}$ to $\R$. Finally, by \eqref{Ve1gamma.eq2} we have
\[
{\fF}(\b0,(0,\lambda))=\int_{\partial\incl}\bn_{\incl}\cdot\nabla \fV_{\incl'}[\b0,(0,\lambda)]\,d\sigma=\int_{\partial\incl}\bn_{\incl}\cdot\nabla v_0\,d\sigma
\]
and the latter integral vanishes because $v_0$ is harmonic at infinity (see \eqref{v0}).\end{proof}

\vspace{\baselineskip}

\section{Asymptotic behavior of $\ueps$ in dimension $n=2$ for $\e1$ close to $0$ and $\e2=1$}\label{e2=1}

 As noticed in the beginning of Section \ref{n=2}, when studying singular perturbation problems in perforated domains in the two-dimensional plane one would expect to have some logarithmic terms in the asymptotic formulas. Such logarithmic terms  are generated by the specific behavior of the fundamental solution upon rescaling (cf.~equality \eqref{n=2.eq1}). However, for our problem there will be no logarithmic term   when $\e2=1$ is fixed and we just consider the dependence upon $\e1$. Indeed, for $\e2=1$, we have
 \[
 {S_2}(\e1 \p+\e1\e2 \pt)={S_2}(\p+\pt)+\frac{\log\e1}{2\pi}
 \]
and thus
\[
{\G}(\e1 \p+\e1\e2 \pt,\e1 \p+\e1\e2 \ps)={S_2}(\pt-\ps) -{S_2}(-2p_2\se_2+(\msigma(\pt)-\ps))
\]
for all $\e1>0$. Accordingly the rescaling of $G$ gives rise to no logarithmic term.

Since we are dealing here with a one parameter problem, we find convenient to take $\varepsilon\equiv\e1$, $\ea\equiv\eps^{\rm ad}_1$, $\Omega_\varepsilon\equiv\Omega_{\e1,1}$, $\omega_\varepsilon\equiv\omega_{\e1,1}$, and $u_\varepsilon\equiv u_{\e1,1}$ for all $\varepsilon\in ]-\ea,\ea[$.

\subsection{Defining the operator $\fN$}
Let $\varepsilon \in ]-\ea,\ea[$. By Proposition \ref{f=WVg} we can look for solutions of problem \eqref{bvpe} under the form
\[
w^i_{{\G}}[\partial \dome,u_{\varepsilon|\partial\dome}]+v_{{\G}}[\dO,\phi_1]+v_{{\G}}[\partial \incle,\phi_2]+v_{{\G}}[\partial \incle,1]\, \xi
\]
for suitable $(\phi_1,\phi_2,\xi)\in \Ca{0}_+(\dO)\times \Ca{0}_\#(\partial \incle)\times\mathbb{R}$.
We  split the integral on $\partial \dome$ as the sum of integrals on $\dO$ and on $\partial \incle$, we add and subtract $v_{G}^i[\dO,\bn_{\dom}\cdot \nabla u_{0|\dO}]$  to obtain the new form
\[
\begin{split}
w^i_{{\G}}[\dO,\go]-v_{G}^i[\dO,\bn_{\dom}\cdot \nabla &u_{0|\dO}]-w^e_{{\G}}\Big[\partial\incle,\gi\big(\frac{\cdot-\eps\p}{\eps}\big)\Big]\\&+v_{{\G}}[\dO,\phi_1+\bn_{\dom}\cdot \nabla u_{0|\dO}]+v_{{\G}}[\partial \incle,\phi_2]+v_{{\G}}[\partial \incle,1]\, \xi\, .
\end{split}
\]
Since
\[
 \uzero=w_{G}^i[\dO,\go]-v_{G}^i[\dO,\bn_{\dom}\cdot \nabla u_{0|\dO}]\, ,
\]
we finally look for solutions of \eqref{bvpe} in the form
\begin{equation}\label{soln2e2=1}
\begin{split}
&\uzero (\px)-\eps \int_{\dincl}\bn_\incl(\ps)\cdot(\nabla_\py {\G})(\px,\eps \p+\eps \ps)\, \gi(\ps)\, d\sigma_{\ps} + v_{{\G}}[\dO,\mu_1](\px)\\
&\quad + \int_{\dincl}{\G}(\px,\eps \p+\eps \ps)\,\mu_2(\ps)\,d\sigma_{\ps}
 + \xi \int_{\dincl}{\G}(\px,\eps \p+\eps \ps)\,d\sigma_{\ps},\qquad\forall \px\in\dome
\end{split}
\end{equation}
for suitable $({\bm},\xi)\in \sB\times\mathbb{R}$ ensuring that the boundary conditions of \eqref{bvpe} are satisfied (here as in Section \ref{n=2} we take $\sB\equiv\Ca{0}_+(\dO)\times \Ca{0}_{\#}(\dincl)$).

The (extension to $\overline{\dome}$ of the) harmonic function in \eqref{soln2e2=1} solves problem \eqref{bvpe} if and only if the pair $({\bm},\xi)$ solves
\begin{equation}\label{Mmu1mu2xi=0}
{\fN}[\eps,{\bm},\xi]=\b0\, ,
\end{equation}
with  ${\fN}[\eps,{\bm},\xi]\equiv ({\fN}_1[\eps,{\bm},\xi],{\fN}_2[\eps,{\bm},\xi])$ defined by
\[\begin{split}
{\fN}_1[\eps,{\bm},\xi](\px)&\equiv v_{{\G}}[\dO,\mu_1](\px)\\
&\quad +\int_{\dincl}{\G}(\px,\eps \p+\eps \ps)\, \mu_2(s)\, d\sigma_{\ps}\\
&\quad +\xi\int_{\dincl}{\G}(\px,\eps \p+\eps \ps)\, d\sigma_{\ps}\\
&\quad  - \eps \int_{\dincl}\bn_\incl(\ps)\cdot(\nabla_\py {\G})(\px,\eps \p+\eps \ps)\, \gi(\ps)\, d\sigma_{\ps},\qquad\forall \px\in\dpO\,,\\
{\fN}_2[\eps,{\bm},\xi](\pt)&\equiv v_{S_2}[\dincl,\mu_2](\pt)\\
&\quad  - \int_{\dincl} {S_2}(-2p_2\se_2+\msigma(\pt)-\ps)\, \mu_2(\ps)\, d\sigma_{\ps}\\
&\quad +\xi\int_{\dincl}\left({S_2}(\pt-\ps)-{S_2}(-2p_2\se_2+\msigma(\pt)-\ps)\right)d\sigma_{\ps}\\
&\quad +\int_{\dpO} {\G}(\eps \p+\eps \pt, \py)\, \mu_1(\py)\, d\sigma_\py\\
&\quad  -  \int_{\dincl}\bn_\incl(\ps)\cdot\nabla {S_2}(-2p_2\se_2+\msigma(\pt)-\ps)\, \gi(\ps)\, d\sigma_{\ps}\\
&\quad  + U_0  (\eps \p+\eps \pt)-w_{S_2}[\dincl,\gi](\pt) - \frac{\gi(\pt)}2,\qquad \qquad\qquad\qquad\forall \pt\in\dincl\,.
\end{split}
\]
Thus, it suffices to find a solution of \eqref{Mmu1mu2xi=0} to solve problem \eqref{bvpe}.
Therefore, we now analyze the behavior of the solutions of the system of integral equations \eqref{Mmu1mu2xi=0}.

\subsection{Real analyticity of ${\fN}$}

In the following Proposition \ref{M} we state the real analyticity of ${\fN}$. We omit the proof, which is a straightforward modification of the proof of Propositions  \ref{prop.N} and \ref{Ln=2}.

\begin{prop}[Real analyticity of ${\fN}$]\label{M}
The map
$$
 \begin{array}{rcl}
]-\e0,\e0[ \times \sB\times \R& \to& \sV^{1,\alpha}(\dpO)\times \Ca{1}(\dincl)\\
(\eps,{\bm},\xi)& \mapsto & {\fN}[\eps,{\bm},\xi]
 \end{array}
 $$
 is real analytic.
\end{prop}

 In the sequel we set
 \[
\tilde\incl\equiv \incl\cup(\msigma(\incl)-2p_2\se_2)\,.
\]
Then $\tilde\incl$ is an open subset of $\R^2$ of class $\Ca{1}$ with two connected components, $\incl$ and $\msigma(\incl)-2p_2\se_2$, and with boundary $\partial\tilde\incl$ consisting of two connected components,   $\dincl$ and $\partial\msigma(\incl)-2p_2\se_2$. One can also observe that $\tilde\incl$ is symmetric with respect to the horizontal axis $\R\times\{-p_2\}$. Then, for all functions $\phi$ from $\partial\omega$ to $\mathbb{R}$, we denote by $\tilde\phi$ the extension of $\phi$ to $\partial\tilde\omega$  defined by
\[
\tilde\phi(\pt)\equiv
\left\{
\begin{array}{ll}
\phi(\pt)&\text{if }\pt\in\partial\omega\,,\\
-\phi(\msigma(\pt)-2p_2\se_2)&\text{if }\pt\in\partial(\msigma(\incl)-2p_2\se_2)\,.
\end{array}
\right.
\]
In particular, the symbol $\tilde 1$ will denote the function  from $\partial\tilde\incl$ to $\mathbb{R}$ defined by
\[
\tilde 1(\pt)\equiv
\left\{
\begin{array}{ll}
1&\text{if }\pt\in\partial\omega\,,\\
-1&\text{if }\pt\in\partial(\msigma(\incl)-2p_2\se_2)\,.
\end{array}
\right.
\]
If $k\in\mathbb{N}$, then  we denote by $\Ca{k}_{\mathrm{odd}}(\partial\tilde\incl)$ the subspace of $\Ca{k}(\partial\tilde\incl)$ consisting of the functions $\psi$ such that $\psi(\pt)=-\psi(\msigma(\pt)-2p_2\se_2)$ for all $\pt\in\partial\tilde\incl$. The extensions $\tilde\phi$ belongs to $\Ca{k}_{\mathrm{odd}}(\partial\tilde\incl)$ for all $\phi\in \Ca{k}(\dincl)$, in particular $\tilde 1\in \Ca{k}_{\mathrm{odd}}(\partial\tilde\incl)$.  One can also prove that $v_{S_2}[\partial\tilde\incl,\psi]_{|\partial\tilde\incl}$ and $w_{S_2}[\partial\tilde\incl,\theta]_{|\partial\tilde\incl}$ belong to $\Ca{1}_{\mathrm{odd}}(\partial\tilde\incl)$ for all $\psi\in \Ca{0}_{\mathrm{odd}}(\partial\tilde\incl)$ and $\theta\in \Ca{1}_{\mathrm{odd}}(\partial\tilde\incl)$.

Then, by classical potential theory we have the following Lemma \ref{vSodd}.

\begin{lem}\label{vSodd}
The map from $\Ca{0}_{\#}(\dincl)\times\R$ to $\Ca{1}_{\mathrm{odd}}(\partial\tilde\incl)$ which takes $(\mu,\xi)$ to
\[
v_{S_2}[\partial\tilde\incl,\tilde\mu]_{|\partial\tilde\incl}  +\xi\, v_{S_2}[\partial\tilde\incl,\tilde 1]_{|\partial\tilde\incl}
\]
is an isomorphism.
\end{lem}
\begin{proof} 
By Lemma \ref{viso} the map  which takes $(\mu,\xi)$ to $v_{S_2}[\partial\tilde\incl,\mu]_{|\partial\tilde\incl}  +\xi$ is an isomorphism from $\Ca{0}_{\#}(\partial\tilde\incl)\times\R$ to $\Ca{1}(\partial\tilde\incl)$. Then the map from $\Ca{0}_{\mathrm{odd}}(\partial\tilde\incl)$ to $\Ca{1}_{\mathrm{odd}}(\partial\tilde\incl)$ which takes $\mu$ to $v_{S_2}[\partial\tilde\incl,\tilde\mu]_{|\partial\tilde\incl}$ is an isomorphism. One concludes by observing that the map from $\Ca{0}_{\#}(\dincl)\times\R$ to $\Ca{0}_{\mathrm{odd}}(\partial\tilde\incl)$ which takes $(\mu,\xi)$ to $\tilde\mu+\xi\,\tilde 1$ is an isomorphism.
\end{proof}

\vspace{\baselineskip}

\subsection{Functional analytic representation theorems}

As an intermediate step in the study of \eqref{Mmu1mu2xi=0} around  $\eps=0$, we now analyze equation \eqref{Mmu1mu2xi=0}  at the singular value $\eps=0$.
\begin{prop}\label{01}
There exists a unique  $({\bm^*},\xi^*)\in \sB\times\mathbb{R}$ such that
\[
{\fN}[0,{\bm^*},\xi^*]=\b0
\]
and we have
\[
\mu_1^*=0
\]
and
\[
v_{S_2}[\partial\tilde\incl,\tilde\mu_2^*](\pt)  +\xi^*\, v_{S_2}[\partial\tilde\incl,\tilde 1](\pt)=- \go  (0) \tilde 1(\pt)+w_{S_2}[\partial\tilde\incl,\tilde g^{\mathrm i}](\pt) + \frac{\tilde g^{\mathrm i}(\pt)}2,\quad \forall \pt\in\partial\tilde\incl\,.
\]
\end{prop}
\begin{proof}  First of all, we observe that for all $({\bm},\xi)\in \sB\times\mathbb{R}$, we have
\[\begin{cases}
{\fN}_1[0,{\bm},\xi](\px)&= v_{{\G}}[\dO,\mu_1](\px),\qquad\forall \px\in\dpO\,,\\
{\fN}_2[0,{\bm},\xi](\pt)&= v_{S_2}[\dincl,\mu_2](\pt)\\[5pt]
&\quad  - \displaystyle\int_{\dincl} {S_2}(-2p_2\se_2+\msigma(\pt)- \ps)\, \mu_2(\ps)\, d\sigma_{\ps}\\[5pt]
&\quad \quad+\xi\displaystyle\int_{\dincl}\left({S_2}(\pt-\ps)-{S_2}(-2p_2\se_2+\msigma(\pt)-\ps)\right)d\sigma_{\ps}\\[5pt]
&\quad  -  \displaystyle\int_{\dincl}\bn_\incl(\ps)\cdot\nabla {S_2}(-2p_2\se_2+\msigma(\pt)-\ps)\, \gi(\ps)\, d\sigma_{\ps}\\[5pt]
&\quad  + \go  (0)-w_{S_2}[\dincl,\gi](\pt) - \cfrac{\gi(\pt)}2 ,\qquad\qquad\qquad\forall \pt\in\dincl\,.
\end{cases}
\]
 By Theorem \ref{V} (ii), the unique function in $\Ca{0}_+(\dO)$ such that $v_{{\G}}[\dO,\mu_1]=0$ on $\dpO$ is $\mu_1=0$. On the other hand, by a change of variable in integrals, one verifies that
\[
{\fN}_2[0,{\bm},\xi](\pt)=v_{S_2}[\partial\tilde\incl,\tilde\mu_2](\pt)  +\xi\, v_{S_2}[\partial\tilde\incl,\tilde 1](\pt)+ \go  (0) \tilde 1(\pt)-w_{S_2}[\partial\tilde\incl,\tilde g^{\mathrm i}](\pt) - \frac{\gi(\pt)}2,\quad \forall \pt\in\dincl\,.
\]
 Then,  by Lemma \ref{vSodd}, there exists a unique pair $(\mu_2,\xi) \in \Ca{0}_{\#}(\dincl)\times\mathbb{R}$ such that
\[
v_{S_2}[\partial\tilde\incl,\tilde\mu_2](\pt)  +\xi\, v_{S_2}[\partial\tilde\incl,\tilde 1](\pt)=- \go  (0) \tilde 1(\pt)+w_{S_2}[\partial\tilde\incl,\tilde g^{\mathrm i}](\pt) + \frac{\tilde g^{\mathrm i}(\pt)}2,\quad\forall \pt\in\partial\tilde\incl\,.
\]
Now the statement is proved.
\end{proof}

\vspace{\baselineskip}

\noindent The main result of this section is obtained by exploiting the implicit function theorem for real analytic maps (see Deimling \cite[Thm.~15.3]{De85}).

\begin{thm}\label{analytic01}
 Let $({\bm^*},\xi^*)$ be as in Proposition \ref{01}. Then there exist $0<\ear*<\ea$, an open neighborhood $\cU^*$ of $({\bm^*},\xi^*)$ in $\sB\times\mathbb{R}$, and a real analytic map $\Psi\equiv(\Psi_1,\Psi_2,\Psi_3)$ from $\Je*$ to $\cU^*$ such that the set of zeros of ${\fN}$ in $\Je*\times\cU^*$ coincides with the graph of $\Psi$.
\end{thm}
\begin{proof}  The partial differential of $\fN$ with respect to $({\bm},\xi)$ evaluated at $(0,{\bm^*},\xi^*)$ is delivered by
\[
\begin{split}
\partial_{({\bm},\xi)}{\fN}_1[0,{\bm^*},\xi^*]({\boldsymbol\phi},\zeta)&=v_{{\G}}[\dO,\phi_1]_{|\dpO}\,,\\
\partial_{({\bm},\xi)}{\fN}_2[0,{\bm^*},\xi^*]({\boldsymbol\phi},\zeta)&=v_{S_2}[\partial\tilde\incl,\tilde\phi_2]_{|\dincl}  +\zeta\, v_{S_2}[\partial\tilde\incl,\tilde 1]_{|\dincl},
\end{split}
\]
for all $({\boldsymbol\phi},\zeta)\in \sB\times\mathbb{R}$. Then $\partial_{({\bm},\xi)}{\fN}[0,{\bm^*},\xi^*]$ is an isomorphism from $\sB\times\mathbb{R}$ to $\sV^{1,\alpha}(\dpO)\times \Ca{1}(\dincl)$ thanks to Proposition \ref{V} and Lemma \ref{vSodd}. The conclusion is reached  by the implicit function theorem (see Deimling \cite[Thm.~15.3]{De85}) and by Proposition \ref{M}.
\end{proof}

\vspace{\baselineskip}

\subsubsection{Macroscopic behavior}

We first provide a representation of the solution $u_\eps$.

\begin{rmk}[Representation formula in the macroscopic variable]\label{macro_e}
Let the assumptions of Theorem \ref{analytic01} hold. Then
\[
\begin{split}
u_{\eps}(\px)&=\uzero (\px)-\eps\int_{\dincl}\bn_\incl(\ps)\cdot(\nabla_\py {\G})(\px,\eps \p+ \eps \ps)\, \gi(\ps)\, d\sigma_{\ps} + v_{{\G}}\bigl[\dO,\Psi_1[\eps]\bigr](\px)\\
&\quad + \int_{\dincl}{\G}(\px,\eps \p+ \eps \ps)\,\Psi_2[\eps]\,d\sigma_{\ps} +\Psi_3[\eps]\int_{\dincl}{\G}(\px,\eps \p+  \eps \ps)\,d\sigma_{\ps},
\end{split}
\]
for all $\px\in\dome$ and for all $\eps \in \Ie*$.
\end{rmk}

As a consequence of Remark \ref{macro_e}, $u_{\eps}(\px)$ can be written in terms of a converging power series of $\eps$ for $\eps$ positive and small. A similar result holds for the restrictions $u_{\eps|\overline{\dom'}}$ where $\dom'$ is an open subset of $\dom$ such that $0\notin\overline{\dom'}$. Namely, we are now in the position to prove Theorem \ref{thm:Ue}.

\vspace{\baselineskip}

\begin{proofof}{Theorem \ref{thm:Ue}} Let $\e*$ be as in Theorem \ref{analytic01}. Let $\eps'\in]0,\e*]$ be such that \eqref{Ue.eq0} holds true. We define
\[
\begin{split}
\fU_{\dom'}[\eps](\px)
&\equiv U_0 (\px)-\eps\int_{\dincl}\bn_\incl(\ps)\cdot(\nabla_\py {\G})(\px,\eps \p+ \eps \ps)\, \gi(\ps)\, d\sigma_{\ps} + v_{{\G}}\bigl[\dO,\Psi_1[\eps]\bigr](\px)\\
&\quad + \int_{\dincl}{\G}(\px,\eps \p+ \eps \ps)\,\Psi_2[\eps]\,d\sigma_{\ps} +\Psi_3[\eps]\int_{\dincl}{\G}(\px,\eps \p+  \eps \ps)\,d\sigma_{\ps},
\end{split}
\]
for all $\px\in\overline{\dom'}$ and for all $\eps\in ]-\eps',\eps'[$. Then, by Theorem \ref{analytic} and  by  a standard argument (see the study of $\fL_2$ in the proof of Proposition  \ref{prop.N}) one verifies that $\fU_{\dom'}$ is real analytic from $]-\eps',\eps'[$ to $\Ca{1}(\overline{\dom'})$. The validity of \eqref{Ue.eq1} follows by Remark \ref{macro_e} and the validity of \eqref{Ue.eq2} can be deduced by Proposition \ref{01}, by Theorem \ref{analytic01}, and by a straightforward computation. 
\end{proofof}

\subsubsection{Microscopic behavior}

As we have done in  Remarks \ref{micro_e1e2} and \ref{micro_e1} for $\be$ small, we now present  a representation formula of $u_{\eps}(\eps\p+\eps\, \cdot)$.

\begin{rmk}[Representation formula in the microscopic variable]\label{micro_e}
Let the assumptions of Theorem \ref{analytic01} hold.  Then
\[
\begin{split}
u_{\eps}(\eps\p+\eps \pt)&=u_0  (\eps \p+\eps \pt) -w^e_{S_2}[\dincl,\gi](\pt) -  \int_{\dincl}\bn_\incl(\ps)\cdot\nabla {S_2}(-2p_2\se_2+\msigma(\pt)-\ps)\, \gi(\ps)\, d\sigma_{\ps}\\
&\quad  +\int_{\dpO} {\G}(\eps \p+\eps \pt, \py)\, \Psi_1[\eps](\py)\, d\sigma_\py\\
&\quad +v_{S_2}[\dincl,\Psi_2[\eps]](\pt) - \int_{\dincl} {S_2}(-2p_2\se_2+\msigma(\pt)-\ps)\, \Psi_2[\eps]\, d\sigma_{\ps}\\
&\quad +\Psi_3[\eps]\int_{\dincl}\left({S_2}(\pt-\ps)-{S_2}(-2p_2\se_2+\msigma(\pt)-\ps)\right)d\sigma_{\ps},
\end{split}
\]
for all $\pt\in\mathbb{R}^2\setminus\omega$ and  for all $\eps \in \Ie\ast$ such that  $\eps\p+\eps \pt\in \overline{\dome}$.
\end{rmk}

We now show that $u_{\eps}(\eps\p+\eps\, \cdot\,)$  can be expressed as a real analytic map of $\eps$ for $\eps$ small.

\begin{thm}\label{Ve}
Let the assumptions of Theorem \ref{analytic01} hold.  Let $\incl'$ be an open bounded subset of $\mathbb{R}^2\setminus\overline\incl$.  Let $\eps'\in]0,\ear{*}[$ be such that
\[
(\eps\p+\eps\overline{\incl'})\subseteq \Bn{r_1}\,,\qquad\forall\eps \in]-\eps',\eps'[\,.
\]
Then there exists a real analytic map $\fV_{\incl'}$ from $]-\eps',\eps'[$ to $\Ca{1}(\overline{\dom'})$ such that
\begin{equation}\label{Ve.eq1}
u_{\eps}(\eps\p+\eps\, \cdot\,)_{|\overline{\incl'}}=\fV_{\incl'}[\eps],\qquad\forall\eps\in]0,\eps'[\,.
\end{equation}
Moreover we have
\begin{equation}\label{Ve.eq2}
\fV_{\incl'}[0]=v_{\ast|\overline{\incl'}}+\go(0),
\end{equation}
where $v_\ast\in \Ca{1}_{\mathrm{loc}}(\mathbb{R}^2\setminus\tilde\incl)$ is the unique solution of
\[
\left\{
\begin{array}{ll}
\Delta v_\ast=0&\text{in }\mathbb{R}^2\setminus\tilde\incl\,,\\
v_\ast=\tilde g^{\mathrm i}-\tilde g^{\mathrm o}(0)&\text{on }\partial\tilde\incl\,,\\
\lim_{\pt\to\infty}v_\ast(\pt)=0\,.&
\end{array}
\right.
\]
\end{thm}
\begin{proof}   We define
\[
\begin{split}
\fV_{\incl'}[\eps](\pt)&=U_0  (\eps \p+\eps \pt) -w^e_{S_2}[\dincl,\gi](\pt) -  \int_{\dincl}\bn_\incl(\ps)\cdot\nabla {S_2}(-2p_2\se_2+\msigma(\pt)-\ps)\, \gi(\ps)\, d\sigma_{\ps}\\
&\quad  +\int_{\dpO} {\G}(\eps \p+\eps \pt, \py)\, \Psi_1[\eps](\py)\, d\sigma_\py\\
&\quad +v_{S_2}[\dincl,\Psi_2[\eps]](\pt) - \int_{\dincl} {S_2}(-2p_2\se_2+\msigma(\pt)-\ps)\, \Psi_2[\eps]\, d\sigma_{\ps}\\
&\quad +\Psi_3[\eps]\int_{\dincl}\left({S_2}(\pt-\ps)-{S_2}(-2p_2\se_2+\msigma(\pt)-\ps)\right)d\sigma_{\ps}
\end{split}
\]
for all $\pt\in\overline{\incl'}$ and for all $\eps \in]-\eps',\eps'[$. Then, one verifies that $\fV_{\incl'}$ is real analytic from $]-\eps',\eps'[$ to $\Ca{1}(\overline{\incl'})$ by Proposition \ref{Ub}, by Theorem \ref{analytic01}, and by a standard argument (see in the proof of Proposition \ref{prop.N} the argument used to study $\fL_{2}$).  Relation \eqref{Ve.eq1} follows by Remark \ref{micro_e}.

Now, by a  change of variables in the integrals and by Proposition \ref{01}, one verifies that
\begin{equation}\label{Ve.eq3}
\fV_{\incl'}[0](\pt)\equiv \go(0)-w^e_{S_2}[\partial\tilde\incl,\tilde g^{\mathrm i}](\pt)+v_{S_2}[\partial\tilde\incl,\tilde\Psi_2[0]](\pt)+\Psi_3[0]v_{S_2}[\partial\tilde\incl,\tilde 1](\pt),
\end{equation}
for all $\pt\in\overline{\incl'}$. The right hand side of \eqref{Ve.eq3} equals $\gi$ on $\partial\incl$ by Proposition \ref{01} and by the jump properties of the double layer potential. Then, the (harmonic) function
\begin{equation}\label{Ve.eq4}
v_*(\pt)\equiv-w^e_{S_2}[\partial\tilde\incl,\tilde g^{\mathrm i}](\pt)+v_{S_2}[\partial\tilde\incl,\tilde\Psi_2[0]](\pt)+\Psi_3[0]v_{S_2}[\partial\tilde\incl,\tilde 1](\pt),\qquad\forall \pt\in\mathbb{R}^2\setminus\tilde\omega
\end{equation}
equals $\tilde g^{\mathrm i}-\tilde g^{\mathrm o}(0)$  on $\partial\tilde\incl$. By the decaying properties at $\infty$ of the single and double layer potentials, $\lim_{\pt\to\infty}v_*(\pt)$ exists and is finite. Since  $v_*(\pt)=-v_*(\msigma(\pt)-2p_2\se_2)$,  $\lim_{\pt\to\infty}v_*(\pt)=0$. Now,  \eqref{Ve.eq2} holds  by the uniqueness of the solution of the exterior Dirichlet problem.\end{proof}

\vspace{\baselineskip}

\begin{rmk}\label{w*}
If we take $w_*(\pt)\equiv v_{\ast}(\pt-\p)+\go(0)$ for all $\pt\in\R^2_+\setminus(\p+\incl)$, then
\[
\fV_{\incl'}[0](\pt-\p)=w_{\ast}(\pt),\qquad\forall \pt\in\p+\overline\incl'
\]
and $w_\ast$ is the unique solution in $\Ca{1}_{\mathrm{loc}}(\overline{\mathbb{R}^2_+}\setminus(\p+\incl))$ of \eqref{w_*}.
\end{rmk}

\subsubsection{Energy integral}

In Theorem \ref{Enge} here below we turn to consider  the energy integral $\displaystyle\int_{\dome}\left|\nabla u_\eps\right|^2\,d\px$ for $\eps$ close to $0$.
\begin{thm}\label{Enge}
Let the assumptions of Theorem \ref{analytic01} hold.    Then there exist $0<\ear{\fE}<\ear\ast$ and a real analytic map
\[
\fE:\Je{\fE}\to\mathbb{R}
\]
such that
\begin{equation}\label{Enge.eq1}
\int_{\dom_{\eps}}\left|\nabla u_{\eps}\right|^2\,d\px=\fE(\eps),\qquad\forall\eps\in \Ie{\fE}\,.
\end{equation}
 Furthermore,
\begin{equation}\label{Enge.eq2}
\fE(0)=\int_{\Omega}|\nabla u_0|^2\,d\px+\frac{1}{2}\int_{\mathbb{R}^2\setminus\tilde\incl}\left|\nabla v_\ast\right|^2\,d\px\,.
\end{equation}
\end{thm}
\begin{proof}   We take  $\incl'$ as in Theorem \ref{Ve} which in addition satisfies the condition $\partial\incl\subseteq\overline{\incl'}$. Then we set $\ear{\fE}\equiv\eps''$ with $\eps''$ as in Theorem  \ref{Ve} and we define
\begin{align*}
\fE_1(\eps)\equiv&\int_{\partial\Omega} \go\; \bn_\dom\cdot \nabla (u_0+v_G[\partial_+\dom,\Psi_1[\eps]])\,d\sigma\,,\\
\fE_2(\eps)\equiv&-\int_{\partial\incl} \gi(\ps)\,\bn_{\incl}(\ps)\cdot\nabla w_G[\partial_+\Omega,\go](\eps\p+\eps \ps)\, d\sigma_{\ps}\,,\\
\fE_3(\eps)\equiv&\int_{\partial\incl} \left(\Psi_2[\eps](\ps)+\Psi_3[\eps]\right)\, w_G[\partial_+\Omega,\go](\eps\p+\eps \ps)\, d\sigma_{\ps}\,,\\
\fE_4(\eps)\equiv&-\int_{\partial\incl} \gi\; \bn_\incl\cdot \nabla \fV_{\incl'}[\eps]\,d\sigma
\end{align*}
and
\begin{equation}\label{Enge.eq3}
\fE(\eps)\equiv \fE_1(\eps)+\eps \fE_2(\eps)+\fE_3(\eps)+\fE_4(\eps),\qquad\forall \eps \in\Je{\fE}\,.
\end{equation}
By Theorems \ref{analytic01}  and  \ref{Ve}, by Lemma \ref{wnear0}, and by a standard argument (see in the proof of Proposition  \ref{prop.N} the study of $\fL_2$), one verifies that the functions $\fE_i$'s and $\fE$ are real analytic from $\Je{\fE}$ to $\mathbb{R}$. Using the definition of $w_G[\partial_+\Omega,\go]$ and by the Fubini's theorem, one gets
\[
\fE_2(\eps)=-\int_{\partial\dom}\go(\px)\,\bn_\dom(\px)\cdot\nabla_\px\left(\int_{\dincl} \bn_{\incl}(\ps)\cdot (\nabla_\py{\G})(\px,\eps \p+\eps \ps) \gi(\ps)\, d\sigma_{\ps}\right)d\sigma_\px
\]
and
\[
\begin{split}
\fE_3(\eps)&=\int_{\partial\dom}\go(\px)\,\bn_\dom(\px)\cdot\nabla_\px\left(\int_{\dincl}{\G}(\px,\eps \p +\eps \ps)\Psi_2[\eps](\ps)\, d\sigma_{\ps}\right)d\sigma_\px\\
&\quad +\int_{\partial\dom}\go(\px)\,\bn_\dom(\px)\cdot\nabla_\px\left(\Psi_3[\eps]\int_{\dincl}{\G}(\px,\eps \p +\eps \ps)\, d\sigma_{\ps}\right)d\sigma_\px
\end{split}
\]
for all $\eps\in]0,\ear{\fE}[$. Then, \eqref{Enge.eq1} follows by the divergence theorem, by Remark \ref{macro_e}, and by  Theorem \ref{Ve} (see also the  proofs of   Theorems \ref{Enge3g} and \ref{Enge2}, where an analog argument is presented in full details).

To prove \eqref{Enge.eq2},  we observe that $\Psi_1[0]=0$ by Proposition \ref{01} and Theorem \ref{analytic01}. Thus
\begin{equation}\label{Enge.eq4}
\fE_1(0)=\int_{\partial\dom}\go\; \bn_\dom\cdot \nabla u_0\, d\sigma=\int_{\Omega}|\nabla u_0|^2\,d\px\,.
\end{equation}
By Lemma \ref{jumps}, $w_G[\partial_+\Omega,\go](0)=\go(0)$. Since $\Psi_2[0] \in \Ca{1}_{\#}(\partial\omega)$, we compute
\begin{equation}\label{Enge.eq5}
 \fE_3(0)= \go(0)\, \Psi_3[0]\, \int_{\partial\incl}\, d\sigma\,.
\end{equation}
Then, we have
\[
\begin{split}
\int_{\partial\incl}  \bn_\incl\cdot \nabla v_\ast\,d\sigma&=-\int_{\partial\incl}  \bn_\incl\cdot \nabla w^e_{S_2}[\partial\tilde\incl,\tilde g^{\mathrm i}]\,d\sigma+\int_{\partial\incl}  \bn_\incl\cdot \nabla v_{S_2}^e[\partial\tilde\incl,\tilde\Psi_2[0]+\tilde\Psi_3[0]]\,d\sigma\\
&=-\int_{\partial\incl}  \bn_\incl\cdot \nabla w^i_{S_2}[\partial\tilde\incl,\tilde g^{\mathrm i}]\,d\sigma+\int_{\partial\incl}\left(\tilde\Psi_2[0]+\tilde\Psi_3[0]\right) d\sigma=\Psi_3[0]\, \int_{\partial\incl}\, d\sigma\,.
\end{split}
\]
where we have used successively \eqref{Ve.eq4},  the jump properties of the (classical) single and double layer potentials, the divergence theorem, and  $\Psi_2[0] \in \Ca{1}_{\#}(\partial\omega)$.
Using \eqref{Ve.eq2} and the equality $v_\ast(\pt)=-v_\ast(\msigma(\pt)-2p_2\se_2)$ which holds for all $\pt\in\R^2\setminus\tilde\incl$,  we have
\begin{equation}\label{Enge.eq6}
\begin{split}
\fE_4(0)&=-\int_{\partial\incl} (\gi-\go(0))\; \bn_\incl\cdot \nabla v_\ast\,d\sigma-\go(0)\int_{\partial\incl}  \bn_\incl\cdot \nabla v_\ast\,d\sigma\\
&=-\frac{1}{2}\int_{\partial\tilde\incl} v_\ast\; \bn_{\tilde\incl}\cdot \nabla v_\ast\,d\sigma-\go(0)\int_{\partial\incl}  \bn_\incl\cdot \nabla v_\ast\,d\sigma\\
&=\frac{1}{2}\int_{\mathbb{R}^2\setminus\tilde\incl}\left|\nabla v_\ast\right|^2\,d\px-\go(0)\, \Psi_3[0]\, \int_{\partial\incl}\, d\sigma\,.
\end{split}
\end{equation}
thanks to the divergence theorem. Relation \eqref{Enge.eq2} follows by \eqref{Enge.eq3} -- \eqref{Enge.eq6}.\end{proof}

\vspace{\baselineskip}

\begin{rmk}\label{Fw*}
If we take $w_*$ as in Remark \ref{w*}, then
\[
\fE(0)=\int_{\Omega}|\nabla u_0|^2\,d\px+\int_{{\mathbb{R}^2_+}\setminus(\p+\incl)}\left|\nabla w_\ast\right|^2\,d\px\,.
\]
\end{rmk}

Finally, we consider in the following Theorem \ref{Fluxe} the total flux on $\dO$. The proof of Theorem \ref{Fluxe} can be deduced by a straightforward modification of the proof of Theorem \ref{Fluxe2} and it is accordingly omitted.

\begin{thm}\label{Fluxe}
Let the assumptions of Theorem \ref{analytic01} hold.    Then there exist $\ear{\fF}\in]0,\ear\ast[$ and a real analytic function
\[
\fF:\Je{\fF}\to\mathbb{R}
\]
such that
\[
\int_{\dO}\bn_\Omega\cdot\nabla u_{\eps}\,d\sigma={\fF}(\eps),\qquad\forall\eps\in\Ie{\fF}\,.
\]
 Furthermore,
\[
{\fF}(0)=\int_{\partial\incl}  \bn_\incl\cdot \nabla v_\ast\,d\sigma=\int_{\p+\partial\incl}  \bn_{\p+\incl}\cdot \nabla w_\ast\,d\sigma\,.
\]
\end{thm}

\section{Conclusions}\label{conclusions}

In this paper, we have studied the asymptotic behavior of the solution to the Dirichlet problem
in a bounded domain in $\mathbb{R}^n$ with a small hole that approaches the boundary. We have shown
that this behavior depends on the space dimension $n$: if $n\geq 3$, the solution exhibits
real-analytic dependency on the perturbation parameters; if $n=2$, logarithmic behavior may occur.
Additionally, in the two-dimensional case we highlight two different regimes. In one, the hole approaches
the outer boundary while shrinking at a faster rate; in the other, the shrinking rate
and the rate of approach to the boundary are comparable. For these two different regimes, the energy
integral and the total flux on the outer boundary have different limiting values. Intuitively,
we may say that when the hole shrinks sufficiently fast in two-dimensional space,
the shrinking effect dominates the effect of its vicinity to the outer boundary.

The method used for our analysis is based on potential theory constructed with the Dirichlet Green's function
in the upper half space. Our results allow us to justify the representation of the solutions and related
functionals as convergent power series, which is usually difficult to achieve with standard asymptotic analysis.
We intend to compute such power series expansions in future publications.  To that end, we can exploit the integral representation of the solution and deduce the coefficients of the series by solving recursive systems of boundary integral equations  (as in Dalla Riva, Musolino, and Rogosin \cite{DaMuRo15}) or we can resort to an approximation method of asymptotic analysis, such  as the multiple scale  expansions method (cf. Bonnaillie-No\"el, Dambrine, Tordeux,  and Vial \cite{BoDaToVi09}), with the advantage that now we  just need to identify the terms of the asymptotic expansion, the convergence being a consequence of the results of the present paper. We also
plan to extend the analysis of perturbation problems in domains with a small hole
close to the boundary to other differential operators and boundary conditions. We remark
that the functional analytic approach developed in this paper within the framework of Schauder spaces
can be extended to a Sobolev space setting under Lipschitz regularity assumptions on the domains.
A first step in this direction has already been completed in Costabel, Dalla Riva, Dauge, and Musolino \cite{CDDM16}.\\

\appendix

\section{Decay properties of the Green's function and the associated single-layer potential} \label{app:pot}

\def\appendixname{}
In the following Lemma \ref{Gatinfty} we present a result concerning the Green's function ${\G}$ which allows us to  study the behavior of $v_{{\G}}[\dO,\phi]$ at infinity.
\begin{lem}\label{Gatinfty}
Let $n \in \mathbb{N}\setminus \{0,1\}$. Let $d\equiv 2\sup_{\py\in\dom}|\py|$. Then the function
\[\begin{array}{rcl}
(\Rn\setminus\Bn{d})\times\overline{\dom} & \to & \R\\
(\px,\py) & \mapsto & |\px|^{n-1}\, {\G}(\px,\py)
\end{array} \]
is bounded.
\end{lem}
\begin{proof} 
We observe that, for all $(\px,\py)\in (\Rn\setminus\Bn{d})\times\overline{\dom}$, we have
\[
|\px-\py|^2-|\msigma(\px)-\py|^2=\sum_{j=1}^{n-1}(x_j-y_j)^2+(x_n-y_n)^2-\sum_{j=1}^{n-1}(x_j-y_j)^2-(x_n+y_n)^2=-4x_ny_n.
\]
Let us first consider $n=2$. By exploiting the inequality $|\px|>2|\py|$, we calculate that for any $(\px,\py)\in (\Rn\setminus\Bn{d})\times\overline{\dom}$,
\[
\begin{split}
|{\G}(\px,\py)|&=\frac{1}{2\pi}\left|\log|\px-\py|-\log|\msigma(\px)-\py|\right|=\frac{1}{4\pi}\left|\log|\px-\py|^2-\log|\msigma(\px)-\py|^2\right|\\
&\le \frac{1}{4\pi}\frac{1}{\min\{|\px-\py|^2\,,\,|\msigma(\px)-\py|^2\}}\left||\px-\py|^2-|\msigma(\px)-\py|^2\right|\le \frac{1}{\pi}\frac{|x_2y_2|}{\min\{|\px-\py|^2\,,\,|\msigma(\px)-\py|^2\}}\\
& \le \frac{1}{\pi}\frac{|\px|\,|\py|}{(|\px|-|\py|)^2}=\frac{1}{\pi}\frac{|\py|}{(1-|\py|/|\px|)^2}\frac{1}{|\px|}\le \frac{4|\py|}{\pi}\frac{1}{|\px|}\le \frac{2d}{\pi}\frac{1}{|\px|}.
\end{split}
\]
To prove the statement for $n\ge 3$ we observe that
\[
\begin{split}
&|{\G}(\px,\py)|=\frac{1}{(n-2)s_n}\left||\px-\py|^{2-n}-|\msigma(\px)-\py|^{2-n}\right|\\
&\ =\frac{1}{(n-2)s_n}\frac{\left||\px-\py|^{2}-|\msigma(\px)-\py|^2\right|}{|\px-\py|\,|\msigma(\px)-\py|(|\px-\py|+|\msigma(\px)-\py|)}\sum_{j=0}^{n-3}|\px-\py|^{j+3-n}|\msigma(\px)-\py|^{-j}\le\frac{2^nd}{s_n}\frac{1}{|\px|^{n-1}}\end{split}
\]
for all $(\px,\py)\in (\Rn\setminus\Bn{d})\times\overline{\dom}$. Hence $|\px|^{n-1}\,|{\G}(\px,\py)|\le 2^nd/s_n$ for all $(\px,\py)\in (\Rn\setminus\Bn{d})\times\overline{\dom}$ and for all $n\in\mathbb{N}\setminus\{0,1\}$.
\end{proof}

\vspace{\baselineskip}

\medskip

Then, by Lemma \ref{Gatinfty} one readily deduces the validity of the following.

\begin{lem}\label{vGatinfty} Let $n \in \mathbb{N}\setminus \{0,1\}$.  Let $\phi\in \Ca{0}(\dO)$. Then the function which takes $\px\in\Rn\setminus(\dom\cup\msigma(\dom))$ to $|\px|^{n-1}\, v_{{\G}}[\dO,\phi](\px)$ is bounded. In particular, $v_{{\G}}[\dO,\phi]$ is harmonic at infinity.
\end{lem}

\def\appendixname{Appendix }
\section{An extension result}\label{app:CK}
\def\appendixname{}

In this Appendix we prove Proposition \ref{Ub}. We find convenient to set $\Bnp{r}\equiv\Bn{r}\cap\Rnp$ and $\Bnm{r}\equiv\Bn{r}\setminus\overline{\Bnp{r}}$ for all $r>0$. Then, possibly shrinking $r_0$ we can assume that $\Bnp{r_0}\subseteq\dom$. By assumption \eqref{go_analytic} and by a standard argument based on the Cauchy-Kovalevskaya Theorem we shows the validity of the following
\begin{lem}\label{A.F}
Let $n \in \mathbb{N}\setminus \{0,1\}$.  There exist $r_1\in]0,r_0]$ and a function $H$ from $\overline{\Bn{r_1}}$ to $\mathbb{R}$ such that $\Delta H=0$ in $\Bn{r_1}$ and $H_{|\overline{\Bn{r_1}}\cap\d0O}=\go_{|\overline{\Bn{r_1}}\cap\d0O}$.
\end{lem}
\begin{proof} 
By the Cauchy-Kovalevskaya Theorem there exists $r_1\in]0,r_0]$, a function $H^+$ from $\overline{\Bnp{r_1}}$ to $\mathbb{R}$, and a function $H^-$ from $\overline{\Bnm{r_1}}$ to $\mathbb{R}$, such that
\[
\begin{split}
&\text{$\Delta H^+=0$ in $\Bnp{r_1}$, $H^+_{|\overline{\Bn{r_1}}\cap\d0O}=\go_{|\overline{\Bn{r_1}}\cap\d0O}$, and $\partial_{x_n}H^+_{|\overline{\Bn{r_1}}\cap\d0O}=0$,}\\
&\text{$\Delta H^-=0$ in $\Bnm{r_1}$, $H^-_{|\overline{\Bn{r_1}}\cap\d0O}=\go_{|\overline{\Bn{r_1}}\cap\d0O}$, and $\partial_{x_n}H^-_{|\overline{\Bn{r_1}}\cap\d0O}=0$.}
\end{split}
\]
We now define
\[
H(\px)\equiv
\left\{
\begin{array}{ll}
H^+(\px)&\text{if }\px\in\overline{\Bnp{r_1}}\,,\\
H^-(\px)&\text{if }\px\in\overline{\Bnm{r_1}}\,,
\end{array}
\right.
\]
for all $\px\in\overline{\Bn{r_1}}$. Note that $H$ is well defined and $H(\px)=\go(\px)$ for $\px\in\Bn{r_1}\cap\d0O$. Then one observes that
\[
\begin{split}
\int_{\Bn{r_1}}H\,\Delta \varphi \,d\px&=\int_{\Bnp{r_1}}H^+\,\Delta \varphi \,d\px+\int_{\Bnm{r_1}}H^-\,\Delta \varphi \,d\px\\
&=-\int_{\Bn{r_1}\cap\d0O} H^+\,\partial_{x_n}\varphi\,d\sigma+\int_{\Bn{r_1}\cap\d0O} H^-\,\partial_{x_n}\varphi\,d\sigma\\
&\quad + \int_{\Bn{r_1}\cap\d0O} (\partial_{x_n}H^+)\,\varphi\,d\sigma - \int_{\Bn{r_1}\cap\d0O} (\partial_{x_n}H^-)\,\varphi\,d\sigma=0
\end{split}
\]
for all test functions $\varphi\in C^\infty_c(\Bn{r_1})$. Hence the lemma is proved.
\end{proof}

\vspace{\baselineskip}

\medskip

We are now ready to prove Proposition \ref{Ub}.

\begin{proofof}{Proposition \ref{Ub}}
Let $H$ be as in Lemma \ref{A.F}. Let $V^+\equiv u _{0|\overline{\Bnp{r_1}}}-H_{|\overline{\Bnp{r_1}}}$. Then we have $\Delta V^+=0$ in $\Bnp{r_1}$ and $V^+_{|\overline{\Bn{r_1}}\cap\d0O}=0$. Then we define $V^-(\px)\equiv-V^+(\msigma(\px))$ for all $\px\in \overline{\Bnm{r_1}}$. Then one verifies that $\Delta V^-=0$ in $\Bnm{r_1}$ and $V^-_{|\overline{\Bn{r_1}}\cap\d0O}=0$. In addition we have
$\partial_{x_n}V^+_{|\overline{\Bn{r_1}}\cap\d0O}=\partial_{x_n}V^-_{|\overline{\Bn{r_1}}\cap\d0O}$. Then we set
\[
V(\px)\equiv
\left\{
\begin{array}{ll}
V^+(\px)&\text{if }\px\in\overline{\Bnp{r_1}}\,,\\
V^-(\px)&\text{if }\px\in\overline{\Bnm{r_1}}\,,
\end{array}
\right.
\]
for all $\px\in\overline{\Bn{r_1}}$. Hence we compute
\[
\begin{split}
\int_{\Bn{r_1}}V\,\Delta \varphi \,d\px&=\int_{\Bnp{r_1}}V^+\,\Delta \varphi \,d\px+\int_{\Bnm{r_1}}V^-\,\Delta \varphi \,d\px\\
&=-\int_{\Bn{r_1}\cap\d0O} V^+\,\partial_{x_n}\varphi\,d\sigma+\int_{\Bn{r_1}\cap\d0O} V^-\,\partial_{x_n}\varphi\,d\sigma\\
&\quad + \int_{\Bn{r_1}\cap\d0O} (\partial_{x_n}V^+)\,\varphi\,d\sigma - \int_{\Bn{r_1}\cap\d0O} (\partial_{x_n}V^-)\,\varphi\,d\sigma=0
\end{split}
\]
for all test functions $\varphi\in C^\infty_c(\Bn{r_1})$. So that $\Delta V=0$ in $\Bn{r_1}$. Finally we take $ U_0 \equiv V+H$ and we readily verify that the statement of Proposition \ref{Ub} is verified (see also Lemma \ref{A.F}).
\end{proofof}


\section*{Acknowledgement}

The authors wish to thank M. Costabel and M. Dauge for several useful discussions and in particular for suggesting a Green's function method to investigate singular perturbation problems. The authors also thank C. Constanda for comments that have improved the presentation of the results. The authors are partially supported by the ANR (Agence Nationale de la Recherche), projects ARAMIS n$^{\rm o}$ ANR-12-BS01-0021.
M. Dalla Riva and P. Musolino acknowledge the support of `Progetto di Ateneo: Singular perturbation problems for differential operators -- CPDA120171/12' - University of Padova. M. Dalla Riva also acknowledges the support of HORIZON 2020 MSC EF project FAANon (grant agreement MSCA-IF-2014-EF - 654795) at the University of Aberystwyth, UK. P.~Musolino also acknowledges the support of `INdAM GNAMPA Project 2015 - Un approccio funzionale analitico per problemi di perturbazione singolare e di omogeneizzazione'  and of an `assegno di ricerca INdAM'.  P.~Musolino has received funding  from the European Union's Horizon
2020 research and innovation programme under the Marie
Sk\l odowska-Curie grant agreement No 663830,  and from the Welsh Government and Higher Education Funding Council for Wales through the S\^er Cymru National Research Network for Low Carbon, Energy and Environment. Part of the work has been carried out while P.~Musolino was visiting the `D\'{e}partement de math\'{e}matiques et applications' of the `\'{E}cole normale sup\'{e}rieure, Paris'. P.~Musolino wishes to thank the `D\'{e}partement de math\'ematiques et applications' and in particular V.~Bonnaillie-No\"el for the kind hospitality and the `INdAM' for the financial support.


\small
\begin{thebibliography}{34}
\expandafter\ifx\csname natexlab\endcsname\relax\def\natexlab#1{#1}\fi
\expandafter\ifx\csname url\endcsname\relax
  \def\url#1{\texttt{#1}}\fi
\expandafter\ifx\csname urlprefix\endcsname\relax\def\urlprefix{}\fi

\bibitem{AmKa07}
Ammari, H., Kang, H., 2007. Polarization and moment tensors. Vol. 162 of
  Applied Mathematical Sciences. Springer, New York, with applications to
  inverse problems and effective medium theory.

\bibitem{BenHassenBonnetier}
Ben~Hassen, M.~F., Bonnetier, E., 2005. Asymptotic formulas for the voltage
  potential in a composite medium containing close or touching disks of small
  diameter. Multiscale Model. Simul. 4~(1), 250--277.
\newline\urlprefix\url{http://dx.doi.org/10.1137/030602083}

\bibitem{BoDa13}
Bonnaillie-No\"el, V., Dambrine, M., 2013. Interactions between moderately
  close circular inclusions: the {D}irichlet-{L}aplace equation in the plane.
  Asymptot. Anal. 84~(3-4), 197--227.

\bibitem{BoDaLa}
Bonnaillie-No\"el, V., Dambrine, M., Lacave, C., 2016. Interactions between
  moderately close inclusions for the two-dimensional {D}irichlet-{L}aplacian.
  Appl. Math. Res. Express. AMRX~(1), 1--23.
\newline\urlprefix\url{http://dx.doi.org/10.1093/amrx/abv008}

\bibitem{BoDaToVi07}
Bonnaillie-No\"el, V., Dambrine, M., Tordeux, S., Vial, G., 2007. On moderately
  close inclusions for the {L}aplace equation. C. R. Math. Acad. Sci. Paris
  345~(11), 609--614.
\newline\urlprefix\url{http://dx.doi.org/10.1016/j.crma.2007.10.037}

\bibitem{BoDaToVi09}
Bonnaillie-No\"el, V., Dambrine, M., Tordeux, S., Vial, G., 2009. Interactions
  between moderately close inclusions for the {L}aplace equation. Math. Models
  Methods Appl. Sci. 19~(10), 1853--1882.
\newline\urlprefix\url{http://dx.doi.org/10.1142/S021820250900398X}

\bibitem{BoLaMa}
Bonnaillie-No\"el, V., Lacave, C., Masmoudi, N., 2015. Permeability through a
  perforated domain for the incompressible 2{D} {E}uler equations. Ann. Inst.
  H. Poincar\'e Anal. Non Lin\'eaire 32~(1), 159--182.
\newline\urlprefix\url{http://dx.doi.org/10.1016/j.anihpc.2013.11.002}

\bibitem{ChCl14}
Chesnel, L., Claeys, X., 2016. A numerical approach for the {P}oisson equation
  in a planar domain with a small inclusion. BIT 56~(4), 1237--1256.
\newline\urlprefix\url{http://dx.doi.org/10.1007/s10543-016-0615-z}

\bibitem{Ci95}
Cialdea, A., 1995. A general theory of hypersurface potentials. Ann. Mat. Pura
  Appl. (4) 168, 37--61.
\newline\urlprefix\url{http://dx.doi.org/10.1007/BF01759253}

\bibitem{CDDM16}
Costabel, M., Dalla~Riva, M., Dauge, M., Musolino, P., 2017. Converging
  expansions for {L}ipschitz self-similar perforations of a plane sector.
  Integral Equations Operator Theory 88~(3), 401--449.
\newline\urlprefix\url{http://dx.doi.org/10.1007/s00020-017-2377-7}

\bibitem{DaMu12}
Dalla~Riva, M., Musolino, P., 2012. Real analytic families of harmonic
  functions in a domain with a small hole. J. Differential Equations 252~(12),
  6337--6355.
\newline\urlprefix\url{http://dx.doi.org/10.1016/j.jde.2012.03.007}

\bibitem{DaMu15}
Dalla~Riva, M., Musolino, P., 2015. Real analytic families of harmonic
  functions in a planar domain with a small hole. J. Math. Anal. Appl. 422~(1),
  37--55.
\newline\urlprefix\url{http://dx.doi.org/10.1016/j.jmaa.2014.08.037}

\bibitem{DaMu}
Dalla~Riva, M., Musolino, P., 2016. A mixed problem for the {L}aplace operator
  in a domain with moderately close holes. Comm. Partial Differential Equations
  41~(5), 812--837.
\newline\urlprefix\url{http://dx.doi.org/10.1080/03605302.2015.1135166}

\bibitem{DaMub}
Dalla~Riva, M., Musolino, P., 2017. The {D}irichlet problem in a planar domain
  with two moderately close holes. J. Differential Equations 263~(5),
  2567--2605.
\newline\urlprefix\url{http://dx.doi.org/10.1016/j.jde.2017.04.006}

\bibitem{DaMuRo15}
Dalla~Riva, M., Musolino, P., Rogosin, S.~V., 2015. Series expansions for the
  solution of the {D}irichlet problem in a planar domain with a small hole.
  Asymptot. Anal. 92~(3-4), 339--361.

\bibitem{DaToVi10}
Dauge, M., Tordeux, S., Vial, G., 2010. Selfsimilar perturbation near a corner:
  matching versus multiscale expansions for a model problem. In: Around the
  research of {V}ladimir {M}az'ya. {II}. Vol.~12 of Int. Math. Ser. (N. Y.).
  Springer, New York, pp. 95--134.
\newline\urlprefix\url{http://dx.doi.org/10.1007/978-1-4419-1343-2_4}

\bibitem{De85}
Deimling, K., 1985. Nonlinear functional analysis. Springer-Verlag, Berlin.
\newline\urlprefix\url{http://dx.doi.org/10.1007/978-3-662-00547-7}

\bibitem{Fo95}
Folland, G.~B., 1995. Introduction to partial differential equations, 2nd
  Edition. Princeton University Press, Princeton, NJ.

\bibitem{GiTr83}
Gilbarg, D., Trudinger, N.~S., 1983. Elliptic partial differential equations of
  second order, 2nd Edition. Vol. 224 of Grundlehren der Mathematischen
  Wissenschaften [Fundamental Principles of Mathematical Sciences].
  Springer-Verlag, Berlin.
\newline\urlprefix\url{http://dx.doi.org/10.1007/978-3-642-61798-0}

\bibitem{Il78}
Il'in, A.~M., 1976. A boundary value problem for an elliptic equation of second
  order in a domain with a narrow slit. {I}. {T}he two-dimensional case. Mat.
  Sb. (N.S.) 99(141)~(4), 514--537.

\bibitem{Il92}
Il'in, A.~M., 1992. Matching of asymptotic expansions of solutions of boundary
  value problems. Vol. 102 of Translations of Mathematical Monographs. American
  Mathematical Society, Providence, RI, translated from the Russian by V.
  Minachin [V. V. Minakhin].

\bibitem{Il99}
Il'in, A.~M., 1999. The boundary layer [ {MR}1066959 (91i:35020)]. In: Partial
  differential equations, {V}. Vol.~34 of Encyclopaedia Math. Sci. Springer,
  Berlin, pp. 173--210, 241--247.
\newline\urlprefix\url{http://dx.doi.org/10.1007/978-3-642-58423-7_5}

\bibitem{KoMaMo99}
Kozlov, V., Maz'ya, V., Movchan, A., 1999. Asymptotic analysis of fields in
  multi-structures. Oxford Mathematical Monographs. The Clarendon Press, Oxford
  University Press, New York, oxford Science Publications.

\bibitem{La07}
Lanza~de Cristoforis, M., 2007. Asymptotic behavior of the solutions of a
  nonlinear {R}obin problem for the {L}aplace operator in a domain with a small
  hole: a functional analytic approach. Complex Var. Elliptic Equ. 52~(10-11),
  945--977.
\newline\urlprefix\url{http://dx.doi.org/10.1080/17476930701485630}

\bibitem{La10}
Lanza~de Cristoforis, M., 2010. Asymptotic behaviour of the solutions of a
  non-linear transmission problem for the {L}aplace operator in a domain with a
  small hole. {A} functional analytic approach. Complex Var. Elliptic Equ.
  55~(1-3), 269--303.
\newline\urlprefix\url{http://dx.doi.org/10.1080/17476930902999058}

\bibitem{LaMu13}
Lanza~de Cristoforis, M., Musolino, P., 2013. A real analyticity result for a
  nonlinear integral operator. J. Integral Equations Appl. 25~(1), 21--46.
\newline\urlprefix\url{http://dx.doi.org/10.1216/JIE-2013-25-1-21}

\bibitem{LaRo04}
Lanza~de Cristoforis, M., Rossi, L., 2004. Real analytic dependence of simple
  and double layer potentials upon perturbation of the support and of the
  density. J. Integral Equations Appl. 16~(2), 137--174.
\newline\urlprefix\url{http://dx.doi.org/10.1216/jiea/1181075272}

\bibitem{Ma07}
Martin, D., 2007. {\sc M\'elina}, biblioth\`eque de calculs \'el\'ements finis.
  \vskip0pt{\tt\
  http\char58//anum-maths.univ-rennes1.fr/melina/danielmartin/melina/}.

\bibitem{MaMoNi13}
Maz'ya, V., Movchan, A., Nieves, M., 2013. Green's kernels and meso-scale
  approximations in perforated domains. Vol. 2077 of Lecture Notes in
  Mathematics. Springer, Heidelberg.
\newline\urlprefix\url{http://dx.doi.org/10.1007/978-3-319-00357-3}

\bibitem{MaNaPl00}
Maz'ya, V., Nazarov, S., Plamenevskij, B., 2000. Asymptotic theory of elliptic
  boundary value problems in singularly perturbed domains. {V}ol. {I}. Vol. 111
  of Operator Theory: Advances and Applications. Birkh\"auser Verlag, Basel,
  translated from the German by Georg Heinig and Christian Posthoff.

\bibitem{MaMoNi16}
Maz'ya, V.~G., Movchan, A.~B., Nieves, M.~J., 2016. Mesoscale models and
  approximate solutions for solids containing clouds of voids. Multiscale
  Model. Simul. 14~(1), 138--172.
\newline\urlprefix\url{http://dx.doi.org/10.1137/151006068}

\bibitem{Mi65}
Miranda, C., 1965. Sulle propriet\`a di regolarit\`a di certe trasformazioni
  integrali. Atti Accad. Naz. Lincei Mem. Cl. Sci. Fis. Mat. Natur. Sez. I (8)
  7, 303--336.

\bibitem{NoSo13}
Novotny, A.~A., Soko{\l}owski, J., 2013. Topological derivatives in shape
  optimization. Interaction of Mechanics and Mathematics. Springer, Heidelberg.
\newline\urlprefix\url{http://dx.doi.org/10.1007/978-3-642-35245-4}

\bibitem{Tr87}
Troianiello, G.~M., 1987. Elliptic differential equations and obstacle
  problems. The University Series in Mathematics. Plenum Press, New York.
\newline\urlprefix\url{http://dx.doi.org/10.1007/978-1-4899-3614-1}

\end{thebibliography}

\end{document}